\documentclass[preprint,12pt]{elsarticle}
\usepackage{amssymb}
\usepackage{amsthm}
\usepackage{amsmath}
\usepackage{mathtools, mathrsfs}
\usepackage{changes}
\usepackage{makecell}
\usepackage{enumerate}
\usepackage{algorithm, algorithmic}
\usepackage{booktabs}
\numberwithin{equation}{section}

\newtheorem{theorem}{Theorem}
\newtheorem{proposition}{Proposition}
\newtheorem{lemma}[theorem]{Lemma}
\newdefinition{remark}{Remark}
\newdefinition{exam}{Example}
\numberwithin{theorem}{section}
\journal{Journal of Computational Physics}

\begin{document}

\begin{frontmatter}
\title{Adaptive Hyperbolic-cross-space Mapped Jacobi Method on Unbounded Domains with Applications to Solving Multidimensional Spatiotemporal Integrodifferential Equations}

\author[label1]{Yunhong Deng}
\affiliation[label1]{
    organization={School of Mathematics},
    addressline={University of Minnesota Twin Cities}, 
    city={Minneapolis},
    state={MN},
    postcode={55455},
    country={USA}}

\author[label2]{Sihong Shao}
\affiliation[label2]{
    organization={CAPT, LMAM and School of Mathematical Sciences},
    addressline={Peking University}, 
    city={Beijing},
    postcode={100871}, 
    country={CHINA}}

\author[label3]{Alex Mogilner}
\affiliation[label3]{
    organization={Courant Institute of Mathematical Science},
    addressline={New York University}, 
    city={New York},
    postcode={10012}, 
    state={NY},
    country={USA}}

\author[label3]{Mingtao Xia\corref{label4}}\ead{xiamingtao@nyu.edu}\cortext[label4]{Corresponding author.}

\begin{abstract}
    In this paper, we develop a new adaptive hyperbolic-cross-space mapped Jacobi (AHMJ) method for solving multidimensional spatiotemporal integrodifferential equations in unbounded domains. By devising adaptive techniques for sparse mapped Jacobi spectral expansions defined in a hyperbolic cross space, our proposed AHMJ method can efficiently solve various spatiotemporal integrodifferential equations such as the anomalous diffusion model with reduced numbers of basis functions. Our analysis of the AHMJ method gives a uniform upper error bound for solving a class of spatiotemporal integrodifferential equations, leading to effective error control.
\end{abstract}

\end{frontmatter}

\section{Introduction}
There have been wide applications of unbounded-domain spatiotemporal integrodifferential to describe the evolution of physical or biological quantities in different physical or biophysical models. As an example, 2D or 3D unbounded-domain anomalous diffusion equations incorporating fractional Laplacian operators are used in material science \cite{Baeumer2008, Andrade2017, dePablo2011}. As another example, unbounded-domain multidimensional aggregation-diffusion equations,  which include a convolutional term to model the nonlocal interaction among individuals, are used to model swarming behavior and chemotaxis \cite{Mogilner1996, Mogilner1999, Carrillo2019, Carrillo2023}.

Spatiotemporal integrodifferential equations in unbounded domains present computational challenges for many existing methods. For instance, mesh-based approaches like finite difference and finite element methods struggle with unbounded domains when applied directly \cite{Griebel2007, DElia2020}. Truncating the unbounded domain into a bounded domain is necessary for using those mesh-based methods. Yet, domain truncation requires devising artificial boundary conditions \cite{Han2000, Sun2020, Zhang2017}, which could be intricate to formulate.

Spectral methods could be an effective approach for solving unbounded-domain problems \cite{Shen2011, Shen2012, Shen2013, Sheng2019, Tang2018, Tang2019, Xia2020, Xia2020(1)}. Some basis functions, such as the Hermite functions and Laguerre functions, are inherently defined in unbounded domains. Thus, with basis functions defined in unbounded domains, spectral methods can be directly applied to solve unbounded-domain spatiotemporal equations. Recently, novel adaptive techniques for spectral methods \cite{Xia2020, Xia2020(1), Xia2022, Xia2023} have been proposed to improve the efficiency of using spectral methods to solve unbounded-domain spatiotemporal equations. By monitoring a \textbf{frequency indicator} and an \textbf{exterior-error indicator}, the adaptive spectral method can automatically adjust the decaying rate, the displacement, and the expansion order of the spectral expansion to accurately capture the dynamic behaviors of solutions in the unbounded domain.

When using spectral methods to solve multidimensional spatiotemporal equations, the ``curse of dimensionality" arises as the number of basis functions needed could grow exponentially with dimensionality \cite{Griebel2007, Shen2010, Shen2010(1), Shen2012, Luo2013}. A sparse spectral method based on a hyperbolic cross space \cite{Griebel2007, Shen2010, Luo2013} has been proposed to effectively reduce the number of basis functions needed when approximating a multidimensional function. The hyperbolic-cross-space spectral methods have yielded good results in solving high-dimensional elliptic equations \cite{Shen2010, Shen2010(1), Shen2012} and high-dimensional parabolic equations \cite{Griebel2007, Luo2013}. Yet, when solving spatiotemporal equations, the behavior of solutions to those unbounded-domain spatiotemporal equations may evolve over time and require adaptive adjustment of the basis functions \cite{Xia2020, Xia2020(1)}. Previous adaptive Hermite methods \cite{Xia2020, Xia2020(1), Xia2023} have primarily focused on using a dense spectral expansion without any dimension reduction techniques. However, compared to sparse spectral methods with spectral expansion in a proper hyperbolic cross space, full-tensor-product Hermite spectral expansions use much more basis functions without substantially improving accuracy and could thus be computationally ineffective.

Furthermore, solutions to certain unbounded-domain multidimensional spatiotemporal equations such as the anomalous diffusion equation \cite{Baeumer2008, Andrade2017, dePablo2011} and the Patlak-Keller-Segel equation \cite{Velzquez2002} decay algebraically at infinity. Solving those equations requires using basis functions that can characterize algebraic decay at infinity such as the mapped Jacobi functions rather than the exponentially decaying Hermite functions \cite{Shen2013, Tang2019, Sheng2019}. Yet, to our knowledge, there has been little research on how to develop adaptive sparse spectral methods to effectively solve unbounded-domain spatiotemporal equations with algebraically decaying basis functions.

In this work, we develop an adaptive hyperbolic-cross-space mapped Jacobi (AHMJ) method to efficiently solve multidimensional spatiotemporal integrodifferential equations in unbounded domains. Our main contributions are summarized as follows: i) we devise adaptive hyperbolic-cross-space techniques for properly adjusting basis functions over time in a spectral expansion defined in \textbf{hyperbolic cross spaces} \cite{Griebel2007, Shen2010, Luo2013} for solving unbounded-domain multidimensional spatiotemporal equations, ii) we extend the adaptive spectral techniques to \textbf{mapped Jacobi} spectral expansions \cite{Shen2012, Shen2013, Sheng2019, Tang2019}, and iii) we carry out an analysis on the error bound of applying the proposed AHMJ method for solving integrodifferential spatiotemporal equations.

We study the following general nonlinear spatiotemporal equation of the weak form
\begin{equation}\label{eq:2-2}
    \begin{aligned}
        &\big(\partial_{t}u,v\big) + a\big(u, v; t\big) = \big(f(u; t), v\big),~(\boldsymbol{x}, t) \in \mathbb{R}^{d}\times [0, T],\\
        &\quad\quad\big(u(\boldsymbol{x}, 0), \Tilde{v}\big) = \big(u_{0}(\boldsymbol{x}), \Tilde{v}\big),\,\,
        \forall (v, \Tilde{v}) \in L^{2}\big([0, T]; H^{1}(\mathbb{R}^{d})\big) \times H^{1}(\mathbb{R}^{d}).
    \end{aligned}
\end{equation}
Here, $H^{1}(\mathbb{R}^{d})$ is the Sobolev space build on $L^{2}(\mathbb{R}^{d})$ \citep[Chapter 6]{Arendt2023}, and $\big(\cdot, \cdot\big)$ denotes the $L^{2}$ inner product \textit{w.r.t.} the spatial variable $\boldsymbol{x}$ (we also use $\big(\cdot, \cdot\big)$ to denote the duality \cite{Dautray1992} between $H^{1}(\mathbb{R}^{d})$ and $H^{-1}(\mathbb{R}^{d})$):
\begin{equation}
    \big(u, v\big) \coloneqq \int_{\mathbb{R}^{d}}u(\boldsymbol{x}) v(\boldsymbol{x})~ \text{d}\boldsymbol{x}.
\end{equation}
$a(u, v; t)$ is a bilinear form, $f(u; t)$ is a nonlinear operator, and $u_{0}(\boldsymbol{x})$ denotes the initial value. We shall prove the following theorem on the error bound for applying the AHMJ method.

\begin{theorem}\label{th:1-1} 
We assume that $a\big(u, v; t\big)$ is a symmetric bilinear form satisfying the following continuous and coercive conditions: there exist two constants $C_{0}, c_{0} > 0$ such that
\begin{equation}\label{eq:1-4}
\begin{aligned}
    a\big(u, v; t\big) \le C_{0}\|u\|_{H^{1}} \|v\|_{H^{1}}, \,\, c_{0}\|u\|_{H^{1}}^{2} \le a\big(u, u; t\big).
\end{aligned}
\end{equation}
Furthermore, we assume that the nonlinear term $f(u; t)$ satisfies the Lipschitz condition: there exists a constant $L > 0$ such that
\begin{equation}\label{eq:1-5}
    \forall u, v, \phi \in L^{2}(\mathbb{R}^{d}) \Longrightarrow \big(f(u; t) - f(v; t), \phi\big) \le L\|u - v\|_{L^{2}} \|\phi\|_{L^{2}}.
\end{equation}
Then, the $L^{2}$ error of using the AHMJ method to solve the model problem Eq.~\eqref{eq:2-2} can be bounded by the sum of three separate error bounds:
\begin{equation}\label{error_estimate}
    \big\|u(\cdot, T) - U_{N, \gamma}^{\boldsymbol{\beta}, \boldsymbol{x}_{0}}(\cdot, T)\big\|_{L^{2}} \le E_{\mathcal{J}}(T) + E_{RK}(T) + E_{A}(T),
\end{equation}
where $U_{N, \gamma}^{\boldsymbol{\beta}, \boldsymbol{x}_{0}}(\boldsymbol{x}, T)$ denotes the numerical solution of the AHMJ method and an implicit Runge-Kutta scheme \cite{Southworth2022, Xia2022, Ern2021}. $E_{\mathcal{J}}$, $E_{RK}$, and $E_{A}$ denote the mapped Jacobi approximation error bound, the IRK time discretization scheme error bound, and the adaptive technique error bound, respectively.
\end{theorem}

In Section \ref{sec:3}, we shall prove that the mapped Jacobi method error bound $E_{\mathcal{J}}$ in Eq.~\eqref{error_estimate} is determined by the hyperbolic-cross-space mapped Jacobi approximation error $u - \pi_{N, \gamma}^{\boldsymbol{\beta}, \boldsymbol{x}_{0}}u$; $E_{RK}$ and $E_{A}$ are determined by the implementation of the time discretization scheme and the implementation of adaptive techniques, respectively. Given a smooth function, $E_{RK}$ can be maintained small if some appropriate high-order time discretization schemes are implemented. Therefore, we can control $E_{\mathcal{J}}$ and $E_{A}$ by choosing an appropriate hyperbolic cross space and properly implementing the adaptive techniques for the sparse spectral expansion approximation. Theorem~\ref{th:1-1} indicates the error in implementing our AHMJ method to solve multidimensional spatiotemporal integrodifferential equations can be well controlled.

The rest of this paper is organized as follows. Section \ref{sec:2} analyzes the model problem Eq. \eqref{eq:2-2} and gives the numerical scheme for applying the AHMJ method to numerically solve it. Section \ref{sec:3} carries out a numerical analysis of the model problem Eq.~\eqref{eq:2-2} and proves Theorem~\ref{th:1-1}. Section \ref{sec:4} presents the AHMJ method and numerical results. Section \ref{sec:5} concludes our paper. A summary of the main variables and notations is given in Table \ref{table:note}.

\begin{table}[h!]
    \centering
    \caption{\textnormal{Definitions of the main variables and notations used in this study.}}\label{table:note}
    \begin{tabular}{c|l}
    \toprule
    Symbol & Definition \\
    \midrule
    $\mathcal{J}_{n}^{\beta, x_{0}}(x)$ & $\mathcal{J}_{n}^{\beta, x_{0}}(x) = \mathcal{J}_{n}\big(\beta(x - x_{0})\big)$ is the $n^{\text{th}}$ order\\
    ~&~mapped Jacobi function \cite{Shen2013, Tang2019}\\
    \midrule
    $\boldsymbol{x}$ & $\boldsymbol{x} \coloneqq (x_{1}, \cdots, x_{d}) \in \mathbb{R}^{d}$ is the $d$-dimensional spatial variable\\
    \midrule
    $\boldsymbol{n}$ &  $\boldsymbol{n}\coloneqq(n_1,\cdots,n_d)$ is the $d$-dimensional index\\
    \midrule
    $\boldsymbol{\beta}$ & $\boldsymbol{\beta}\coloneqq(\beta_1,\cdots,\beta_d)$ is the $d$-dimensional scaling factor\\
    ~&~$\beta_i$ is the scaling factor in the $i^{\text{th}}$ dimension\\
    \midrule
    $\boldsymbol{x}_{0}$ & $\boldsymbol{x}_0\coloneqq({x_0}_1,\cdots,{x_0}_d)$ is the $d$-dimensional displacement of the\\ 
    ~&~basis functions. ${x_0}_i$ is the displacement in the $i^{\text{th}}$ dimension\\
    \midrule
    $\mathcal{J}_{\boldsymbol{n}}^{\boldsymbol{\beta}, \boldsymbol{x}_{0}}(\boldsymbol{x})$ & $\mathcal{J}_{\boldsymbol{n}}^{\boldsymbol{\beta}, \boldsymbol{x}_{0}}(\boldsymbol{x})\coloneqq \prod_{i = 1}^{d}\mathcal{J}_{n_{i}}^{\beta_{i}, {x_{0}}_{i}}(x_{i})$ is the tensorial\\
    ~&~mapped Jacobi function; here $\boldsymbol{n}\coloneqq(n_1,...,n_d)$\\
    \midrule
    $\Upsilon_{N, \gamma}$ & $\Upsilon_{N, \gamma} \coloneqq \big\{\boldsymbol{n} \in \mathbb{N}^{d} : |\boldsymbol{n}|_{\text{mix}} |\boldsymbol{n}|_{\infty}^{-\gamma} \le N^{1 - \gamma}\big\}$ is the hyperbolic\\
    ~&~cross index set \cite{Shen2010, Luo2013}, where $|\boldsymbol{n}|_{\text{mix}} \coloneqq \prod_{i = 1}^{d}\max\{1, n_{i}\}$\\
    \midrule
    $V_{N, \gamma}^{\boldsymbol{\beta}, \boldsymbol{x}_{0}}$ & $V_{N,\gamma}^{\boldsymbol{\beta}, \boldsymbol{x}_{0}} \coloneqq \operatorname{span}_{\boldsymbol{n} \in \Upsilon_{N,\gamma}}\left\{\mathcal{J}_{\boldsymbol{n}}^{\boldsymbol{\beta}, \boldsymbol{x}_{0}}(\boldsymbol{x})\right\}$ is the hyperbolic cross\\
    ~&~ mapped Jacobi approximation space\\
    \midrule
    $\pi_{N, \gamma}^{\boldsymbol{\beta}, \boldsymbol{x}_{0}}$ & the projection operator $\pi_{N, \gamma}^{\boldsymbol{\beta}, \boldsymbol{x}_{0}} : L^{2}(\mathbb{R}^{d}) \to V_{N, \gamma}^{\boldsymbol{\beta}, \boldsymbol{x}_{0}}$\\
    ~&~ such that $\big(\pi_{N, \gamma}^{\boldsymbol{\beta}, \boldsymbol{x}_{0}}u - u, \pi_{N, \gamma}^{\boldsymbol{\beta}, \boldsymbol{x}_{0}}u\big) = 0,~\forall u \in L^{2}(\mathbb{R}^{d})$\\
    \midrule
    $L^{2}([a, b]; V)$ & the Bochner space $\{u : [a, b] \to V;~ \int_{a}^{b}\|u(t)\|_{V}^{2}dt < \infty\}$\\
    \midrule
    $X(t_{0}, t_{1})$ & the Sobolev-Bochner space \citep[page 472]{Dautray1992}\\
    ~&~$\left\{u \in L^{2}\big([t_{0}, t_{1}]; H^{1}(\mathbb{R}^{d})\big) : \partial_{t}u \in L^{2}\big([t_{0}, t_{1}]; H^{-1}(\mathbb{R}^{d})\big)\right\}$\\
    \midrule
    $Y(t_{0}, t_{1})$ & the Sobolev-Bochner space $L^{2}\big([t_{0}, t_{1}]; H^{1}(\mathbb{R}^{d})\big) \times H^{1}(\mathbb{R}^{d})$\\
    \bottomrule
\end{tabular}
\end{table}

\section{Model problem analysis and the numerical scheme}\label{sec:2}
In this section, we prove the existence and uniqueness of a solution to the model problem Eq.~\eqref{eq:2-2}. Then, we introduce the sparse hyperbolic-cross-space mapped Jacobi approximation and present the AHMJ method to solve the model problem Eq. \eqref{eq:2-2}.

\subsection{Analysis on the model problem Eq. \eqref{eq:2-2}}
The following theorem establishes the existence and the uniqueness of a weak solution $u(\boldsymbol{x}, t) \in X(0, T)$ to the model problem Eq.~\eqref{eq:2-2}, where $X(t_{0}, t_{1})$ is the Bochner–Sobolev space defined in Table~\ref{table:note}. The norm of $u(\boldsymbol{x}, t)$ in the Bochner-Sobolev $X(t_{0}, t_{1})$ space is defined by
\begin{equation}\label{eq:3-18}
    \|u\|_{X(t_{0}, t_{1})}^{2} \coloneqq \int_{t_{0}}^{t_{1}}\big(\|\partial_{t}u\|_{H^{-1}}^{2} + \|u\|_{H^{1}}^{2}\big) \text{d}t + \|u(\cdot, t_{0})\|^{2}_{L^{2}}.
\end{equation}
\begin{theorem}\label{th:2-3}
Assume that the continuous and coercive conditions in Eqs.~\eqref{eq:1-4} and \eqref{eq:1-5} are satisfied. If we additionally assume that $L < c_{0}$, then there exists a unique solution $u(\boldsymbol{x}, t) \in X(0, T)$ to the model problem Eq. \eqref{eq:2-2}.
\end{theorem}

The proof of Theorem \ref{th:2-3} is given in \ref{ap:2}. Actually, a wide range of spatiotemporal integrodifferential equations can be cast into the model problem Eq.~\eqref{eq:2-2}. As an example, consider an $a\big(u, v; t\big)$ containing convolutional operator:
\begin{equation}\label{eq:2-22}
    a\big(u, v; t\big) \coloneqq \big(G \ast \nabla u, \nabla v\big) + \varepsilon\big((u, v) + (\nabla u, \nabla v)\big),
\end{equation}
where $(G\ast u)(\boldsymbol{x}) \coloneqq \int_{\mathbb{R}^{d}} G(\boldsymbol{x} - \boldsymbol{y}) u(\boldsymbol{y})\text{d}\boldsymbol{y}$ is a spatial convolutional operator. The following proposition shows under which assumptions on the convolutional kernel $G$ does $a(u, v;t)$ in Eq.~\eqref{eq:2-22} satisfy the conditions in Theorem~\ref{th:1-1}.

\begin{proposition}\label{prop3}
Assume that the convolution kernel $G(\boldsymbol{x})$ in Eq.~\eqref{eq:2-22} satisfies: 
\begin{enumerate}
    \item $G(\boldsymbol{x}) \in L^{1}(\mathbb{R}^{d})$
    \item the Fourier transform $\mathscr{F}(G)(\boldsymbol{x}) \geq 0,\,\, \forall \boldsymbol{x}\in\mathbb{R}^d$
\end{enumerate}
Then the bilinear form $a(u, v; t)$ defined in Eq.~\eqref{eq:2-22} satisfies the continuous condition and coercive conditions in Theorem \ref{th:1-1}.
\end{proposition}

The proof of Proposition~\ref{prop3} is given in \ref{ap:6}. The assumptions on the convolutional kernel $G$ in Lemma~\ref{prop3} can be met by many commonly used convolutional kernels which are radial symmetric functions, such as the Gaussian potential kernel \cite{Burger2014} and the Morse potential kernel \cite{Carrillo2014}.

\subsection{Hyperbolic-cross-space mapped Jacobi approximation}
Now, we introduce the mapped Jacobi functions \cite{Shen2012, Shen2013, Tang2019, Sheng2019} defined in $\mathbb{R}$. We denote $\{j_{n}^{\alpha_{1}, \alpha_{2}}(\xi)\}_{n = 0}^{\infty}$ to be the set of Jacobi polynomials defined on the interval $(-1, 1)$ with two fixed parameters $\alpha_{1}, \alpha_{2} > -1$ \citep[Chapter 3.2]{Shen2011}. Fixing $\alpha_{1}, \alpha_{2}$, the Jacobi polynomials form a set of orthogonal basis functions \textit{w.r.t.} the weight function $w^{\alpha_{1}, \alpha_{2}} \coloneqq (1 - \xi)^{\alpha_{1}}(1 + \xi)^{\alpha_{2}}$.

Given a one-to-one mapping $h_{\beta, r}(x)$ from $x\in\mathbb{R}$ to $\xi\in(-1, 1)$, we can formulate a novel orthogonal basis in the unbounded domain $\mathbb{R}$, through the images of the Jacobi polynomials under the mapping $\xi \coloneqq h_{\beta, r}(x)$,
\begin{equation}
    \begin{aligned}
        \int_{\mathbb{R}} j_{m}^{\alpha_{1}, \alpha_{2}}(\xi)j_{n}^{\alpha_{1}, \alpha_{2}}(\xi) w^{\alpha_{1}, \alpha_{2}}(\xi) \frac{\text{d}\xi}{\text{d}x} \text{d}x =& \sqrt{\gamma_{m} \gamma_{n}}\delta_{m, n},
    \end{aligned}
\end{equation}
where $\gamma_{n} \coloneqq \int_{-1}^1(j_{n}^{\alpha_{1}, \alpha_{2}})^2w^{\alpha_{1}, \alpha_{2}}\text{d}\xi$. Here, we consider a family of mappings $ h_{\beta, r}(x)$ in \citep[Section 2.2]{Shen2012} defined by
\begin{equation}\label{eq:1-2}
    \frac{\text{d}h_{\beta, r}(x)}{\text{d}x} = \beta \big(1 - h_{\beta, r}^{2}(x)\big)^{1 + r/2} \text{ and } h_{\beta, r}(0) = 0.
\end{equation}
Here, $\beta$ is the scaling factor, and $r \ge 0$ is a non-negative integer. For $r = 0, 1$, $h_{\beta, r}(x)$ can computed explicitly as follows:
\begin{equation}\label{eq:1-3}
    h_{\beta, r}(x) =
    \begin{cases}
        \tanh(\beta x) & \text{logarithmic mapping}~r = 0\\
        \displaystyle\frac{\beta x}{\sqrt{1 + \beta^{2}x^{2}}} &\text{algebraic mapping}~r = 1
    \end{cases}
\end{equation}

Using the mapping $h_{\beta, r}(x)$ defined in Eq. \eqref{eq:1-2}, we define the mapped Jacobi functions on the unbounded domain $\mathbb{R}$ as
\begin{equation}\label{eq:2-13}
    \begin{aligned}
        \mathcal{J}_{n, \alpha_{1}, \alpha_{2}, r}^{\beta, x_{0}}(x) \coloneqq \frac{1}{\sqrt{\gamma_{n}}} j_{n}^{\alpha_{1}, \alpha_{2}}\big(h_{\beta, r}(x - x_{0})\big)\mu_{\alpha_{1}, \alpha_{2}}\big(h_{\beta, r}(x - x_{0})\big),
    \end{aligned}
\end{equation}
where $\mu_{\alpha_{1}, \alpha_{2}}\big(h_{\beta, r}(x)\big)\coloneqq \sqrt{w^{\alpha_{1}, \alpha_{2}}\big(h_{\beta, r}(x)\big) h^{\prime}_{\beta, r}(x)}$ is the modified weight function. This modified weight function $\mu_{\alpha_{1}, \alpha_{2}}(x)$ makes the mapped Jacobi functions $\{\mathcal{J}_{n, \alpha_{1}, \alpha_{2}, r}^{\beta, x_{0}}\}$ a complete and orthogonal basis of the Hilbert space $L^{2}(\mathbb{R})$ \citep[Proposition 2.2]{Shen2013}. For notational simplicity, we omit $\alpha_1, \alpha_2$ for the mapped Jacobi basis functions and the mapping parameter $r$ in the subindex of $\mathcal{J}_{n, \alpha_{1}, \alpha_{2}, r}^{\beta, x_{0}}(x)$, \textit{i.e.}, we use the notation $\mathcal{J}_{n}^{\beta, x_{0}}(x)$ instead.

Detailed theoretical properties of the mapped Jacobi functions can be found in \cite{Shen2012, Shen2013, Sheng2019, Tang2019}. Different from the generalized Hermite functions which decay at an exponential rate of $e^{-\frac{1}{2}|\beta x|^{2}}$ for large $|x|$, the decaying rate of the mapped Jacobi basis functions can be tuned by choosing an appropriate $r$ in the mapping Eq.~\eqref{eq:1-2}. For example, when using the algebraic mapping ($r = 1$ in Eq.~\eqref{eq:1-2}), the mapped Jacobi basis functions decay at a rate of $|\beta x|^{-1}$ \cite{Shen2013} for large 
$|x|$, which are suitable for approximating a function that decays faster than $|x|^{-1}$ for large $|x|$ \cite{Shen2012}.

We shall use sparse mapped Jacobi spectral expansions defined in the hyperbolic cross space $V_{N, \gamma}^{\boldsymbol{\beta}, \boldsymbol{x}_{0}}$ characterized by the hyperbolic cross index set $\Upsilon_{N, \gamma}$ (defined in Table \ref{table:note}). The following two inverse inequalities hold for the mapped Jacobi spectral expansions in the hyperbolic cross space $V_{N,\gamma}^{\boldsymbol{\beta}, \boldsymbol{x}_{0}}$.

\begin{lemma}\label{th:5-1} For all $U_{N, \gamma}^{\boldsymbol{\beta}, \boldsymbol{x}_{0}} \in V_{N,\gamma}^{\boldsymbol{\beta}, \boldsymbol{x}_{0}}$,
\begin{equation}\label{eq:1-1}
    \big\|\partial_{x_{i}} U_{N, \gamma}^{\boldsymbol{\beta}, \boldsymbol{x}_{0}}\big\|_{L^{2}} \le \beta_{i}^{3/2} N_{\alpha, r}^{1/2} \big\|U_{N, \gamma}^{\boldsymbol{\beta},\boldsymbol{x}_{0}}\big\|_{L^{2}},
\end{equation}
If additionally restricting that $r \le 1$ in Eq.~\eqref{eq:1-2}, we have
\begin{equation}\label{eq:1-12}
    \big\|x_{i}\partial_{x_{i}} U_{N, \gamma}^{\boldsymbol{\beta}, \boldsymbol{x}_{0}}\big\|_{L^{2}} \le \beta_{i}^{1/2}N_{\alpha, r}^{1/2} \big\|U_{N, \gamma}^{\boldsymbol{\beta}, \boldsymbol{x}_{0}}\big\|_{L^{2}},
\end{equation}
where $N_{\alpha, r} \coloneqq 2N(N + \alpha_{1} + \alpha_{2} + 1) + 2(1 + \alpha_{1} + \alpha_{2} + r/2)^{2}$.
\end{lemma}
The proof of Lemma \ref{th:5-1} is given in \ref{ap:1}.

\subsection{Numerical scheme}\label{sec:2-3}
Here, we describe the AHMJ method to solve the model problem Eq.~\eqref{eq:2-2}. We define two function spaces, $X_{N,\gamma}^{\boldsymbol{\beta}, \boldsymbol{x}_{0}}(t_{0}, t_{1})$ and  $Y_{N,\gamma}^{\boldsymbol{\beta}, \boldsymbol{x}_{0}}(t_{0}, t_{1})$:
\begin{equation}
    \begin{aligned}
        X_{N,\gamma}^{\boldsymbol{\beta}, \boldsymbol{x}_{0}}(t_{0}, t_{1}) &\coloneqq \big\{U_{N, \gamma}^{\boldsymbol{\beta}, \boldsymbol{x}_{0}}\in L^{2}([t_{0}, t_{1}]; V_{N,\gamma}^{\boldsymbol{\beta}, \boldsymbol{x}_{0}}) : \partial_{t}U_{N, \gamma}^{\boldsymbol{\beta},\boldsymbol{x}_{0}}\in L^{2}([t_{0}, t_{1}]; V_{N,\gamma}^{\boldsymbol{\beta}, \boldsymbol{x}_{0}})\big\},\\
        Y_{N,\gamma}^{\boldsymbol{\beta}, \boldsymbol{x}_{0}}(t_{0}, t_{1}) &\coloneqq L^{2}\big([t_{0}, t_{1}]; V_{N,\gamma}^{\boldsymbol{\beta}, \boldsymbol{x}_{0}}\big) \times V_{N,\gamma}^{\boldsymbol{\beta}, \boldsymbol{x}_{0}}.\\
    \end{aligned}
\end{equation}
$X_{N,\gamma}^{\boldsymbol{\beta}, \boldsymbol{x}_{0}}(t_{0}, t_{1})$ is a subspace of the Sobolev-Bochner space $X(t_{0}, t_{1})$, which inherits the norm $\|\cdot\|_{X(t_{0}, t_{1})}$ defined in Eq. \eqref{eq:3-18}. The space $Y_{N,\gamma}^{\boldsymbol{\beta}, \boldsymbol{x}_{0}}(t_{0}, t_{1})$ is equipped with the norm
\begin{equation}
    \|\boldsymbol{v}\|_{Y(t_{0}, t_{1})}^{2} \coloneqq \|(v, \Tilde{v})\|_{Y(t_{0}, t_{1})}^{2} \coloneqq \int_{t_{0}}^{t_{1}}\|v\|_{H^{1}}^{2}\text{d}t + \|\Tilde{v}\|_{L^{2}}^{2}.
\end{equation}

To obtain a continuous-time mapped Jacobi approximation to the solution $u(\boldsymbol{x}, t)$ of the model problem Eq.~\eqref{eq:2-2}, we wish to find 
\begin{equation}\label{eq:3-17}
    \Tilde{U}_{N, \gamma}^{\boldsymbol{\beta}, \boldsymbol{x}_{0}}(\boldsymbol{x}, t) \coloneqq \sum_{\boldsymbol{n} \in \Upsilon_{N, \gamma}} \Tilde{u}_{\boldsymbol{n}}^{\boldsymbol{\beta}, \boldsymbol{x}_{0}}(t) \mathcal{J}_{\boldsymbol{n}}^{\boldsymbol{\beta}, \boldsymbol{x}_{0}}(\boldsymbol{x})\in X_{N,\gamma}^{\boldsymbol{\beta}, \boldsymbol{x}_{0}}(0, T),
\end{equation}
such that $\forall \big(v_{N}, \Tilde{v}_{N}\big) \in Y_{N,\gamma}^{\boldsymbol{\beta}, \boldsymbol{x}_{0}}(0, T)$,
\begin{equation}\label{eq:3-1}
    \begin{aligned}
        \big(\partial_{t}\Tilde{U}_{N, \gamma}^{\boldsymbol{\beta}, \boldsymbol{x}_{0}}, v_{N}\big) + a\big(\Tilde{U}_{N, \gamma}^{\boldsymbol{\beta}, \boldsymbol{x}_{0}}, v_{N}; t\big) = \big(f(\Tilde{U}_{N, \gamma}^{\boldsymbol{\beta}, \boldsymbol{x}_{0}}; t),v_{N}\big),~ \forall t \in [0, T],\,\,
        \big(\Tilde{U}_{N, \gamma}^{\boldsymbol{\beta}, \boldsymbol{x}_{0}}(\cdot, 0), \Tilde{v}_{N}\big) = \big(u_{0}(\cdot), \Tilde{v}_{N}\big).
    \end{aligned}
\end{equation}

We rearrange the coefficients in the hyperbolic-cross-space mapped Jacobi spectral expansion $ \Tilde{U}_{N, \gamma}^{\boldsymbol{\beta}, \boldsymbol{x}_{0}}$ into a vector by arranging the coefficients in dictionary order, \textit{i.e.}, we shall define the following order relation on the index set $\Upsilon_{N, \gamma}$:
\begin{equation}
    \boldsymbol{n}^{1}\le \boldsymbol{n}^{2} : \exists~ i\in\mathbb{N}^+ \text{ such that } n_{i}^{1} \le n_{i}^{2} \text{ and }  \forall j < i,~n_{j}^{1} = n_{j}^{2}.
\end{equation}
Thus, the basis functions can be indexed by $\{1, \cdots, |\Upsilon_{N, \gamma}|\}$, and $\Tilde{U}_{N, \gamma}^{\boldsymbol{\beta}, \boldsymbol{x}_{0}}$ can be rewritten as
\begin{equation}
    \Tilde{U}_{N, \gamma}^{\boldsymbol{\beta}, \boldsymbol{x}_{0}}(\boldsymbol{x}, t) \coloneqq \sum_{i = 1}^{|\Upsilon_{N, \gamma}|} \Tilde{u}_{i}^{\boldsymbol{\beta}, \boldsymbol{x}_{0}}(t) \mathcal{J}_{\boldsymbol{n}^{i}}^{\boldsymbol{\beta}, \boldsymbol{x}_{0}}(\boldsymbol{x}),
\end{equation}
Denoting 
\begin{equation}
\Tilde{\boldsymbol{u}}_{N, \gamma}^{\boldsymbol{\beta}, \boldsymbol{x}_{0}}(t)\coloneqq \big(\Tilde{u}_{1}^{\boldsymbol{\beta}, \boldsymbol{x}_{0}}(t), \cdots, \Tilde{u}_{|\Upsilon_{N, \gamma}|}^{\boldsymbol{\beta}, \boldsymbol{x}_{0}}(t)\big),
\end{equation}
$\Tilde{\boldsymbol{u}}_{N, \gamma}^{\boldsymbol{\beta}, \boldsymbol{x}_{0}}(t)$ satisfies the following ODE
\begin{equation}\label{eq:2-21}
    \begin{aligned}
        \frac{\text{d}}{\text{d}t}\Tilde{\boldsymbol{u}}_{N, \gamma}^{\boldsymbol{\beta}, \boldsymbol{x}_{0}} + A_{N}^{\boldsymbol{\beta}}(t) \Tilde{\boldsymbol{u}}_{N, \gamma}^{\boldsymbol{\beta}, \boldsymbol{x}_{0}} = F_{N}^{\boldsymbol{\beta}}(\Tilde{\boldsymbol{u}}_{N, \gamma}^{\boldsymbol{\beta}, \boldsymbol{x}_{0}}; t),~\forall t \in [0, T],\,\,
        \Tilde{u}_{i}^{\boldsymbol{\beta}, \boldsymbol{x}_{0}}(0) = \big(u_{0}, \mathcal{J}_{\boldsymbol{n}^{i}}^{\boldsymbol{\beta}, \boldsymbol{x}_{0}}\big),~\forall \boldsymbol{n}\in \Upsilon_{N, \gamma}.
    \end{aligned}
\end{equation}
%\deleted{$\Tilde{\boldsymbol{u}}_{N, \gamma}^{\boldsymbol{\beta}, \boldsymbol{x}_{0}} \coloneqq \big(u_{1}^{\boldsymbol{\beta}, \boldsymbol{x}_{0}}, \cdots, u_{i}^{\boldsymbol{\beta}, \boldsymbol{x}_{0}}, \cdots, u_{|\Upsilon_{N, \gamma}|}^{\boldsymbol{\beta}, \boldsymbol{x}_{0}}\big)$ is the vector of coefficients in the hyperbolic-cross-cross mapped Jacobi spectral expansion Eq.~\eqref{eq:3-17}.} 
Additionally, when acting on $\Tilde{\boldsymbol{u}}_{N, \gamma}^{\boldsymbol{\beta}, \boldsymbol{x}_{0}}$, the $i^{\text{th}}$ components of $A_{N}^{\boldsymbol{\beta}}\Tilde{\boldsymbol{u}}_{N, \gamma}^{\boldsymbol{\beta}, \boldsymbol{x}_{0}}$ and $F_{N}^{\boldsymbol{\beta}}(\Tilde{\boldsymbol{u}}_{N, \gamma}^{\boldsymbol{\beta}, \boldsymbol{x}_{0}})$ are calculated by
\begin{equation}
\begin{aligned}
    \big(A_{N}^{\boldsymbol{\beta}}(t)\Tilde{\boldsymbol{u}}_{N, \gamma}^{\boldsymbol{\beta}, \boldsymbol{x}_{0}}\big)_i &= \sum_{j = 1}^{|\Upsilon_{N, \gamma}|} a\big(\mathcal{J}_{\boldsymbol{n}^i}^{\boldsymbol{\beta}, \boldsymbol{x}_{0}}, \mathcal{J}_{\boldsymbol{n}^{j}}^{\boldsymbol{\beta}, \boldsymbol{x}_{0}}; t\big)\Tilde{u}_{j}^{\boldsymbol{\beta}, \boldsymbol{x}_{0}},\,\,
    \big(F_{N}^{\boldsymbol{\beta}}\big(\Tilde{\boldsymbol{u}}_{N, \gamma}^{\boldsymbol{\beta}, \boldsymbol{x}_{0}}; t\big)\big)_{i} &= \big(f\big(\Tilde{U}_{N, \gamma}^{\boldsymbol{\beta}, \boldsymbol{x}_{0}}; t\big), \mathcal{J}_{\boldsymbol{n}^i}^{\boldsymbol{\beta}, \boldsymbol{x}_{0}}\big).
\end{aligned}
\end{equation}
 
The ODE \eqref{eq:2-21} on the mapped Jacobi expansion coefficients usually cannot be analytically solved. Instead, it can be numerically solved using implicit Runge-Kutta (IRK) schemes \citep[Chapters 69-70]{Ern2021}. To be specific, we divide the time interval $[0, T]$ into $K$ subintervals $[t_{\ell}, t_{\ell + 1}]$ using a uniform step size $\Delta t$, where $t_{\ell} = \ell \Delta t$ for $\ell \in \{0, 1, 2, \cdots, K\}$. Given the parameters $(\boldsymbol{\beta}_{\ell}, N_{\ell}, {\boldsymbol{x}_{0}}_{\ell})$ within the time interval $(t_{\ell}, t_{\ell + 1})$ and the numerical solution at time $t=t_{\ell}$,
\begin{equation}
    U_{N, \gamma}^{\boldsymbol{\beta}, \boldsymbol{x}_{0}}(\boldsymbol{x}, t_{\ell}) = \sum_{i = 1}^{|\Upsilon_{N,\gamma}|} u_{i}^{\boldsymbol{\beta}, \boldsymbol{x}_{0}}(t_{\ell}) \mathcal{J}_{\boldsymbol{n}^{i}}^{\boldsymbol{\beta}, \boldsymbol{x}_{0}}(\boldsymbol{x}),
\end{equation}
the $q^{\text{th}}$-order IRK scheme for forwarding time from $t_{\ell}$ to $t_{\ell + 1}$ is
\begin{equation}\label{eq:3-16}
    \begin{aligned}
        \boldsymbol{u}_{N, \gamma}^{\boldsymbol{\beta}, \boldsymbol{x}_{0}}(t_{\ell + 1}) &= \boldsymbol{u}_{N, \gamma}^{\boldsymbol{\beta}, \boldsymbol{x}_{0}}(t_{\ell}) + \Delta t\sum_{s = 1}^{q}b^{s}_{RK} G_{N}^{\boldsymbol{\beta}}(\boldsymbol{w}_{s}, t_{\ell} + c^{s}_{RK}\Delta t),\\
        \boldsymbol{w}_{s} &= \boldsymbol{u}_{N, \gamma}^{\boldsymbol{\beta}, \boldsymbol{x}_{0}}(t_{\ell}) + \Delta t\sum_{r = 1}^{q} a^{rs}_{RK}G_{N}^{\boldsymbol{\beta}}(\boldsymbol{w}_{r}, t_{\ell} + c^{r}_{RK}\Delta t),
    \end{aligned}
\end{equation}
where 
$a^{rs}_{RK}$, $b^{s}_{RK}$ and $c^{s}_{RK}$ are the IRK coefficients. $G_{N}^{\boldsymbol{\beta}}$ on the RHS is given by
\begin{equation}
    G_{N}^{\boldsymbol{\beta}}(\boldsymbol{w}_{s}, t) \coloneqq F_{N}^{\boldsymbol{\beta}}(\boldsymbol{w}_{s}; t) - A_{N}^{\boldsymbol{\beta}}(t)\boldsymbol{w}_{s}.
\end{equation}
The numerical solution at $t_{\ell+1}$ is thus
\begin{equation}
    U_{N, \gamma}^{\boldsymbol{\beta}, \boldsymbol{x}_{0}}(\boldsymbol{x}, t_{\ell + 1}) = \sum_{i = 1}^{|\Upsilon_{N,\gamma}|} u_{i}^{\boldsymbol{\beta}, \boldsymbol{x}_{0}}(t_{\ell + 1}) \mathcal{J}_{\boldsymbol{n}^{i}}^{\boldsymbol{\beta}, \boldsymbol{x}_{0}}(\boldsymbol{x}).
\end{equation}
Existence of the solution to the IRK system Eq. \eqref{eq:3-16} is proved in \cite{Lubich2016}. 

Finally, it has been revealed that adaptively adjusting the parameters $\boldsymbol{\beta}, \boldsymbol{x}_0$, and $N$ over time is crucial for efficiently applying spectral methods to solve spatiotemporal equations \cite{Xia2020, Xia2020(1), Xia2022, Xia2023}. Suppose we use the IRK scheme Eq.~\eqref{eq:3-16} to forward time and get $U_{N_{\ell}, \gamma}^{\boldsymbol{\beta}_{\ell}, {\boldsymbol{x}_{0}}_{\ell}}(\boldsymbol{x}, t_{\ell+1})$ at $t_{\ell+1}$ given the numerical solution $U_{N_{\ell}, \gamma}^{\boldsymbol{\beta}_{\ell}, {\boldsymbol{x}_{0}}_{\ell}}(\boldsymbol{x}, t_{\ell})$ at $t_{\ell}$ (the IRK scheme will not adjust the three parameters $\boldsymbol{\beta}$, $\boldsymbol{x}_0$, and $N$). We then apply the adaptive hyperbolic-cross-space techniques for spectral methods (described in Section \ref{sec:4}) to update the parameters $(\boldsymbol{\beta}_{\ell}, N_{\ell}, {\boldsymbol{x}_{0}}_{\ell}) \longrightarrow (\boldsymbol{\beta}_{\ell+1}, N_{\ell+1}, {\boldsymbol{x}_{0}}_{\ell+1})$ and get the new numerical solution at $t_{\ell+1}$:
\begin{equation}\label{eq:4-22}
     U_{{N}_{\ell+1}, \gamma}^{\boldsymbol{\beta}_{\ell+1}, {\boldsymbol{x}_{0}}_{\ell+1}}(\boldsymbol{x}, t_{\ell + 1}) \longleftarrow \pi_{{N}_{\ell+1}, \gamma}^{{\boldsymbol{\beta}}_{\ell+1}, {\boldsymbol{x}_{0}}_{\ell+1}} U_{N_{\ell}, \gamma}^{\boldsymbol{\beta}_{\ell}, {\boldsymbol{x}_{0}}_{\ell}}(\boldsymbol{x}, t_{\ell + 1}).
\end{equation}
$\pi_{{N}_{\ell+1}, \gamma}^{{\boldsymbol{\beta}}_{\ell+1}, {\boldsymbol{x}_{0}}_{\ell+1}}$ is the projection operator defined in Table.~\ref{table:note}.

\section{Analysis on the AHMJ method}\label{sec:3}
In this section, we give an upper error bound of $\big\|u(\cdot, t) - U_{N, \gamma}^{\boldsymbol{\beta}, \boldsymbol{x}_{0}}(\cdot, t)\big\|_{L^{2}}$, where $u$ solves the model problem~\eqref{eq:2-2} and $U_{N, \gamma}^{\boldsymbol{\beta}, \boldsymbol{x}_{0}}(\cdot, t)$  is the numerical solution obtained by the AHMJ method described in Subsection~\ref{sec:2-3}, respectively. In Subsection~\ref{sec:3-1}, we analyze the error bound on the mapped Jacobi approximation which solves the continuous-time problem Eq.~\eqref{eq:3-1}. Next, we derive the error bound for applying the IRK scheme, detailed in Subsection \ref{sec:3-2}. Then, we carry out an analysis on the error bound for implementing adaptive techniques in Subsection~\ref{sec:3-3}. Integrating the aforementioned error analysis, we shall eventually prove Theorem~\ref{th:1-1}.

\subsection{Continuous-time mapped Jacobi approximation error}\label{sec:3-1}
In this subsection, we give the upper error bound of solving the continuous-time approximation equation Eq.~\eqref{eq:3-1} with a hyperbolic-cross-space mapped Jacobi approximation.
\begin{theorem}\label{th:6-1}
Suppose $\Tilde{U}_{N_{\ell}, \gamma}^{\boldsymbol{\beta}_{\ell}, {\boldsymbol{x}_{0}}_{\ell}}(\boldsymbol{x}, t) \in X_{N_{\ell}, \gamma}^{\boldsymbol{\beta}_{\ell}, {\boldsymbol{x}_{0}}_{\ell}}(t_{\ell}, t_{\ell + 1})$ solves

\begin{equation}\label{eq:4-23}
    \begin{aligned}
        \big(\partial_{t}\Tilde{U}_{N, \gamma}^{\boldsymbol{\beta}, \boldsymbol{x}_{0}}, v_{N}\big) + a\big(\Tilde{U}_{N, \gamma}^{\boldsymbol{\beta}, \boldsymbol{x}_{0}}, v_{N}; t\big) &= \big(f(\Tilde{U}_{N, \gamma}^{\boldsymbol{\beta}, \boldsymbol{x}_{0}}; t),v_{N}\big),~ \forall t \in [t_{\ell}, t_{\ell+1}],\\
        \big(\Tilde{U}_{N, \gamma}^{\boldsymbol{\beta}, \boldsymbol{x}_{0}}(\cdot, t_{\ell}), \Tilde{v}_{N}\big) &= \big(U(\boldsymbol{x}, t_{\ell}), \Tilde{v}_{N}\big),~ \forall \big(v_{N}, \Tilde{v}_{N}\big) \in Y_{N_{\ell},\gamma}^{\boldsymbol{\beta}_{\ell}, {\boldsymbol{x}_{0}}_{\ell}}(t_{\ell}, t_{\ell + 1}).
    \end{aligned}
\end{equation}
Then, there exist two constants, $C_{\mathcal{M}}$ and $c_{\mathcal{J}}$, that only depend on $a(u, v; t)$ and $f(u; t)$, such that
\begin{equation}\label{eq:2-20}
    \begin{aligned}
        \big\|u(\cdot, t_{\ell + 1}) - \Tilde{U}_{N_{\ell}, \gamma}^{\boldsymbol{\beta}_{\ell}, {\boldsymbol{x}_{0}}_{\ell}}(\cdot, t_{\ell + 1})\big\|_{L^{2}} &\le \exp\big((L - c_{0})\Delta t\big)\big\|u(\cdot, t_{\ell}) - U(\cdot, t_{\ell})\big\|_{L^{2}}\\
        &\quad+ C_{\mathcal{M}}\exp(c_{\mathcal{J}}\Delta t)\big\|u - \pi_{N_{\ell},\gamma}^{\boldsymbol{\beta}_{\ell}, {\boldsymbol{x}_{0}}_{\ell}}u\big\|_{X(t_{\ell}, t_{\ell + 1})},
    \end{aligned}
\end{equation}
where $u(\boldsymbol{x}, t)$ is the analytical solution to the model problem Eq. \eqref{eq:2-2}.
\end{theorem}
The proof of Theorem \ref{th:6-1} is given in \ref{ap:3}.

\subsection{Implicit Runge-Kutta scheme error}\label{sec:3-2}
Next, we discuss the error bound for implementing the IRK scheme Eq.~\eqref{eq:3-16} to forward time from $t_{\ell}$ to $t_{\ell+1}$ to solve Eq.~\eqref{eq:4-23}. 

Given $U_{N, \gamma}^{\boldsymbol{\beta}, \boldsymbol{x}_{0}}(\boldsymbol{x}, t_{\ell})$ at $t_{\ell}$ as the numerical solution at $t_{\ell}$, we have
\begin{equation}
    \begin{aligned}
        \big\|u(\cdot, t_{\ell + 1}) - U_{N_{\ell}, \gamma}^{\boldsymbol{\beta}_{\ell}, {\boldsymbol{x}_{0}}_{\ell}}(\cdot, t_{\ell + 1})\big\|_{L^{2}} &\le \big\|u(\cdot, t_{\ell + 1}) - \Tilde{U}_{N_{\ell}, \gamma}^{\boldsymbol{\beta}_{\ell}, {\boldsymbol{x}_{0}}_{\ell}}(\cdot, t_{\ell + 1})\big\|_{L^{2}}\\
        &\quad+ \big\|\Tilde{U}_{N_{\ell}, \gamma}^{\boldsymbol{\beta}_{\ell}, {\boldsymbol{x}_{0}}_{\ell}}(\cdot, t_{\ell + 1}) - U_{N_{\ell}, \gamma}^{\boldsymbol{\beta}_{\ell}, {\boldsymbol{x}_{0}}_{\ell}}(\cdot, t_{\ell + 1})\big\|_{L^{2}},
    \end{aligned}
\end{equation}
where $\Tilde{U}_{N_{\ell}, \gamma}^{\boldsymbol{\beta}_{\ell}, {\boldsymbol{x}_{0}}_{\ell}}(\boldsymbol{x}, t_{\ell + 1})$ is the solution of the continuous-time problem Eq. \eqref{eq:4-23}. $\big\|\Tilde{U}_{N_{\ell}, \gamma}^{\boldsymbol{\beta}_{\ell}, {\boldsymbol{x}_{0}}_{\ell}}(\cdot, t_{\ell + 1}) - U_{N_{\ell}, \gamma}^{\boldsymbol{\beta}_{\ell}, {\boldsymbol{x}_{0}}_{\ell}}(\cdot, t_{\ell + 1})\big\|_{L^{2}}$ is the error from applying the IRK scheme. The analysis of the IRK scheme has been carried out in \cite{Ostermann1995, Lubich2016}, which is presented in Theorem \ref{th:6-2}.

\begin{theorem}\label{th:6-2}
Let $U_{N_{\ell}, \gamma}^{\boldsymbol{\beta}_{\ell}, {\boldsymbol{x}_{0}}_{\ell}}(\boldsymbol{x}, t_{\ell + 1})$ be the numerical solution to Eq. \eqref{eq:4-23} obtained by the IRK scheme in Eq. \eqref{eq:3-16}. Suppose that the IRK scheme in Eq. \eqref{eq:3-16} satisfies
\begin{enumerate}
    \item \citep[section 4]{Lubich2016} The IRK scheme has a stage order $q$ and a quadrature order at least $q + 1$.
    \item \citep[algebraic stability]{Lubich2016} The weights $(b^{s}_{RK})_{s = 1}^{q}$ are positive, and the matrix $\mathcal{M}\coloneqq (a^{rs}_{RK}b^{r}_{RK} + a^{sr}_{RK}b^{s}_{RK} - b^{r}_{RK}b^{s}_{RK})_{r, s = 1}^{q}\in\mathbb{R}^{q\times q}$ is positive semi-definite.
\end{enumerate}
We assume that the time step $\Delta t$ satisfies:
\begin{equation}
    \Delta t \le \frac{c_{0}}{4\sqrt{2}L(C_{0} + L)C_{ab}} \text{ where } C_{ab}^{2} = \sum_{s = 1}^{q}\sum_{r = 1}^{q} \big(a_{RK}^{rs}\big)^{2}b_{RK}^{r}/b_{RK}^{s}.
\end{equation}
Then, there exists a constant $C_{RK}$ that depends on the bilinear form $a(u, v; t)$, the nonlinear operator $f(u; t)$, and the IRK coefficients, such that
\begin{equation}\label{RK_error_bound}
    \big\|\Tilde{U}_{N_{\ell}, \gamma}^{\boldsymbol{\beta}_{\ell}, {\boldsymbol{x}_{0}}_{\ell}}(\cdot, t_{\ell + 1}) - U_{N_{\ell}, \gamma}^{\boldsymbol{\beta}_{\ell}, {\boldsymbol{x}_{0}}_{\ell}}(\cdot, t_{\ell + 1})\big\|_{L^{2}} \le C_{RK} \Delta t^{q + 1} \big\|\partial_{t}^{(q + 1)}\Tilde{U}_{N_{\ell}, \gamma}^{\boldsymbol{\beta}_{\ell}, {\boldsymbol{x}_{0}}_{\ell}} \big\|_{X(t_{\ell}, t_{\ell + 1})},
\end{equation}
where $\Tilde{U}_{N_{\ell}, \gamma}^{\boldsymbol{\beta}_{\ell}, {\boldsymbol{x}_{0}}_{\ell}}(\boldsymbol{x}, t)$ is the solution to Eq. \eqref{eq:4-23}.
\end{theorem}

The proof of Theorem \ref{th:6-2} is in \ref{ap:4}. Combining the error estimation of the mapped Jacobi spectral method in Theorem \ref{th:6-1} and the IRK scheme in Theorem \ref{th:6-2}, we have:
\begin{equation}\label{eq:4-15}
    \begin{aligned}
        \big\|u(\cdot, t_{\ell + 1}) - U_{N_{\ell}, \gamma}^{\boldsymbol{\beta}_{\ell}, {\boldsymbol{x}_{0}}_{\ell}}(\cdot, t_{\ell + 1})\big\|_{L^{2}} &\le \exp\big((L - c_{0})\Delta t\big) \big\|u(\cdot, t_{\ell}) - U(\cdot, t_{\ell}) \big\|_{L^{2}}\\
        &\hspace{-1.1in}\quad + e_{\mathcal{J}}([t_{\ell}, t_{\ell + 1}])+ \big\|\Tilde{U}_{N_{\ell}, \gamma}^{\boldsymbol{\beta}_{\ell}, {\boldsymbol{x}_{0}}_{\ell}}(\cdot, t_{\ell + 1}) -U_{N_{\ell}, \gamma}^{\boldsymbol{\beta}_{\ell}, {\boldsymbol{x}_{0}}_{\ell}}(\cdot, t_{\ell + 1})\big\|_{L^{2}}\\
        &\hspace{-3.2cm}\le \exp\big((L - c_{0})\Delta t\big) \big\|u(\cdot, t_{\ell}) - U(\cdot, t_{\ell})\big\|_{L^{2}} + e_{\mathcal{J}}([t_{\ell}, t_{\ell + 1}]) + e_{RK}([t_{\ell}, t_{\ell + 1}]).
    \end{aligned}
\end{equation}
$e_{\mathcal{J}}([t_{\ell}, t_{\ell + 1}])$ and $e_{RK}([t_{\ell}, t_{\ell + 1}])$ are the mapped Jacobi approximation error bound and the IRK scheme error bound when advancing time from $t_{\ell}$ to $t_{\ell+1}$, respectively:
\begin{equation}
\begin{aligned}
    e_{\mathcal{J}}([t_{\ell}, t_{\ell + 1}]) &\coloneqq C_{\mathcal{M}}\exp(c_{\mathcal{J}}\Delta t)\big\|u - \pi_{N_{\ell},\gamma}^{\boldsymbol{\beta}_{\ell}, {\boldsymbol{x}_{0}}_{\ell}}u\big\|_{X(t_{\ell}, t_{\ell + 1})},\\
    e_{RK}([t_{\ell}, t_{\ell + 1}]) &\coloneqq C_{RK} \Delta t^{q + 1} \big\|\partial_{t}^{(q + 1)}\Tilde{U}_{N_{\ell}, \gamma}^{\boldsymbol{\beta}_{\ell}, {\boldsymbol{x}_{0}}_{\ell}} \big\|_{X(t_{\ell}, t_{\ell + 1})}.
\end{aligned}
\end{equation}

\subsection{Adaptive techniques error}\label{sec:3-3}
Finally, we analyze the error bound for adjusting the scaling factor $\boldsymbol{\beta}_{\ell}$, displacement ${\boldsymbol{x}_{0}}_{\ell}$, and the expansion order $N_{\ell}$ of the mapped Jacobi spectral expansion in Eq.~\eqref{eq:4-22}. We adopt the posterior estimation of the adaptive technique error of adjusting the parameters $\big(\boldsymbol{\beta}, N, \boldsymbol{x}_{0}\big)$ introduced in \cite{Xia2023}, which gives
\begin{equation}\label{eq:4-5}
    \big\|U_{N_{\ell}, \gamma}^{\boldsymbol{\beta}_{\ell}, {\boldsymbol{x}_{0}}_{\ell}}(\cdot, t_{\ell + 1}) - U_{N_{\ell + 1}, \gamma}^{\boldsymbol{\beta}_{\ell + 1}, {\boldsymbol{x}_{0}}_{\ell + 1}}(\cdot, t_{\ell + 1})\big\|_{L^{2}} \le e_{A}(t_{\ell + 1}),
\end{equation}
where
\begin{equation}
    e_{A}(t_{\ell + 1}) \coloneqq e_{m}(t_{\ell + 1}) + e_{s}(t_{\ell + 1}) + e_{c}(t_{\ell + 1}).
\end{equation}
Here, the moving error bound $e_{m}$, the scaling error bound $e_{s}$, and the coarsening error bound $e_{c}$ are given by:
\begin{equation}\label{adaerror}
\begin{aligned}
    e_{m}(t_{\ell + 1}) &\coloneqq \sum_{i = 1}^{d}\big|{x_{0}}_{\ell + 1, i} - {x_{0}}_{\ell, i}\big|\big\|\partial_{x_{i}}U_{N_{\ell}, \gamma}^{\boldsymbol{\beta}_{\ell}, {\boldsymbol{x}_{0}}_{\ell}}(\cdot, t_{\ell + 1})\big\|_{L^{2}},\\
    e_{s}(t_{\ell + 1}) &\coloneqq \sum_{i = 1}^{d} \left|1 - \frac{\beta_{\ell, i}}{\beta_{\ell + 1, i}}\right| \sqrt{\frac{\beta_{\ell + 1, i} + \beta_{\ell, i}}{2\beta_{\ell, i}}} \big\|x_{i}\partial_{x_{i}} U_{N_{\ell}, \gamma}^{{\boldsymbol{\beta}}_{\ell}, {\boldsymbol{x}_{0}}_{\ell + 1}}(\cdot, t_{\ell + 1})\big\|_{L^{2}},\\
    e_{c}(t_{\ell + 1}) &\coloneqq \big\|U_{N_{\ell}, \gamma}^{\boldsymbol{\beta}_{\ell + 1}, {\boldsymbol{x}_{0}}_{\ell + 1}}(\cdot, t_{\ell + 1}) - \pi_{N_{\ell + 1}, \gamma}^{\boldsymbol{\beta}_{\ell + 1}, {\boldsymbol{x}_{0}}_{\ell + 1}} U_{N_{\ell}, \gamma}^{\boldsymbol{\beta}_{\ell + 1}, {\boldsymbol{x}_{0}}_{\ell + 1}}(\cdot, t_{\ell + 1})\big\|_{L^{2}}.
\end{aligned}
\end{equation}
$\beta_{\ell, i}$ and ${x_{0}}_{\ell, i}$ denote the $i^{\text{th}}$ component of $\boldsymbol{\beta}_{\ell}$ and ${\boldsymbol{x}_{0}}_{\ell}$, respectively. Invoking the inverse inequalities (Lemma \ref{th:5-1}), we have
\begin{equation}
\begin{aligned}
    e_{m}(t_{\ell + 1}) &\le \sum_{i = 1}^{d}\big|{x_{0}}_{\ell + 1, i} - {x_{0}}_{\ell, i}\big| \sqrt{\beta_{\ell, i}^{3}} N_{\ell, \alpha, r}^{1/2}(t_{\ell}) \|U_{N_{\ell}, \gamma}^{{\boldsymbol{\beta}}_{\ell}, {\boldsymbol{x}_{0}}_{\ell}}(\cdot, t_{\ell + 1})\|_{L^{2}},\\
    e_{s}(t_{\ell + 1}) &\le \sum_{i = 1}^{d} \left|1 - \frac{\beta_{\ell, i}}{\beta_{\ell + 1, i}}\right| \sqrt{\frac{\beta_{\ell + 1, i} + \beta_{\ell, i}}{2}}N_{\ell, \alpha, r}^{1/2} \|U_{N_{\ell}, \gamma}^{{\boldsymbol{\beta}}_{\ell}, {\boldsymbol{x}_{0}}_{\ell}}(\cdot, t_{\ell + 1})\|_{L^{2}}.
\end{aligned}
\end{equation}
Here, $N_{\ell, \alpha, r}\coloneqq 2N_{\ell}(N_{\ell} + \alpha_{1} + \alpha_{2} + 1) + 2(1 + \alpha_{1} + \alpha_{2} + r/2)^{2}$.
Specifically, if $\boldsymbol{\beta}_{\ell} = \boldsymbol{\beta}_{\ell + 1}$, then $e_s(t_{\ell+1})=0$; if ${\boldsymbol{x}_{0}}_{\ell} = {\boldsymbol{x}_{0}}_{\ell + 1}$, then $e_m(t_{\ell+1})=0$; if $N_{\ell}\geq N_{\ell+1}$, then $e_c(t_{\ell+1})=0$.

Finally, by combining Eqs.~\eqref{eq:4-15} and~\eqref{eq:4-5}, the single-step error bound of the AHMJ method and the IRK scheme can be obtained:
\begin{equation}\label{eq:4-13}
    \begin{aligned}
        E(t_{\ell + 1}) \le \exp\big((L - c_{0})\Delta t\big) E(t_{\ell}) +  e_{\mathcal{J}}([t_{\ell}, t_{\ell + 1}]) + e_{RK}([t_{\ell}, t_{\ell + 1}]) + e_{A}(t_{\ell + 1}),
    \end{aligned}
\end{equation}
where $E(t_{\ell}) \coloneqq \|u(\cdot, t_{\ell}) - U_{N_{\ell}, \gamma}^{\boldsymbol{\beta}_{\ell}, {\boldsymbol{x}_{0}}_{\ell}}(\cdot, t_{\ell})\|_{L^{2}}$.
By iterating the single-time-step error bound in Eq.~\eqref{eq:4-13} from $t_{0} = 0$ to $t_{K} = T$, we give the error analysis in Theorem \ref{th:2-4}.

\begin{theorem}[Restated Theorem \ref{th:1-1}]\label{th:2-4}
Let $U_{N, \gamma}^{\boldsymbol{\beta}, \boldsymbol{x}_{0}}(\boldsymbol{x}, t)$ be the numerical solution of the AHMJ method and the implicit Runge-
Kutta scheme in Eq.~\eqref{eq:3-16}, then
\begin{equation}\label{eq:4-6}
    \begin{aligned}
        \big\|u(\cdot, T) - U_{N_{K}, \gamma}^{\boldsymbol{\beta}_{K}, {\boldsymbol{x}_{0}}_{K}}(\cdot, T)\big\|_{L^{2}}
        \le E_{\mathcal{J}}(T) + E_{RK}(T) + E_{A}(T).
    \end{aligned}
\end{equation}
The mapped Jacobi approximation error $E_{\mathcal{J}}$, the IRK scheme error $E_{RK}$ and the adaptive techniques error $E_{A}$ are given by
\begin{equation}
    E_{\mathcal{J}}(T) \coloneqq C_{\mathcal{M}}\exp(c_{\mathcal{J}}\Delta t)\sum_{\ell = 1}^{K}\Big( \exp\big((L - c_{0}) (T - t_{\ell})\big) \cdot \big\|u - \pi_{N_{\ell},\gamma}^{\boldsymbol{\beta}_{\ell}, {\boldsymbol{x}_{0}}_{\ell}}u\big\|_{X(t_{\ell - 1}, t_{\ell})}\Big),
\end{equation}
\begin{equation}
    E_{RK}(T) \coloneqq C_{RK} \Delta t^{q + 1}\sum_{\ell = 1}^{K} \Big(\exp\big((L - c_{0}) (T - t_{\ell})\big) \cdot \big\|\partial_{t}^{(q + 1)}\Tilde{U}_{N_{\ell}, \gamma}^{\boldsymbol{\beta}_{\ell}, {\boldsymbol{x}_{0}}_{\ell}}\big\|_{X(t_{\ell - 1}, t_{\ell})}\Big),
\end{equation}
\begin{equation}
    E_{A}(T) \coloneqq \sum_{\ell = 1}^{K} \Big(\exp\big((L - c_{0}) (T - t_{\ell})\big) \cdot \big(e_{m}(t_{\ell}) + e_{s}(t_{\ell}) + e_{c}(t_{\ell})\big)\Big),
\end{equation}
where $e_m$, $e_{s}$ and $e_{c}$ are defined in Eqs. \eqref{adaerror}.
\end{theorem}

Theorem \ref{th:2-4} gives the upper error bound when using the AHMJ method to solve the spatiotemporal integrodifferential equation Eq.~\eqref{eq:2-2}. Theorem \ref{th:2-4} greatly extends error analysis of using adaptive spectral methods in \cite[Theorem1]{Xia2023} on solving linear equations to solving a class of spatiotemporal integrodifferential equations in unbounded domains. Specifically,
the mapped Jacobi approximation error bound $E_{\mathcal{J}}$ only depends on the spectral expansion approximation to the analytical solution $u(\boldsymbol{x}, t)$, and the adaptive technique error bound $E_{A}$ depends on the implementing the adaptive techniques for the sparse spectral expansion. The application of the IRK scheme does not influence these error bounds $E_{\mathcal{J}}(T)$ and $E_A(T)$. Additionally, the error bound for using the IRK scheme $E_{RK}$ only depends on the high-order temporal derivative of the mapped Jacobi approximation $\Tilde{U}_{N, \gamma}^{\boldsymbol{\beta}, \boldsymbol{x}_{0}}(\boldsymbol{x}, t)$. Through the error analysis in Theorem~\ref{th:2-4}, we can control the error of the AHMJ method by separately analyzing and controlling the three error bounds $E_{\mathcal{J}}$, $E_{A}$, and $E_{RK}$.

\section{Numerical Results}\label{sec:4}
In this section, we first present the AHMJ method. To be specific, we introduce two hyperbolic-cross-space frequency indicators ($\mathcal{F}_{x_{i}}$ and $\mathcal{F}_{p}$) tailored for hyperbolic cross space to properly adjust the scaling factors in each dimension and adjust the expansion order $N$ of the hyperbolic cross space $V_{N, \gamma}^{\boldsymbol{\beta}, \boldsymbol{x}_{0}}$.

First, for a hyperbolic-cross-space mapped Jacobi spectral expansion $U_{N, \gamma}^{\boldsymbol{\beta}, \boldsymbol{x}_{0}} \in V_{N, \gamma}^{\boldsymbol{\beta}, \boldsymbol{x}_{0}}$, we define the hyperbolic-cross-space frequency indicator in the $i^{\text{th}}$ dimension as
\begin{equation}
    \mathcal{F}_{x_{i}}(U_{N, \gamma}^{\boldsymbol{\beta}, \boldsymbol{x}_{0}}) \coloneqq \frac{\big\|U_{N, \gamma}^{\boldsymbol{\beta}, \boldsymbol{x}_{0}} - \pi_{N, \gamma, i}^{\boldsymbol{\beta}, \boldsymbol{x}_{0}}U_{N, \gamma}^{\boldsymbol{\beta}, \boldsymbol{x}_{0}}\big\|_{L^{2}}}{\|U_{N, \gamma}^{\boldsymbol{\beta}, \boldsymbol{x}_{0}}\|_{L^{2}}},\quad \forall i \in \{1, \cdots, d\}.
    \label{scaling_hyper}
\end{equation}
Here, $\pi_{N, \gamma, i}^{\boldsymbol{\beta}, \boldsymbol{x}_{0}}$ denotes the projection operator onto the space spanned by basis functions whose indices fall into the following index set
\begin{equation}
    \Upsilon_{N, \gamma, i} \coloneqq \Big\{\boldsymbol{n} :\big|\big(n_{1}, \cdots, \tfrac{3}{2}n_{i},\cdots, n_{d}\big)\big|_{\text{mix}} \cdot \big|\big(n_{1}, \cdots, \tfrac{3}{2}n_{i},\cdots, n_{d}\big)\big|_{\infty}^{-\gamma} \le N^{1 - \gamma}\Big\},
\end{equation}
$(n_{1}, \cdots, \tfrac{3}{2}n_{i},\cdots, n_{d})$ indicates that the $i^{\text{th}}$ component of $\boldsymbol{n}$ is multiplied by a factor $\tfrac{3}{2}$ (following the common $\tfrac{2}{3}$-rule \cite{Hou2007, Xia2020}). $\mathcal{F}_{x_i}$ thus measures the high-frequency components in the $i^{\text{th}}$ direction of $U_{N, \gamma}^{\boldsymbol{\beta}, \boldsymbol{x}_{0}}$. $\mathcal{F}_{x_i}$ can help us adjust $\beta_i$, the scaling factor in the $i^{\text{th}}$ dimension.

Next, we define the hyperbolic-cross-space expansion order frequency indicator $\mathcal{F}_{p}$ for $U_{N, \gamma}^{\boldsymbol{\beta}, \boldsymbol{x}_{0}} \in V_{N, \gamma}^{\boldsymbol{\beta}, \boldsymbol{x}_{0}}$
\begin{equation}\label{p_hyper}
    \mathcal{F}_{p}(U_{N, \gamma}^{\boldsymbol{\beta}, \boldsymbol{x}_{0}}) \coloneqq \frac{\|U_{N, \gamma}^{\boldsymbol{\beta}, \boldsymbol{x}_{0}} - \pi_{N, \gamma , p}^{\boldsymbol{\beta}, \boldsymbol{x}_{0}}U_{N, \gamma}^{\boldsymbol{\beta}, \boldsymbol{x}_{0}}\|_{L^{2}}}{\|U_{N, \gamma}^{\boldsymbol{\beta}, \boldsymbol{x}_{0}}\|_{L^{2}}}.
\end{equation}
Here, $\pi_{N, \gamma, p}^{\boldsymbol{\beta}, \boldsymbol{x}_{0}}$ denotes the projection operator onto the space spanned by basis functions whose indices fall into the following index set
\begin{equation}
   \Upsilon_{N, \gamma, p} \coloneqq \Big\{\boldsymbol{n} :|\boldsymbol{n}|_{\text{mix}}\cdot |\boldsymbol{n}|_{\infty}^{-\gamma} \le \big(\tfrac{2}{3}N\big)^{1 - \gamma}\Big\}.
\end{equation}
$\mathcal{F}_p$ measures the overall high-frequency components in $U_{N, \gamma}^{\boldsymbol{\beta}, \boldsymbol{x}_{0}}$. For one-dimensional spectral expansions, the hyperbolic-cross-space frequency indicators $\mathcal{F}_{x}$ and $\mathcal{F}_{p}$ coincide with the frequency indicators introduced in \cite{Xia2020, Xia2020(1)}. 

Previous adaptive techniques for applying Hermite functions to solve multiple-dimensional spatiotemporal equations 
use a ``direct truncation" strategy \cite{Xia2022} to calculate the direct-truncation-strategy frequency indicators in the $i^{\text{th}}$ direction, and such strategies are mainly applied to the case of using the full-tensor-product spectral expansions instead of the sparse spectral expansions. As a comparison, we implement previous adaptive techniques for our sparse mapped Jacobi functions to solve multidimensional 
spatiotemporal equations, which we refer to as the ADMJ method. Specifically, the ADMJ method calculates the direct-truncation-strategy frequency indicator in the $i^{\text{th}}$ direction for adjusting $\beta_i$ using the following formula
%\added{As a comparison, we shall implement an adaptive direct truncation mapped Jacobi (ADMJ) method. This involves setting the expansion coefficients that correspond to the complement of the \added{hyperbolic-cross-space} index set \cite{Shen2010} to zero and utilizing the full-tensor-based frequency indicators \cite{Xia2020, Xia2020(1), Xia2022} to calculate the direct truncation  defined as}
\begin{equation}\label{c_scale}
    \Tilde{\mathcal{F}}_{x_i}(U_{N, \gamma}^{\boldsymbol{\beta}, \boldsymbol{x}_{0}}) \coloneqq \frac{\|U_{N, \gamma}^{\boldsymbol{\beta}, \boldsymbol{x}_{0}} - \Tilde{\pi}_{N, \gamma, i}^{\boldsymbol{\beta}, \boldsymbol{x}_{0}}U_{N, \gamma}^{\boldsymbol{\beta}, \boldsymbol{x}_{0}}\|_{L^{2}}}{\|U_{N, \gamma}^{\boldsymbol{\beta}, \boldsymbol{x}_{0}}\|_{L^{2}}},\quad \forall i \in \{1, \cdots, d\}.
\end{equation}
Here, $\Tilde{\pi}_{N, \gamma, i}^{\boldsymbol{\beta}, \boldsymbol{x}_{0}}$ denotes the projection operator onto the space spanned by basis functions whose indices fall into the following index set
\begin{equation}
    \Tilde{\Upsilon}_{N, \gamma, i} \coloneqq \Big\{\boldsymbol{n} \in \Upsilon_{N, \gamma} : n_{i} \le \tfrac{2}{3}N\Big\}.
\end{equation}
The direct-truncation-strategy expansion order frequency indicator $\Tilde{\mathcal{F}}_p$ for adjusting the expansion order $N$ is defined as
\begin{equation}\label{cp}
    \Tilde{\mathcal{F}}_{p}(U_{N, \gamma}^{\boldsymbol{\beta}, \boldsymbol{x}_{0}}) \coloneqq \frac{\|U_{N, \gamma}^{\boldsymbol{\beta}, \boldsymbol{x}_{0}}  - \Tilde{\pi}_{N, \gamma, p}^{\boldsymbol{\beta}, \boldsymbol{x}_{0}}U_{N, \gamma}^{\boldsymbol{\beta}, \boldsymbol{x}_{0}}\|_{L^{2}}}{\|U_{N, \gamma}^{\boldsymbol{\beta}, \boldsymbol{x}_{0}}\|_{L^{2}}}.
\end{equation}
Here, $\Tilde{\pi}_{N, \gamma, p}^{\boldsymbol{\beta}, \boldsymbol{x}_{0}}$ denotes the projection operator onto the space spanned by basis functions whose indices fall into the following index set 
\begin{equation}\label{dp}
    \Tilde{\Upsilon}_{N, \gamma, p} \coloneqq \Big\{\boldsymbol{n} \in \Upsilon_{N, \gamma} : n_{i} \le \tfrac{2}{3}N,~ \forall i = \{1, \cdots, d\}\Big\}.
\end{equation}

We plot the basis functions that we use to calculate the hyperbolic-cross-space frequency indicators $\mathcal{F}_{x_i}$ and $\mathcal{F}_p$ in Fig.~\ref{fig:ils} (a, b, c). Additionally, we plot the basis functions that we use to calculate the direct-truncation-strategy frequency indicators $\tilde{\mathcal{F}}_{x_i}$ and $\tilde{\mathcal{F}}_p$ in Fig.~\ref{fig:ils} (d, e, f). %For multidimensional mapped Jacobi spectral expansions, 
When $\gamma = -\infty$ in the index set $\Upsilon_{N, \gamma}$, our hyperbolic-space frequency indicators $\mathcal{F}_{x_i}$ and $\mathcal{F}_p$ coincide with the direct-truncation-strategy frequency indicators $\tilde{\mathcal{F}}_{x_i}$ and $\tilde{\mathcal{F}}_p$. When $\gamma>\infty$, compared to our hyperbolic-cross-space frequency indicators, the direct-truncation-strategy frequency indicators $\tilde{\mathcal{F}}_{x_i}$ and $\tilde{\mathcal{F}}_p$ fail to take into account the features of the hyperbolic cross space and use fewer basis functions in the calculation of the numerators in Eqs.~\eqref{c_scale} and ~\eqref{cp}. In Examples~\ref{example4} and \ref{example5}, we shall show that our AHMJ method is more robust and efficient than the ADMJ method.
\begin{figure}[t]\label{fig:ils}
    \centering
    \includegraphics[width = 5in]{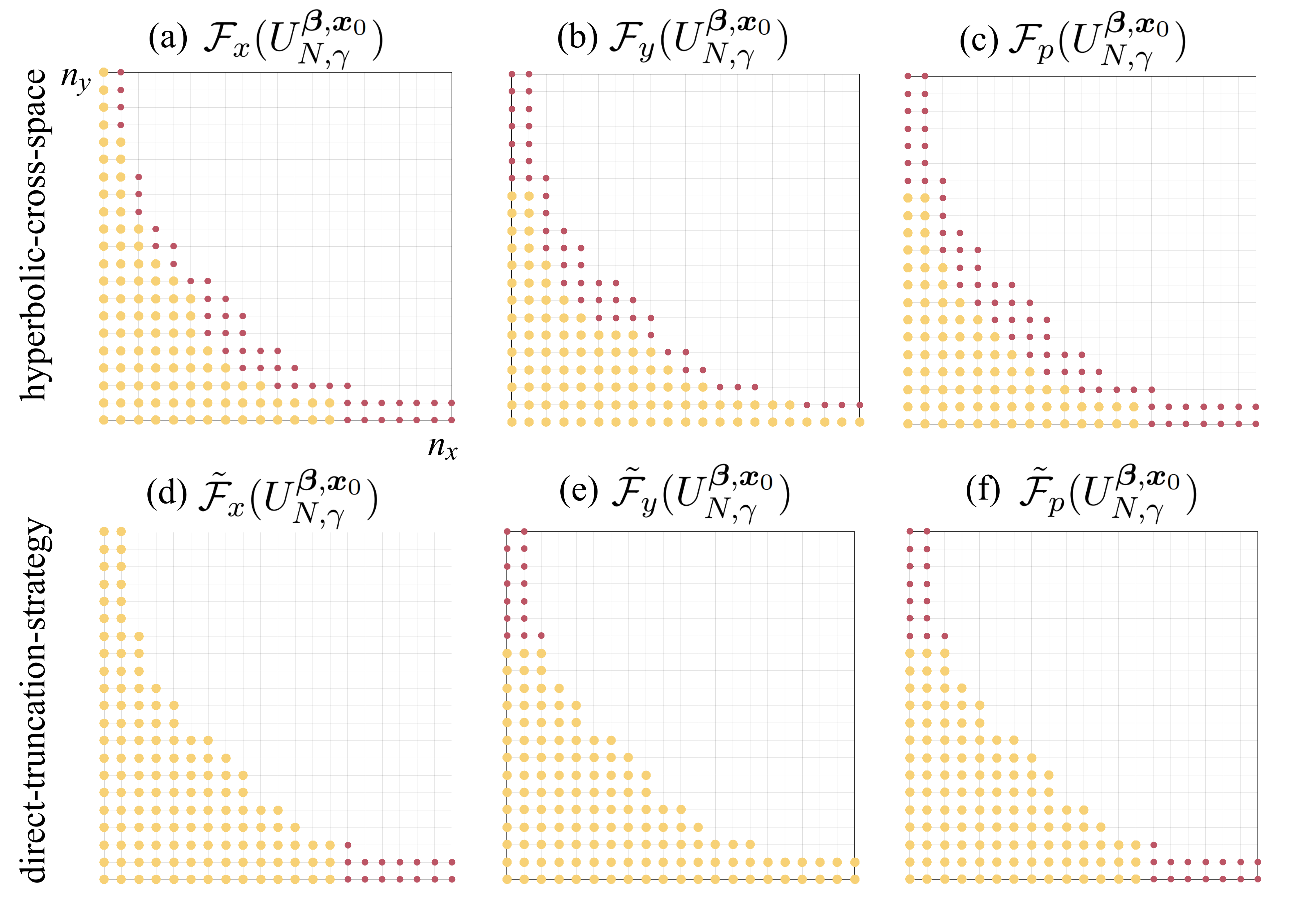}
    \caption{\textnormal{\footnotesize{(a, b, c) The basis functions that are used to our proposed hyperbolic-cross-space frequency indicators $\mathcal{F}_x, \mathcal{F}_y$, and $\mathcal{F}_p$ defined in Eqs~\eqref{scaling_hyper} and \eqref{p_hyper}. The red dots are the indices of the mapped Jacobi basis functions used in the calculation of the numerators of Eq.~\eqref{scaling_hyper} and \eqref{p_hyper}. The red and yellow dots are the indices of the mapped Jacobi basis functions used in the calculation of the denominators of Eqs.~\eqref{scaling_hyper} and \eqref{p_hyper}. (d, e, f) The basis functions used to calculate the direct-truncation-strategy frequency indicators $\tilde{\mathcal{F}}_x, \tilde{\mathcal{F}}_y$, and $\tilde{\mathcal{F}}_p$ defined in Eqs~\eqref{c_scale} and \eqref{cp}. The red dots are the indices of the mapped Jacobi basis functions used in the calculation of the numerators of Eq.~\eqref{c_scale}, and \eqref{cp}. The red and yellow dots are the indices of the mapped Jacobi basis functions used in the calculation of the denominators of Eq.~\eqref{c_scale}, and \eqref{cp}. Here, we take $N=20, \gamma=-1$ for the hyperbolic space $V_{N, \gamma}^{\boldsymbol{\beta}, \boldsymbol{x}_0}$.}}}
\end{figure}

%\added{Furthermore, by setting $\gamma \to -\infty$, the AHMJ method is the same as the ADMJ method.}

Finally, for both the ADMJ method and the AHMJ method, the exterior-error indicators for adjusting the displacement ${x_0}_i$ for the basis functions in the $i^{\text{th}}$ direction are calculated in the same way as the exterior-error indicators \cite{Xia2020}:
\begin{equation}\label{exte}
    \begin{aligned}
        \mathcal{E}_{L}^{x_{i}}(U_{N, \gamma}^{\boldsymbol{\beta}, \boldsymbol{x}_{0}}) &\coloneqq \frac{\big\|\partial_{x_{i}}U_{N, \gamma}^{\boldsymbol{\beta}, \boldsymbol{x}_{0}}\cdot \mathbb{I}_{\mathbb{R} \times \cdots \times (-\infty, x_{i, L}) \times \cdots \times \mathbb{R}}\big\|_{L^{2}}}{\big\|\partial_{x_{i}}U_{N, \gamma}^{\boldsymbol{\beta}, \boldsymbol{x}_{0}}\big\|_{L^{2}}},\\
        \mathcal{E}_{R}^{x_{i}}(U_{N, \gamma}^{\boldsymbol{\beta}, \boldsymbol{x}_{0}}) &\coloneqq \frac{\big\|\partial_{x_{i}}U_{N, \gamma}^{\boldsymbol{\beta}, \boldsymbol{x}_{0}}\cdot \mathbb{I}_{\mathbb{R} \times \cdots \times (x_{i, R}, \infty) \times \cdots \times \mathbb{R}}\big\|_{L^{2}}}{\big\|\partial_{x_{i}}U_{N, \gamma}^{\boldsymbol{\beta}, \boldsymbol{x}_{0}}\big\|_{L^{2}}},
    \end{aligned}
\end{equation}
where $\mathbb{I}_{A}$ denotes the characteristic function of the set $A$. $x_{i, L} = x^{\beta_{i}, {x_{0}}_{i}}_{[\frac{N}{3}]}$ and $x_{i, R} = x^{\beta_{i}, {x_{0}}_{i}}_{[\frac{2N + 2}{3}]}$ are the $[\frac{N}{3}]^{\text{th}}$ and $[\frac{2N + 2}{3}]^{\text{th}}$ node of the quadrature nodes $\{x^{\boldsymbol{\beta}, \boldsymbol{x}_{0}}_{n}\}_{n = 0}^{N}$ in the $x_{i}$ direction. Hyperparameters for implementing the adaptive spectral methods and the details of the scaling, moving, and $p$-adaptive techniques (similar to the implementation of adaptive techniques in \cite{Xia2023}) are given in \ref{appendix:2}.

The spatial fractional Laplacian operator $(-\Delta)^{\alpha/2}$ defined as the following singular integral \cite{Tang2019} will be often used in our numerical examples,
\begin{equation}
    (-\Delta)^{s/2} u(\boldsymbol{x}) \coloneqq C_{d, s} \operatorname{p.v.}\int_{\mathbb{R}^{d}} \frac{u(\boldsymbol{x}) - u(\boldsymbol{y})}{|\boldsymbol{x} - \boldsymbol{y}|^{s}}\text{d}\boldsymbol{y}, \,\, C_{d,s} = \frac{\alpha 2^{s - 1} \Gamma(\frac{s + d}{2})}{\pi^{d/2} \Gamma(\frac{2 - s}{2})}.
\end{equation}
Here, p.v. stands for the Cauchy principal value. In \cite{Sheng2019}, an efficient method for computing $\big((-\Delta)^{s/2}\mathcal{J}_{\boldsymbol{m}}^{\boldsymbol{\beta}, \boldsymbol{x}_{0}}, \mathcal{J}_{\boldsymbol{n}}^{\boldsymbol{\beta}, \boldsymbol{x}_{0}}\big)$ is given. In Examples~\ref{example1},~\ref{example3},~\ref{example4},and~\ref{example5}, we use the modified mapped Chebyshev functions ($\alpha_{1} = \alpha_{2} = -\tfrac{1}{2}$). In Example~\ref{example2}, we use the modified mapped Legendre functions ($\alpha_{1} = \alpha_{2} = 0$). 
In this study, the error denotes the following relative $L^2$ error
\begin{equation}
    e(t) \coloneqq \frac{E(t)}{\|u(\cdot, t)\|_{L^{2}}}=\frac{\big\|u(\cdot, t) - U_{N,\gamma}^{\boldsymbol{\beta}, \boldsymbol{x}_{0}}(\cdot, t)\big\|_{L^{2}}}{\|u(\cdot, t)\|_{L^{2}}}.
\end{equation}

We use a fourth-order IRK scheme and a timestep $\Delta t = 0.1$ in all examples. The IRK scheme is solved based on the Newton method in \cite{Southworth2022}. %\added{Besides, the source terms of each equation are provided in \ref{appendix:1}}.

First, we compare the performance of the AHMJ method versus the adaptive Hermite method \cite{Xia2023} for solving a 1D spatiotemporal integrodifferential equation where the solution exhibits algebraic decay at infinity.

\begin{exam}\label{example1}\rm
Consider the following 1D fractional reaction-diffusion equation
\begin{equation}\label{exam:1}
    \begin{aligned}
        \partial_{t}u + (-\Delta)^{1/2} u + u (1 - u^{2}) = f(x, t),\,\,
        u(x, 0) = \frac{1}{(1 + x^{2})^{6}},
    \end{aligned}
\end{equation}
where $f(x, t)$ is the source term given in Eq. \eqref{f1}. Eq.~\eqref{exam:1} admits an analytical solution 
\begin{equation}\label{exam:11}
    u(x, t) = \frac{(1 + t)^{12}}{\big((1 + t)^{2} + x^{2}\big)^{6}}.
\end{equation}

\begin{figure}[h]
    \centering
    \includegraphics[width = 5in]{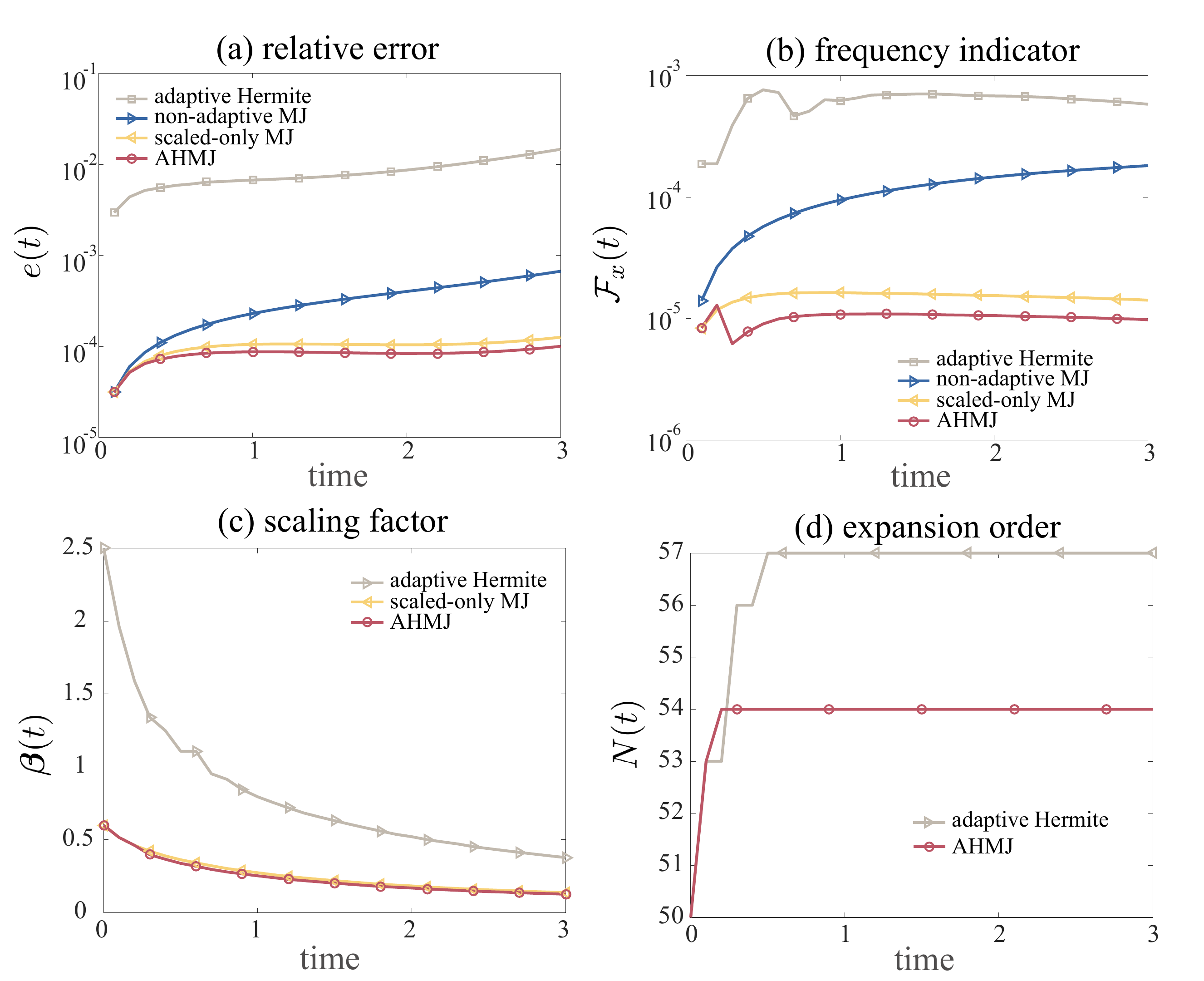}
    \caption{\textnormal{\footnotesize{(a) The errors of the non-adaptive mapped Jacobi method, the scaled-only mapped Jacobi method, and the AHMJ method as well as the adaptive Hermite method. (b) The frequency indicator of the adaptive Hermite method, the frequency indicators of the non-adaptive mapped Jacobi method, the scaled-only mapped Jacobi method, and the AHMJ method (c) The scaling factor $\beta$ of the scaled-only mapped Jacobi method, and the AHMJ method as well as the adaptive Hermite method. (d) The expansion order of the AHMJ method and the adaptive Hermite method.}}}\label{fig:1}
\end{figure}

We set $r = 1$ in Eq. \eqref{eq:1-3} so that the mapped Jacobi basis functions $\mathcal{J}_{n}^{\beta, x_{0}}$ decay algebraically at a rate of $|\beta x|^{-1}$ at infinity \cite{Shen2013}. The initial scaling factor $\beta = 0.6$ for the mapped Jacobi methods, and $\beta = 2.5$ for the Hermite method. 
We set the initial expansion order $N = 50$ and the initial displacement $x_{0} = 0$. 

As shown in Fig. \ref{fig:1} (a), the AHMJ method outperforms the adaptive Hermite method. This is because the algebraic decaying property of the mapped Jacobi functions matches the algebraic decay of the analytical solution Eq. \eqref{exam:11} as $|x|\rightarrow\infty$. On the other hand, the Hermite functions decay at a rate of $\exp(-\beta^2x^2/2)$ at infinity, and thus they cannot capture the solution's behavior when $|x|$ is large. Also, from Fig. \ref{fig:1} (a), our AHMJ method gives a much more accurate numerical solution than using a fixed scaling factor, which verifies the effectiveness of the scaling technique. From Fig. \ref{fig:1} (b), the frequency indicators of the AHMJ method and the scaled-only mapped Jacobi method (in this manuscript, "scaled-only" refers to the AHMJ method with the $p$-adaptive technique deactivated by disallowing the expansion order $N$ to increase or decrease) are well controlled as a result of properly adjusting the scaling factors (shown in Fig.~\ref{fig:1} (c)). The analytical solution Eq. \eqref{exam:1} does not become more or less oscillatory over time, and the expansion order $N$ for the ADHJ method remains almost unchanged (shown in Fig. \ref{fig:1} (d)) because the $p$-adaptive technique is rarely activated. However, the expansion order $N$ of the adaptive Hermite method is activated when $t$ is small which may be because the adaptive Hermite method cannot maintain a small frequency indicator only by scaling when $t$ is small (Fig. \ref{fig:1} (b)).
\end{exam}

In the next example, we further compare the performance of the AHMJ method with the adaptive Hermite method \cite{Xia2023} in solving a Keller-Segel equation that describes the dynamics of insect swarms in \cite{Grindrod, Carrillo2014}.
\begin{exam}\label{example2}\rm
Consider the following Keller-Segel equation
\begin{equation}\label{exam:4}
    \begin{aligned}
        \partial_{t}u + 2\cdot\nabla u - \frac{1}{2}\Delta u + \nabla\cdot\big(u \nabla (|x|\ast u)\big) = 0,\,\,
        u(x, 0) = \frac{1}{8}\cosh^{-2}\left(\frac{x}{4}\right),\\
    \end{aligned}
\end{equation}
which admits an analytical solution:
\begin{equation}\label{exam:42}
    u(x, t) = \frac{1}{8}\cosh^{-2}\left(\frac{x - 2t}{4}\right).
\end{equation}
\begin{figure}[h]
    \centering
    \includegraphics[width = 5in]{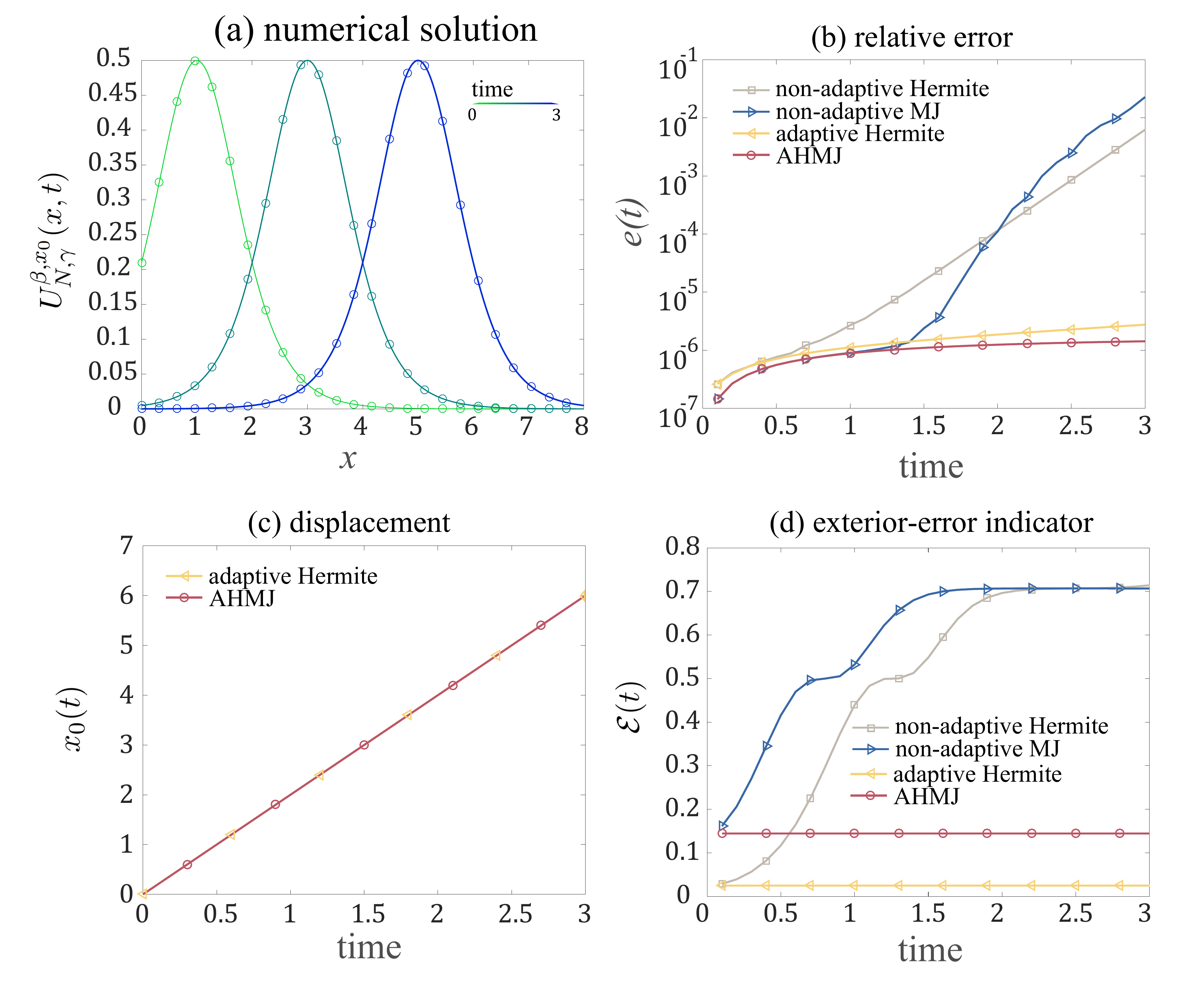}
    \caption{\textnormal{\footnotesize{(a) The analytical solution, which translates rightward over time. (b) The errors of the AHMJ, the non-adaptive mapped Jacobi, the adaptive Hermite, and the non-adaptive Hermite methods. (c) The displacement of the basis function for the AHMJ method as well as the displacement for the adaptive Hermite method. (d) The exterior-error indicators of the AHMJ, the non-adaptive mapped Jacobi, the adaptive Hermite, and the non-adaptive Hermite methods. The exterior-error indicator of the AHMJ method and the exterior-error indicator of the adaptive Hermite method are well controlled.}}}
    \label{fig:4}
\end{figure}

The solution of Eq.~\eqref{exam:42} decays exponentially at infinity. Therefore, we set $r = 1$ in Eq. \eqref{eq:1-3} such that the mapped Jacobi basis functions $\mathcal{J}_{n}^{\beta, x_{0}}$ decay at a rate of $\exp(-|\beta x|)$ to numerically solve Eq.~\eqref{exam:4}. As a comparison, we also apply the adaptive Hermive method \cite{Xia2023} to solve Eq.~\eqref{exam:4}. The initial scaling factor $\beta = 0.4$ for the mapped Jacobi methods, and the initial scaling factor $\beta = 1.2$ for the Hermite methods. We set the initial expansion order $N = 50$ and the initial displacement $x_{0} = 0$ for the mapped Jacobi and Hermite methods. Also, the analytical solution corresponds to a pulse moving rightward with constant speed over time (Fig.~\ref{fig:4} (a)), requiring properly adjusting the displacement $x_0$ of the mapped Jacobi spectral expansion. 

The AHMJ method can achieve high accuracy compared to the non-adaptive Hermite method and the non-adaptive mapped Jacobi methods (Fig.~\ref{fig:4} (b)). Failure to adjust $x_0$ will lead to a large right exterior-error indicator (shown in Fig.~\ref{fig:4} (c)), indicating a large error for the spectral expansion approximation as $x\rightarrow\infty$. For both the adaptive Hermite and the AHMJ method, they can accurately capture the change in the displacement $x_0$ (Fig.~\ref{fig:4} (d)), again verifying the effectiveness of the moving technique. From Eq.~\eqref{exam:42}, the solution is decaying at a rate of $\exp(-|x|/2)$ when $x\rightarrow\pm\infty$. As the Hermite functions vanish faster than the analytic solution at infinity,
% it cannot accurately capture the analytic solution's decaying behavior when $|x|$ is large. Therefore, 
the error of the adaptive Hermite method is slightly larger than the error of the AHMJ method (Fig.~\ref{fig:4} (b)).
\end{exam}

From Examples~\ref{example1} and \ref{example2}, the AHMJ method is more appropriate than the adaptive Hermite method if the solution decays algebraically or decays at a rate of $\exp(-\beta|x|)$ as $|x|\rightarrow\infty$. Next, we shall apply the AHMJ method to solve multidimensional spatiotemporal integrodifferential equations.

\begin{exam}\label{example3}\rm
Consider the following 2D fractional advection-diffusion equation
\begin{equation}\label{exam:23}
    \begin{aligned}
        \partial_{t}u + v\cdot \nabla u + (-\Delta)^{1/2} u + u(1 - u) &= f(x, y, t),~ v = \big(\cos(\tfrac{\pi}{3}), \sin(\tfrac{\pi}{3})\big),\\
        u(x, y, 0) &= \frac{1}{(1 + x^{2} + y^{2})^{7}},
    \end{aligned}
\end{equation}
where $f(x, y, t)$ is the source term given in Eq. \eqref{f3}. Eq. \eqref{exam:23} has an analytical solution
\begin{equation}\label{exam:2}
    u(x, y, t) = \frac{(t + 1)^{13}}{\big((t + 1)^{2} + \big(x - \cos\left(\frac{\pi}{3}\right)t\big)^{2} + \big(y - \sin\left(\frac{\pi}{3}\right)t\big)^{2}\big)^{7}}.
\end{equation}
\begin{figure}[h!]
    \centering
    \includegraphics[width = 4.6in]{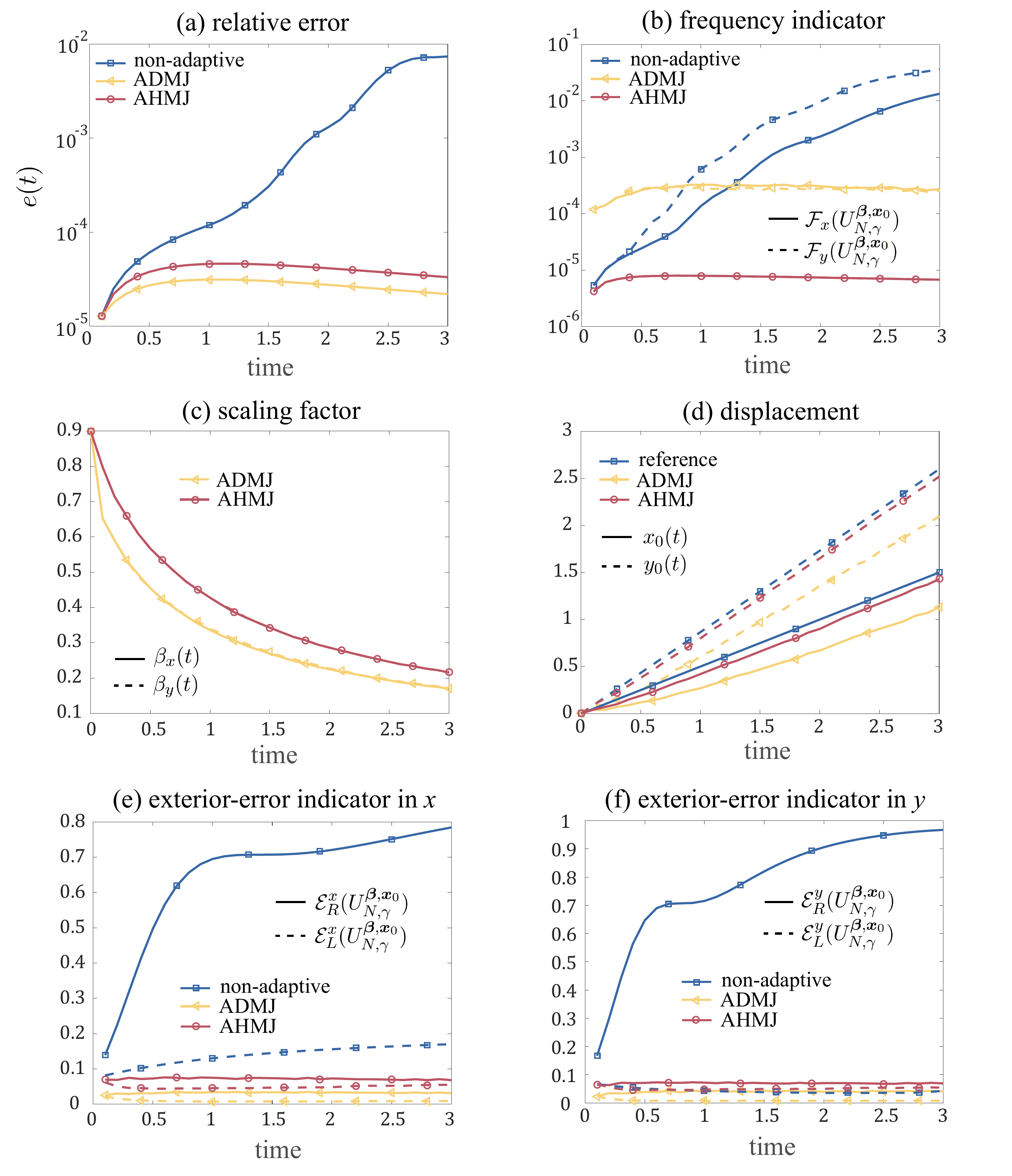}
    \caption{\textnormal{\footnotesize{(a) The errors of the non-adaptive mapped Jacobi, the ADMJ, and the AHMJ methods. (b) The frequency indicators of the non-adaptive, the ADMJ, and the AHMJ methods. (c) The scaling factors of the ADMJ method and the AHMJ method. (d) The displacements $x_0, y_0$ of the ADMJ method and the AHMJ method. Here, the reference displacement is the center of the analytical solution Eq. \eqref{exam:2} $(x(t), y(t)) = (\cos(\frac{\pi}{3})t, \sin(\frac{\pi}{3})t)$. (e, f) The left and right exterior-error indicators (Eq. \eqref{exte}) of the ADMJ method and the AHMJ method.}}}
    \label{fig:2}
\end{figure}

The analytical solution $u(\boldsymbol{x}, t)$ in Eq. \eqref{exam:2} decays at a rate of $|x^{2} + y^{2}|^{-7}$ at infinity. Thus, we set $r = 1$ in Eq. \eqref{eq:3-1} for the mapped Jacobi basis functions so that the basis functions also decay algebraically at infinity. We set the initial hyperbolic cross index set as $N = 30$ and $\gamma = -5$, the initial scaling factor $\boldsymbol{\beta} = (0.9, 0.9)$, and the initial displacement $\boldsymbol{x}_{0} = (0, 0)$. The analytic solution $u(x, y, t)$ translates at a velocity $v = (\cos(\frac{\pi}{3}), \sin(\frac{\pi}{3}))$ and also diffuses over time, so both the scaling technique and the moving technique are required to capture the diffusive and translative behavior of the solution.

The AHMJ method and the ADMJ method can achieve a much smaller error compared to the non-adaptive mapped Jacobi method (shown in Fig.~\ref{fig:2} (a)) because both the ADMJ method and the AHMJ method can adaptively adjust the scaling factors $\beta_x,\beta_y$ as well as the displacements ${x_0}, y_0$ in both directions (Fig.~\ref{fig:2} (c, d)). From (Fig. \ref{fig:2} (e, f)), both the ADMJ method and the AHMJ method can maintain the frequency indicators as well as the left- and right-exterior-error indicators small. In contrast, the non-adaptive mapped Jacobi method fails to do so and results in a large frequency indicator as well as a right-exterior-error indicator. Though the errors of the ADMJ method and the AHMJ method are close to each other, the AHMJ method gives more accurate displacements $x_0, y_0$. This could result from the fact that the AHMJ method can more accurately adjust the scaling factors $\beta_x, \beta_y$ using the hyperbolic frequency indicators in Eq. \eqref{scaling_hyper}.
\end{exam}

Next, we compare the performance of the proposed AHMJ method with the ADMJ method in numerically solving a 3D fractional diffusion equation.
\begin{exam}\label{example4}\rm
Consider the following 3D fractional diffusion equation
\begin{equation}\label{exam:35}
    \begin{aligned}
        \partial_{t}u + (-\Delta)^{3/4} u = f(x, y, z, t),
        u(x, y, z, 0) = \sin\big(x + 6y/5 + z/2\big) \exp\left(\frac{-(x^{2} + y^{2} + z^{2})}{2}\right),
    \end{aligned}
\end{equation}
where $f(x, y, z, t)$ is the source term given in Eq. \eqref{f5}. Eq. \eqref{exam:33} admits an analytical solution
\begin{equation}\label{exam:5}
    u(x, y, z, t) = \frac{\sin\big(x + 6y/5 + z/2\big)}{(3t + 1)^{3/2}} \exp\left(\frac{-(x^{2} + y^{2} + z^{2})}{6t + 2}\right).
\end{equation}

As the analytical solution decays exponentially at infinity, we set $r = 0$ in Eq. \eqref{eq:1-3} for the mapped Jacobi basis functions so that the basis functions decay at a rate of $\exp(-|\beta x|)$ for large $|x|$.
%, which has been shown to perform better than using an algebraic mapping ($r = 0$) for approximating exponentially decaying functions \citep[Chapter 7.5]{Shen2011}. 
We set the initial hyperbolic cross index set as $N = 25$ and $\gamma = -10$, the initial scaling factor $\boldsymbol{\beta} = (0.4, 0.37, 0.3)$, and the initial displacement $\boldsymbol{x}_{0} = (0, 0, 0)$. The analytical solution $u(x, y, z, t)$ in Eq. \eqref{exam:5} decays more slowly at infinity over time in $x$, $y$, and $z$ directions, which requires properly decreasing the scaling factors $\beta_x, \beta_y, \beta_z$ in all directions. By conducting a change of variable $\tilde{x}_t=\frac{x}{\sqrt{3t+1}}$, $\tilde{y}_t=\frac{y}{\sqrt{3t+1}}$, and $\tilde{z}_t=\frac{z}{\sqrt{3t+1}}$, the analytic solution Eq.~\eqref{exam:5} can be rewritten as
\begin{equation}
    u(\tilde{x}_t, \tilde{y}_t, \tilde{z}_t, t) = \frac{\sin\big(\sqrt{3t+1}(\tilde{x}_t + 6\tilde{y}_t/5 + \tilde{z}_t/2)\big)}{(3t + 1)^{3/2}}\exp\left(\frac{-(\tilde{x}^{2}_{t} + \tilde{y}^{2}_{t} + \tilde{z}^{2}_{t})}{2}\right).
\label{analytic_scaled}
\end{equation}
Thus, after appropriately decreasing $\beta_x, \beta_y$, and $\beta_z$, the factor $\sin\big(\sqrt{3t+1}(\tilde{x}_t + 6\tilde{y}_t/5 + \tilde{z}_t/2)\big)$ in Eq.~\eqref{analytic_scaled} becomes more oscillatory and requires incorporating higher-order basis functions to capture such oscillatory behavior.

\begin{figure}[h]
    \centering
    \includegraphics[width = \linewidth]{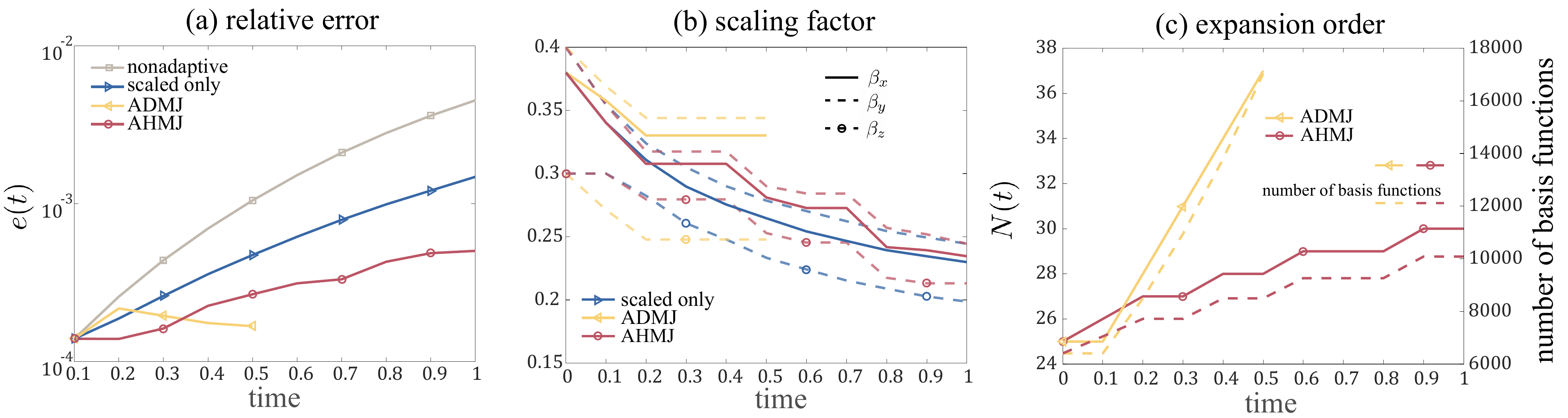}
    \caption{\textnormal{\footnotesize{(a) The errors of the non-adaptive, scaled-only mapped Jacobi method, the ADMJ method, and the AHMJ method.  (b) The scaling factors $\beta_x, \beta_y$, and $\beta_z$ of the scaled-only mapped Jacobi method, ADMJ method, and the AHMJ method, respectively. (c) The expansion order $N$ is generated by the ADMJ method and generated by the AHMJ method.} Our proposed AHMJ method can maintain the error small over time without using a too large number of basis functions, while the previous ADMJ method terminates prematurely because the number of basis functions increases too fast, leading to memory overflow.}}\label{fig:5}
\end{figure}
%\added{

In Fig. \ref{fig:5} (a), our proposed AHMJ method exhibits an improved performance compared to the non-adaptive and scaled-only mapped Jacobi methods, successfully controlling the relative error below $10^{-3}$. This is because our AHMJ method can properly decrease the scaling factors $\beta_x, \beta_y$, and $\beta_z$ in all three directions (Fig. \ref{fig:3} (b)) and increase $N$ (Fig. \ref{fig:3} (c)).

On the other hand, the ADMJ method increases $N$ too much (Fig. \ref{fig:3} (c)) and terminates before $t=1$ as a result of memory overflow. The direct-truncation expansion order frequency indicator uses a too-small number of basis functions for calculating the numerator of $\tilde{\mathcal{F}}_p$. Thus, $\tilde{\mathcal{F}}_p$ could be subjected to large fluctuations. Therefore, the ADMJ method can be less robust, which may lead to a drastic increase in the total number of basis functions and memory overflow. In comparison, the proposed AHMJ method is more robust and prevents $N$ and the number of basis functions from increasing too fast (shown in Fig.~\ref{fig:5} (c)).
\end{exam}

Finally, we extend our adaptive hyperbolic-cross-space techniques to generalized Hermite spectral expansions for numerically solving an unbounded domain spatiotemporal equation.
\begin{exam}\label{example5}\rm
Consider the following 4D equation in \cite{Luo2013}
\begin{equation}\label{exam:33}
    \begin{aligned}
        \partial_{t}u -\Delta u& + (x^{2} + y^{2} + z^{2} + w^{2}) u = f(x, y, z, w, t),\\
        u&(x, y, z, w) = \cos(x + y + z + w)\exp\left(-(x^{2} + y^{2} + z^{2} + w^{2})\right),
    \end{aligned}
\end{equation}
where $f(x, y, z, w, t)$ is the source term given in Eq. \eqref{f4}. Eq. \eqref{exam:33} admits an analytical solution
\begin{equation}\label{exam:3}
    u(x, y, z, w, t) = \frac{\cos(x + y + z + w)}{(t + 1)^{2}} \exp\left(\frac{-(x^{2} + y^{2} + z^{2} + w^{2})}{t + 1}\right).
\end{equation}

The analytical solution exhibits exponential decay at infinity, which is consistent with the decaying rate of the Hermite basis functions. Therefore, we use the Hermite basis functions. We set the initial hyperbolic cross index set as $N = 11$ and $\gamma = -3$, the initial scaling factor $\boldsymbol{\beta} = (1.05, 1.05, 1.05, 1.05)$, and the initial displacement $\boldsymbol{x}_{0} = (0, 0, 0, 0)$. The analytical solution Eq.~\eqref{exam:3} also requires both decreasing the scaling factors in all four directions and increasing the expansion order $N$. We shall apply our proposed hyperbolic-cross-space frequency indicators $\mathcal{F}_{x_i}$ and $\mathcal{F}_p$ defined in Eqs.~\eqref{scaling_hyper} and \eqref{p_hyper} for adjusting the scaling factors and the expansion order (denoted as \textbf{adaptive hyperbolic-cross-space Hermite method}) versus applying previous adaptive Hermite methods in \cite{Xia2020, Xia2022} (denoted as \textbf{adaptive Hermite method}), which use the direct-truncation-strategy frequency indicators  $\tilde{\mathcal{F}}_{x_i}$ and $\tilde{\mathcal{F}}_p$ defined in Eqs.~\eqref{c_scale} and \eqref{cp} for adjusting the scaling factors and the expansion order.
\begin{figure}[h]
    \centering
    \includegraphics[width = \linewidth]{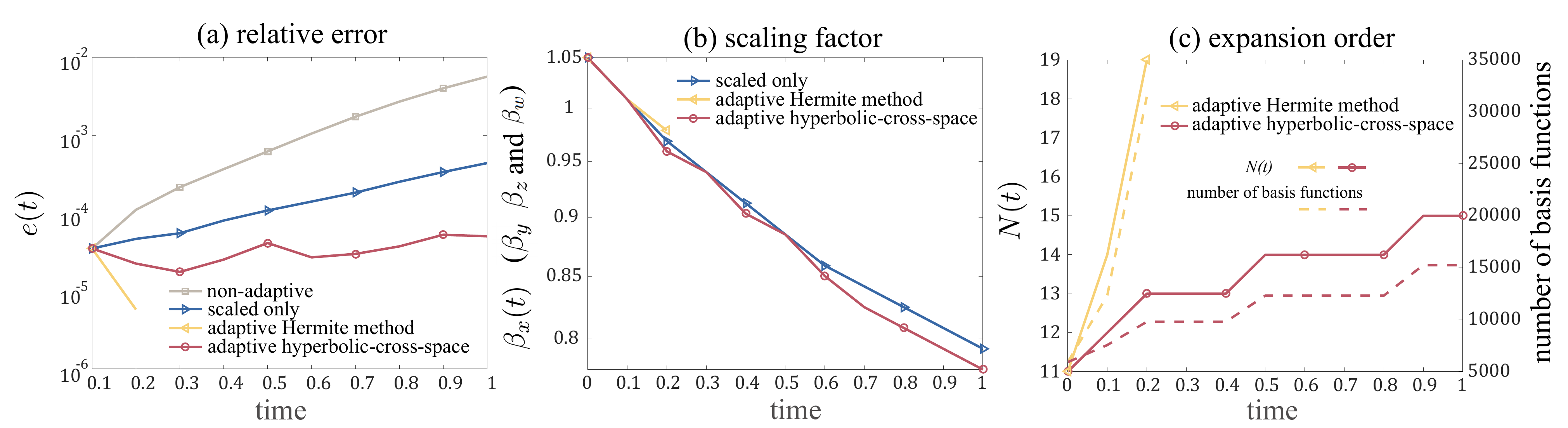}
    \caption{\textnormal{\footnotesize{(a) The errors of the non-adaptive, scaled-only (the adaptive hyperbolic-cross-space Hermite method with the $p$-adaptive technique for adjusting the expansion order $N$ deactivated), and adaptive Hermite methods, as well as the adaptive hyperbolic-cross-space Hermite method. (b) The scaling factors $\beta_x, \beta_y$, $\beta_z$, and $\beta_w$ of the scaled-only and adaptive Hermite methods as well as the adaptive hyperbolic-cross-space Hermite method. Since the analytic solution is homogeneous, the scaling technique yields the same $\beta_x, \beta_y, \beta_z, \beta_w$ in four directions for the adaptive Hermite methods. (c) The expansion order $N$ and the number of basis functions generated from the adaptive Hermite method and the adaptive hyperbolic-cross-space Hermite method. The adaptive Hermite method terminates prematurely as a result of a too-large number of basis functions and memory overflow.}}}
    \label{fig:3}
\end{figure}

In Fig. \ref{fig:3} (a), the adaptive hyperbolic-cross-space Hermite method exhibits an improved accuracy compared to the non-adaptive and scaled-only Hermite methods, successfully controlling the relative error below $10^{-4}$. The adaptive hyperbolic-cross-space Hermite method can properly decrease the scaling factors $\beta_x$, $\beta_y$, $\beta_z$, and $\beta_w$ in all four directions (Fig. \ref{fig:3} (b)) and appropriately increase the expansion order $N$ (Fig. \ref{fig:3} (c)). The previous adaptive Hermite method increases $N$ too fast (Fig. \ref{fig:3} (c)) and terminates prematurely before $t=1$ as a result of memory overflow. 
Again, from (Fig. \ref{fig:3} (c)), the previous direct-truncation strategy for adjusting the expansion order is less robust and subjects to memory overflow as a result of a too-fast-increasing number of basis functions. For both the adaptive Hermite method and the adaptive hyperbolic-cross-space Hermite method, we find that the moving technique will not be activated because the function is origin-symmetric.
%The calculation of the direct truncation expansion order frequency indicator \added{the numerator of $\tilde{\mathcal{F}}_p$ uses a too small number of basis functions, thus the calculated $\tilde{\mathcal{F}}_p$ could be subjected to large fluctuations.} Thus, it can be less robust and subject to fluctuations, which may lead to a drastic increase in the total number of basis functions especially when solving high-dimensional problems. In comparison, 
Compared to the previous adaptive Hermite method, our adaptive hyperbolic-cross-space Hermite method is more robust and prevents the number of basis functions from increasing too fast while achieving high accuracy. In conclusion, our adaptive hyperbolic-cross-space techniques can also be applied to generalized Hermite spectral expansions, and the resulting adaptive hyperbolic-cross-space Hermite method is more robust and efficient than previous adaptive Hermite methods.

% can avoid using redundant basis functions compared to previous adaptive direct truncation techniques intended for full-tensor-product spectral expansions.
\end{exam}

\section{Conclusions}\label{sec:5}
In this paper, we proposed an adaptive hyperbolic-cross-space mapped Jacobi (AHMJ) method for efficiently solving multidimensional spatiotemporal integrodifferential equations in unbounded domains whose solutions decay algebraically at infinity. We devised two hyperbolic-cross-space frequency indicators $\mathcal{F}_{x_i}$ and $\mathcal{F}_p$ in Eqs.~\eqref{scaling_hyper} and \eqref{p_hyper} for efficiently implementing adaptive techniques to sparse multidimensional spectral expansions defined in hyperbolic cross spaces \cite{Shen2010, Luo2013}. Our AHMJ method is more robust compared to previous adaptive techniques for spectral methods \cite{Xia2020, Xia2020(1), Xia2023} and can effectively reduce the number of basis functions needed while maintaining accuracy. Additionally. we provided an upper error bound for applying our AHMJ method to solve a wide range of spatiotemporal integrodifferential equations. We showed that the error of implementing our AHMJ method can be effectively controlled
as long as we chose an appropriate hyperbolic cross space and properly implemented the adaptive techniques for the sparse spectral expansions. 

A promising future direction is to figure out an appropriate strategy for modifying the hyperbolic cross space index $\gamma$, and to develop adaptive techniques on the asymmetric hyperbolic cross space as presented in \cite{Griebel2007}. This could require further investigation on how to tackle the heterogeneity of a multidimensional function. Furthermore, considering operator-splitting strategies for forwarding time, which are easier to implement than high-order implicit Runge-Kutta schemes, could be prospective \cite{Shen2012(1), Raviola2024, wang2024time, raviola2024performance}. 
Also, it is prospective to extend hyperbolic-cross-space adaptive techniques to generalized Laguerre functions in the semi-unbounded domain $\mathbb{R}^+$ \cite{vismara2022seamless, vismara2024efficient, xiong2022short}.
Finally, applying the proposed AHMJ method to solve inverse problems of reconstructing spatiotemporal equations \cite{Xia2023A} is worth further research.

\section*{CRediT authorship contribution statement}
\textbf{Yunhong Deng}: Investigation, Software, Formal Analysis, Writing – original draft, Writing – review \& editing. 
\textbf{Sihong Shao}: Supervision, Conceptualization, Writing – review \& editing, Funding Acquisition. 
\textbf{Alex Mogilner}: Supervision, Conceptualization, Writing – review \& editing, Funding Acquisition. 
\textbf{Mingtao Xia}: 
Supervision, Project Administration, Conceptualization, Formal Analysis, Investigation, Writing – original draft, Writing – review \& editing.
\section*{Acknowledgement}
MX and AM are supported by NSF grants DMS 2052515 and DMS 2011544. SS is partially supported by the National Natural Science Foundation of China (Nos. 12325112, 12288101).

\section*{Declaration of competing interest}
The authors declare that they have no known competing financial interests or personal relationships that could have appeared to influence the work reported in this paper.

\section*{Data Availability}
No data was used for the research described in the article.

\appendix
\section{Proof of Theorem \ref{th:2-3}}\label{ap:2} In this section, we prove the existence of a solution $u \in X(0, T)$ to the model problem \eqref{eq:2-2}. We denote
\begin{equation}
    \mathscr{V}^{1} \coloneqq L^{2}\big([0, T]; H^{1}(\mathbb{R}^{d})\big) \text{ and } 
    \mathscr{V}^{-1} \coloneqq L^{2}\big([0, T]; H^{-1}(\mathbb{R}^{d})\big).
\end{equation}
$ \mathscr{V}^{1}$ and $ \mathscr{V}^{-1}$ are equipped with norms 
\begin{equation}
    \|u\|_{\mathscr{V}^{1}}^{2} \coloneqq \int_{0}^{T}\|u\|_{H^{1}}^{2}\text{d}t \text{ and } \|u\|_{\mathscr{V}^{-1}}^{2} \coloneqq \int_{0}^{T}\|u\|_{H^{-1}}^{2}\text{d}t.
\end{equation}

Fix $w \in X(0, T)$ and consider the following problem of finding $u\in X(0, T)$ such that
\begin{equation}\label{p:2-1}
    \begin{aligned}
        \big(\partial_{t}u,v\big) + a\big(u, v; t\big) &= \big(f(w; t), v\big),~(\boldsymbol{x}, t) \in\mathbb{R}^{d} \times [0, T],\\
        \big(u(x, 0), \Tilde{v}(x)\big) &= \big(u_{0}(x), \Tilde{v}(x)\big),\,\,
        \forall (v,\Tilde{v}) \in L^{2}\big([0, L], H^{1}(\mathbb{R}^{d})\big) \times H^{1}(\mathbb{R}^{d}).
    \end{aligned}
\end{equation}
Since $w \in X(0, T) \subseteq C([0, T]; L^{2}(\mathbb{R}))$, using \citep[Theorem 1, page 473]{Dautray1992}, we have
\begin{equation}\label{A1:8}
    f(w; t) \in C\big([0, T]; L^{2}(\mathbb{R}^{d})\big).
\end{equation}
Using \citep[Theorem 2, page 513]{Dautray1992}, we can prove that $\forall w\in X(0, T)$ there exists a $G(w)\in X(0, T)$ which solves the following parabolic equation
\begin{equation}\label{A1:1}
    \begin{aligned}
        \big(\partial_{t} G(w), v\big) + a\big(G(w), v; t\big) &= \big(f(w; t), v\big),\\
        \big(G(w)(\cdot, 0), \Tilde{v}\big) &= \big(u_{0}(\cdot), \Tilde{v}\big),\,\,
        \forall (v,\Tilde{v}) \in L^{2}\big([0, L], H^{1}(\mathbb{R}^{d})\big) \times H^{1}(\mathbb{R}^{d}).
    \end{aligned}
\end{equation}

The existence of a weak solution $u(\boldsymbol{x}, t)$ to the model problem Eq. \eqref{eq:2-2} can be established by proving the existence of a fixed point to the mapping $G$ such that $G(w) = w \in X(0, T)$. We shall prove the following lemma.
\begin{lemma}\label{A1:6}
For the mapping $G$ defined in Eq.~\eqref{A1:1}, there exists $u(\boldsymbol{x}, t) \in \mathscr{V}^{1}$ such that  $u = G(u)$.
\end{lemma}

\begin{proof}
The existence of the mapping $G$ is guaranteed by \citep[Theorem 2, page 513]{Dautray1992}. We define a sequence $\{u^{n}\}_{n = 0}^{\infty}$ in the Banach space $X(0, t)$ by $u^{n + 1} = G(u^{n})$ such that
\begin{equation}\label{eq:2-18}
    \begin{aligned}
        \big(\partial_{t} u^{n + 1}, v\big) + a\big(u^{n + 1}, v; t\big) &= \big(f(u^{n}; t), v\big),~
        \forall v \in L^{2}\big([0, T], H^{1}(\mathbb{R}^{d})\big),\,\,
        u^{n + 1}(\boldsymbol{x}, 0) &= u_{0}(\boldsymbol{x}).
    \end{aligned}
\end{equation}

%Before we prove the existence of a fixed point in $X(0, t)$, 
We first prove the existence of the fixed point $u\in\mathscr{V}^{1}$ such that $G(u)=u$. Next, we prove that for this fixed point $u$ we have $\partial_{t}u \in H^{-1}(\mathbb{R}^{d})$ and then show the convergence of $\{\partial_{t}u^{n}\}_{n = 0}^{\infty}$ in $\mathscr{V}^{-1}$. By shifting the index $n$ in Eq \eqref{eq:2-18}, we have
\begin{equation}\label{eq:2-19}
    \begin{aligned}
        \big(\partial_{t} u^{n}, v\big) + a\big(u^{n}, v; t\big) &= \big(f(u^{n - 1}; t), v\big),~
        \forall v \in L^{2}\big([0, T], H^{1}(\mathbb{R}^{d})\big),\\
        u^{n}(x, 0) &= u_{0}(x).
    \end{aligned}
\end{equation}
By subtracting Eq. \eqref{eq:2-19} from Eq.~\eqref{eq:2-18}, for $x^{n + 1} \coloneqq u^{n + 1} - u^{n}$, we have
\begin{equation}\label{p:2-4}
    \big(\partial_{t}x^{n + 1}, v\big) + a\big(x^{n + 1}, v; t\big) = \big(f(u^{n}; t) - f(u^{n - 1}; t), v\big), \forall v \in L^{2}\big([0, T], H^{1}(\mathbb{R}^{d})\big).
\end{equation}
We can take $
    v = x^{n + 1}$ in Eq. \eqref{p:2-4} to obtain
\begin{equation}\label{p:2-2}
    (\partial_{t}x^{n + 1}, x^{n + 1})+ a(x^{n + 1}, x^{n + 1}; t) = \big(f(u^{n}; t) - f(u^{n - 1}; t), x^{n + 1}\big).
\end{equation}
Additionally, by the coercity condition of the operators $a(u, v; t)$ and the Lipschitz continuity of $f(u; t)$, we have
\begin{equation}\label{A1:7}
\begin{aligned}
    c_{0}\|x^{n + 1}\|_{H^{1}}^{2} \le a(x^{n + 1}, x^{n + 1}; t),\,\,\big(f(u^{n}; t) - f(u^{n - 1}; t), x^{n + 1}\big) \le L\|x^{n}\|_{L^{2}} \|x^{n + 1}\|_{L^{2}}.
\end{aligned}
\end{equation}
Therefore, we have
\begin{equation}
    \begin{aligned}
        (\partial_{t}x^{n + 1}, x^{n + 1}) + c_{0}\|x^{n + 1}\|^{2}_{H^{1}} &\le (\partial_{t}x^{n + 1}, x^{n + 1}) + a(x^{n + 1}, x^{n + 1}; t)= \big(f(u^{n}; t) - f(u^{n - 1}; t), x^{n + 1}\big)\\
        &\le L\|x^{n}\|_{L^{2}} \|x^{n + 1}\|_{L^{2}}
%    \end{aligned}
%\end{equation}
%Furthermore, by using the Cauchy inequality, we have
%\begin{equation}
 %   \begin{aligned}
 %       (\partial_{t}x^{n + 1}, x^{n + 1}) + c_{0}\|x^{n + 1}\|^{2}_{H^{1}} &\le \frac{L^{2}}{2c_{0}}\|x^{n}\|_{L^{2}}^{2} + \frac{c_{0}}{2} \|x^{n + 1}\|_{L^{2}}^{2}\\
         \leq\frac{L^{2}}{2c_{0}}\|x^{n}\|_{H^{1}}^{2} + \frac{c_{0}}{2} \|x^{n + 1}\|_{H^{1}}^{2}.
    \end{aligned}
\end{equation}
Integrating the above inequality over time, we have
\begin{equation}
    \begin{aligned}
        \frac{1}{2}\|x^{n + 1}(\cdot, T)\|_{L^{2}}^{2} &- \frac{1}{2}\|x^{n + 1}(\cdot, 0)\|_{L^{2}}^{2} + c_{0}\int_{0}^{T}\|x^{n + 1}\|^{2}_{H^{1}}\text{d}t\\
        &\quad \le \frac{L^{2}}{2c_{0}} \int_{0}^{T} \|x^{n}\|_{H^{1}}^{2}\text{d}t + \frac{c_{0}}{2} \int_{0}^{T} \|x^{n + 1}\|_{H^{1}}^{2}\text{d}t.
    \end{aligned}
\end{equation}
Since $\|x^{n + 1}(\cdot, 0)\|_{L^{2}}^{2} = 0$, we can show that
\begin{equation}\label{p:2-3}
    \int_{0}^{T}\|x^{n + 1}(\cdot, t)\|_{H^{1}}^{2}\text{d}t \le \frac{L^{2}}{c_{0}^{2}}\int_{0}^{T}\|x^{n}(\cdot, t)\|_{H^{1}}^{2}\text{d}t.
\end{equation}
Thus, we have proved that
\begin{equation}\label{p:2-18}
    \|u^{n + 1} - u^{n}\|_{\mathscr{V}^{1}} \le \frac{L}{c_{0}} \|u^{n} - u^{n - 1}\|_{\mathscr{V}^{1}},
\end{equation}
Using the Banach contraction principle:
\begin{equation}\label{A1:3}
    L < c_{0} \Longrightarrow \exists!~ u \in \mathscr{V}^{1}:~ u = \lim_{n\to \infty} u^{n} \text{ such that } u = G(u).
\end{equation}
\end{proof}

Next, we shall prove the following result for $\partial_t u$:
\begin{lemma}\label{A1:5} For the same $u(\boldsymbol{x}, t)$ as defined in Eq. \eqref{A1:3}, we have
\begin{equation}
     \partial_{t}u \in \mathscr{V}^{-1}.
\end{equation}
Additionally, we show the strong convergence of the sequence $\partial_{t} u^{n}(\boldsymbol{x}, t)$ to $\partial_{t}u(\boldsymbol{x}, t)$ \textit{w.r.t.} the $H^{-1}$ norm:
\begin{equation}
    \lim_{n \to \infty} \|\partial_{t}u^{n} - \partial_{t}u\|_{H^{-1}} = 0.
\end{equation}
\end{lemma}
\begin{proof}
Since $u$ is the fixed point of the mapping $G$, we have
\begin{equation}\label{A1:4}
    \big(\partial_{t}u, v\big) = \big(f(u; t), v\big) - a\big(u, v, t\big),~\forall v \in H^{1}(\mathbb{R}^{d}),
\end{equation}
Using the Cauchy inequality, we have
\begin{equation}\label{p:2-15}
    \begin{aligned}
        \big(\partial_{t}u, v\big) = \big(f(u; t), v\big) - a\big(u, v; t\big) \le \big(\|f(u; t)\|_{L^{2}} + C_{0}\|u\|_{H^{1}}\big) \|v\|_{H^{1}}.
    \end{aligned}
\end{equation}
Therefore, $\partial_{t}u$ defines a continuous linear functional for $v \in H^{1}(\mathbb{R}^{d})$.
By the definition of the dual space $\mathscr{V}^{-1} = L^{2}\big([0, T]; H^{-1}(\mathbb{R}^{d})\big)$, we conclude that $\partial_{t}u \in \mathscr{V}^{-1}$.

Next, to show the convergence of the $\{\partial_{t}u^{n}\}_{n = 0}^{\infty}$ in the Banach space $\mathscr{V}^{-1}$, we prove that $\{\partial_{t}u^{n}\}_{n = 0}^{\infty}$ is a Cauchy sequence $\mathscr{V}^{-1}$. Denote $x^{n + 1} \coloneqq u^{n + 1} - u^{n}$. By Eq. \eqref{eq:2-18}, we have
\begin{equation}
    \begin{aligned}
        \|\partial_{t}x^{n + 1}\|_{H^{-1}} &= \sup_{v \in H^{1}(\mathbb{R})} \frac{\big(\partial_{t}x^{n + 1}, v\big)}{\|v\|_{H^{1}}}\\ 
        &= \sup_{v \in H^{1}(\mathbb{R})} \frac{\big(f(u^{n}) - f(u^{n - 1}), v\big) - a\big(x^{n + 1}, v; t\big)}{\|v\|_{H^{1}}}\le C_{0}\|x^{n + 1}\|_{H^{1}} + L\|x^{n}\|_{H^{1}}.
    \end{aligned}
\end{equation}

Integrating the above inequality over time, we can show that
\begin{equation}
        \int_{0}^{T}\|\partial_{t}x^{n + 1}\|_{H^{-1}}^{2}\text{d}t \le 2C_{0}^{2}\int_{0}^{T}\|x^{n + 1}\|_{H^{1}}^{2}\text{d}t + 2L^{2}\int_{0}^{T}\|x^{n}\|_{H^{1}}^{2}\text{d}t.
\end{equation}
In other words,
\begin{equation}
    \|\partial_{t}x^{n + 1}\|_{\mathscr{V}^{-1}}^{2} \le 2C_{0}^{2}\|x^{n + 1}\|_{\mathscr{V}^{1}}^{2} + 2L^{2}\|x^{n}\|_{\mathscr{V}^{1}}^{2}.
\end{equation}

Since we have proved that $\|x^{n + 1}\|_{\mathscr{V}^{1}} \le \|x^{n}\|_{\mathscr{V}^{1}}$ in Eq. \eqref{p:2-18}, we find
\begin{equation}\label{A1:9}
    \|\partial_{t}x^{n + 1}\|_{\mathscr{V}^{-1}}^{2} \le (2L^{2} + 2C_{0}^{2})\|x^{n}\|_{\mathscr{V}^{1}}^{2}.
\end{equation}
Furthermore, by leveraging the property of the Cauchy sequence $\{x^{n}\}_{n = 0}^{\infty}$ in $\mathscr{V}^{1}$, we have
\begin{equation}
    \|\partial_{t}u^{m} - \partial_{t}u^{n}\|_{\mathscr{V}^{-1}} \le \sum_{i = 1}^{m - n} \|\partial_{t}x^{n + i}\|_{\mathscr{V}^{-1}} ~\forall n < m.
\end{equation}
Therefore, by Eq. \eqref{A1:9}, we have
\begin{equation}
    \|\partial_{t}u^{m} - \partial_{t}u^{n}\|_{\mathscr{V}^{-1}} \le (2C_{0}^{2} + 2L^{2})^{1/2} \sum_{i = 0}^{m - n - 1} \|x^{n + i}\|_{\mathscr{V}^{1}}.
\end{equation}
Using the inequality~\eqref{p:2-3}, we have
\begin{equation}
    \begin{aligned}
        \|\partial_{t}u^{m} - \partial_{t}u^{n}\|_{\mathscr{V}^{-1}} &\le (2C_{0}^{2} + 2L^{2})^{1/2} \sum_{i = 0}^{m - n - 1} q^{i} \|x^{n}\|_{\mathscr{V}^{1}}\le (2C_{0}^{2} + 2L^{2})^{1/2} \frac{1}{1 - q} \|x^{n}\|_{\mathscr{V}^{1}}\\
        &\le (2C_{0}^{2} + 2L^{2})^{1/2} \frac{1}{1 - q} \|u^{n + 1} - u^{n}\|_{\mathscr{V}^{1}},
    \end{aligned}
\end{equation}
where $q = L/c_{0}$. Because $\|u^{n + 1} - u^{n}\|_{\mathscr{V}^{1}} \to 0$ as $n \to \infty$, we find
\begin{equation}
    \forall \varepsilon \in \mathbb{R},~\exists N \text{ such that } \forall m > n > N \Longrightarrow \|\partial_{t}u^{m} - \partial_{t}u^{n}\|_{\mathscr{V}^{-1}} \le \varepsilon,
\end{equation}
which implies that $\{\partial_{t}u^{n}\}_{n = 1}^{\infty}$ converges in $\mathscr{V}^{-1}$.
\end{proof}

Finally, by combining Lemma~\ref{A1:6} and Lemma~\ref{A1:5}, $\{u^{n}\}_{n = 1}^{\infty}$ converge to $u(\boldsymbol{x}, t)$ in the space $X(0, T)$, and
\begin{equation}
    u = G(u) \in X(0, T).
\end{equation}
Therefore, the solution to the equation \eqref{eq:2-2} exists and is unique.

\section{Proof of Proposition~\ref{prop3}}\label{ap:6}
In this section, we prove that the bilinear form defined in \eqref{eq:2-22}:
\begin{equation}
    a\big(u, v; t\big) \coloneqq \big(G \ast \nabla u, \nabla v\big) + \varepsilon\big((u, v) + (\nabla u, \nabla v)\big),
\end{equation}
satisfies the continuous condition and the coercive condition such that
\begin{equation}
\begin{aligned}
    a\big(u, v; t\big) \le C_{0}\|u\|_{H^{1}} \|v\|_{H^{1}}, \,\, c_{0}\|u\|_{H^{1}}^{2} \le a\big(u, u; t\big).
\end{aligned}
\end{equation}

To prove the continuity of $a(u, v; t)$, we use the following lemma \citep[Lemma 1.2.30, page 28]{Tuomas2016}.
\begin{lemma}[Young's inequality]
\label{youngs}
For $u \in L^{p}(\mathbb{R}^{d})$ and $G \in L^{1}(\mathbb{R}^{d})$, we have
\begin{equation}
    \|G \ast u\|_{L^{p}} \le \|G\|_{L^{1}} \|u\|_{L^{p}}.
\end{equation}
\end{lemma}

By the Cauchy-Schwartz inequality, we have
\begin{equation}\label{A2:1}
    \begin{aligned}
        a\big(u, v; t\big) &= \big(G\ast \nabla u, \nabla v\big) + \varepsilon\big(u, v\big) + \varepsilon \big(\nabla u, \nabla v\big)\\
        &\le \|G\ast \nabla u\|_{L^{2}} \|\nabla v\|_{L^{2}} + \varepsilon \|u\|_{L^{2}}\|v\|_{L^{2}} + \varepsilon \|\nabla u\|_{L^{2}}\|\nabla v\|_{L^{2}}
    \end{aligned}
\end{equation}
By the Young's inequality (Lemma~\ref{youngs}),
\begin{equation}\label{A2:2}
    \|G\ast \nabla u\|_{L^{2}} \le \|G\|_{L^{1}}\|\nabla u\|_{L^{2}}
\end{equation}
Therefore, combining Eqs. \eqref{A2:1} and \eqref{A2:2}, we have
\begin{equation}
    \begin{aligned}
        a\big(u, v; t\big) \le \big(\varepsilon + \|G\|_{L^{1}}\big) \|\nabla u\|_{L^{2}} \|\nabla v\|_{L^{2}} + \varepsilon \|u\|_{L^{2}}\|v\|_{L^{2}}\le \big(2\varepsilon + \|G\|_{L^{1}}\big) \|u\|_{H^{1}} \|v\|_{H^{1}}
    \end{aligned}
\end{equation}

The coercive condition on $a\big(u, v; t\big)$ can be proved by using the Fourier transform,
\begin{equation}\label{A2:3}
    a\big(u, u; t\big) = \left(\mathscr{F}(G)\mathscr{F}(\nabla u), \overline{\mathscr{F}(\nabla u)}\right) + \varepsilon \|u\|_{H^{1}}^{2}.
\end{equation}
Since we have already assume that $\mathscr{F}(G) \ge 0$, we have
\begin{equation}
     \left(\mathscr{F}(G)\mathscr{F}(\nabla u), \overline{\mathscr{F}(\nabla u)}\right) = \int_{\mathbb{R}}\mathscr{F}(G)|\mathscr{F}(\nabla u)|^{2}dx \ge 0.
\end{equation}
Therefore, by Eq. \eqref{A2:3}, we have
\begin{equation}
    a\big(u, u; t\big) \ge \left(\mathscr{F}(G)\mathscr{F}(\nabla u), \overline{\mathscr{F}(\nabla u)}\right) + \varepsilon \|u\|_{H^{1}}^{2} \ge \varepsilon \|u\|_{H^{1}}^{2}.
\end{equation}

\section{Proof of Lemma \ref{th:5-1}}\label{ap:1}
In this section, we first prove the inverse inequalities~\eqref{eq:1-1} and~\eqref{eq:1-12} for 1D mapped Jacobi spectral expansions. Next, we extend the inverse inequalities to $d$-dimensional sparse mapped Jacobi spectral expansions defined in a hyperbolic cross space.

\subsection{Inverse inequalities for 1-dimensional mapped Jacobi spectral expansions}
1-D mapped Jacobi spectral expansion $U_{N}^{\beta, x_{0}}$ can be written as
\begin{equation}
    \begin{aligned}
        U_{N}^{\beta, x_{0}}(x) &= p_{N}(\xi) \mu_{\alpha_{1}, \alpha_{2}}(\xi),
    \end{aligned}
\end{equation}
Herein, $p_{N}(\xi)$ is a polynomial spanned by the Jacobi polynomials $\{j_{n}^{\alpha_{1}, \alpha_{2}}(\xi)\}_{n = 0}^{N}$,
\begin{equation}
    p_{N}(\xi) \coloneqq \sum_{n = 0}^{N}u_{n}^{\beta, x_{0}} j_{n}^{\alpha_{1}, \alpha_{2}}(\xi) \text{ where } \xi \coloneqq h_{\beta, r}(x - x_{0}),
\end{equation}
and $\mu_{\alpha_{1}, \alpha_{2}}(\xi)\coloneqq \sqrt{\omega^{\alpha_1, \alpha_2}(\xi)\frac{\text{d} \xi}{\text{d} x}}$ is the modified weight function.

First, we prove Eq. \eqref{eq:1-1} for mapped Jacobi spectral expansions, \textit{i.e.},
\begin{equation}\label{A3:1}
    \big\|\partial_{x}U_{N}^{\beta, x_{0}}\big\|_{L^{2}} \le \beta^{3/2}N_{\alpha, r}^{1/2} \big\|U_{N}^{\beta, x_{0}}\big\|_{L^{2}} 
\end{equation}

We define an auxiliary function (introduced in \cite{Tang2019})
\begin{equation}
    g(\xi) \coloneqq \beta (1 - \xi^{2}) \frac{\text{d}x}{\text{d}\xi} = \frac{1}{(1 - \xi^{2})^{r/2}}.
\end{equation}
It can be verified that $1 \le g(\xi)$. Thus, the following inequality holds true for $U_{N}^{\beta, x_{0}}$:
\begin{equation}\label{A3:2}
    \big\|\partial_{x} U_{N}^{\beta, x_{0}}\big\|^{2}_{L^{2}} \le \big\|g^{1/2} \partial_{x} U_{N}^{\beta, x_{0}}\big\|^{2}_{L^{2}} = \int_{\mathbb{R}} g(\xi)\big(\partial_{x}(p_{N}\mu_{\alpha_{1}, \alpha_{2}})\big)^{2}\d x.
\end{equation}

A direct computation shows that
\begin{equation}
    \begin{aligned}
        \partial_{x}\big(p_{N}(\xi) \mu_{\alpha_{1}, \alpha_{2}}(\xi)\big) &= \partial_{x}p_{N}(\xi) \mu_{\alpha_{1}, \alpha_{2}}(\xi) + p_{N}(\xi) \partial_{x}\mu_{\alpha_{1}, \alpha_{2}}(\xi)\\
        &= p_{N}^{\prime}\mu_{\alpha_{1}, \alpha_{2}} \frac{\text{d}\xi}{\text{d}x}  + p_{N}\mu_{\alpha_{1}, \alpha_{2}}^{\prime}\frac{\text{d}\xi}{\text{d}x}.
    \end{aligned}
\end{equation}
Therefore, from the Cauchy inequality, we have
\begin{equation}\label{p:1-1}
    \big\|\partial_{x}U_{N}^{\beta, x_{0}}(x)\big\|^{2}_{L^{2}} \le 2\int_{-1}^{1}g(\xi)\Big(p_{N}^{\prime}\mu_{\alpha_{1}, \alpha_{2}}\frac{\text{d}\xi}{\text{d}x}\Big)^{2}\text{d}\xi + 2\int_{-1}^{1}g(\xi)\Big(p_{N}\mu_{\alpha_{1}, \alpha_{2}}^{\prime}\frac{\text{d}\xi}{\text{d}x}\Big)^{2}\text{d}\xi
\end{equation}

To prove the inverse inequality Eq. \eqref{eq:1-1}, it is sufficient to derive a bound for the RHS of Eq. \eqref{p:1-1}. For the first term on RHS of Eq. \eqref{p:1-1}, we have
\begin{equation}\label{p:2-5}
    \begin{aligned}
        g(\xi)\Big(p_{N}^{\prime} \mu_{\alpha_{1}, \alpha_{2}}\frac{\text{d}\xi}{\text{d}x}\Big)^{2} &= \beta (1 - \xi^{2}) \frac{\text{d}x}{\text{d}\xi} (p_{N}^{\prime})^{2} \mu_{\alpha_{1}, \alpha_{2}}^{2} \Big(\frac{\text{d}\xi}{\text{d}x}\Big)^{2}\\
        & = \beta (1 - \xi^{2}) \frac{\text{d}x}{\text{d}\xi} (p_{N}^{\prime})^{2} w^{\alpha_{1}, \alpha_{2}} \Big(\frac{\text{d}\xi}{\text{d}x}\Big)^{3}
%        & = \beta (p_{N}^{\prime})^{2} (1 - \xi^{2}) w^{\alpha_{1}, \alpha_{2}} \Big(\frac{\text{d}\xi}{\text{d}x}\Big)^{2}\\
        = \beta (p_{N}^{\prime})^{2} w^{\alpha_{1} + 1, \alpha_{2} + 1}\Big(\frac{\text{d}\xi}{\text{d}x}\Big)^{2}.
    \end{aligned}
\end{equation}

Integrating Eq. \eqref{p:2-5} over $\xi\in[1, 1]$, we have
\begin{equation}\label{p:2-6}
    \begin{aligned}
        \int_{-1}^{1}g(\xi)\Big(p_{N}^{\prime}\mu_{\alpha_{1}, \alpha_{2}} \frac{\text{d}\xi}{\text{d}x} \Big)^{2}\text{d}\xi & = \int_{-1}^{1} \beta (p_{N}^{\prime})^{2} w^{\alpha_{1} + 1, \alpha_{2} + 1} \Big(\frac{\text{d}\xi}{\text{d}x}\Big)^{2}\text{d}\xi\\
        &= \beta^{3} \int_{-1}^{1} (p_{N}^{\prime})^{2} w^{\alpha_{1} + 1, \alpha_{2} + 1}(1 - \xi^{2})^{2 + r} \text{d}\xi\le \beta^{3}\int_{-1}^{1}(p_{N}^{\prime})^{2} w^{\alpha_{1} + 1, \alpha_{2} + 1}\text{d}\xi
    \end{aligned}
\end{equation}
Using the inverse inequalities of Jacobi polynomials \citep[Theorem 3.31]{Shen2011}, we have
\begin{equation}\label{A3:3}
    \begin{aligned}
        &\int_{-1}^{1}g(\xi) \Big(p_{N}^{\prime}\mu_{\alpha_{1}, \alpha_{2}} \frac{\text{d}\xi}{\text{d}x} \Big)^{2}\text{d}\xi \le \int_{-1}^{1}(p_{N}^{\prime})^{2} w^{\alpha_{1} + 1, \alpha_{2} + 1}\text{d}\xi\\
        &\hspace{2cm}\le \beta^{3}\lambda_{N}^{\alpha_{1}, \alpha_{2}}\int_{-1}^{1} p_{N}^{2} w^{\alpha_{1}, \alpha_{2}}\text{d}\xi= \beta^{3}\lambda_{N}^{\alpha_{1}, \alpha_{2}} \int_{\mathbb{R}} (U_{N}^{\beta, x_{0}})^{2} \text{d}x,
    \end{aligned}
\end{equation}
where $\lambda_{N}^{\alpha_{1}, \alpha_{2}} \coloneqq N(N + \alpha_{1} + \alpha_{2} + 1)$.

Furthermore, for the second term on the RHS of Eq. \eqref{p:1-1}, we have
\begin{equation}\label{p:2-9}
    \begin{aligned}
        g(\xi)\Big(p_{N} \mu_{\alpha_{1}, \alpha_{2}}^{\prime} \frac{\text{d}\xi}{\text{d}x}\Big)^{2} &= g(\xi)\Big(p_{N} \frac{\mu_{\alpha_{1}, \alpha_{2}}^{\prime}}{\mu_{\alpha_{1}, \alpha_{2}}}\mu_{\alpha_{1}, \alpha_{2}} \frac{\text{d}\xi}{\text{d}x}\Big)^{2}\\
        &= g(\xi)\Big(\frac{\mu_{\alpha_{1}, \alpha_{2}}^{\prime}}{\mu_{\alpha_{1}, \alpha_{2}}} \Big(\frac{\text{d}\xi}{dx}\Big)^{3/2}\Big) ^{2} \Big(p_{N}\mu_{\alpha_{1}, \alpha_{2}} \Big(\frac{\text{d}x}{\text{d}\xi}\Big)^{1/2}\Big)^{2}
    \end{aligned}
\end{equation}
Additionally, a direct computation shows that
\begin{equation}\label{p:2-7}
    \begin{aligned}
        g(\xi)^{1/2}\frac{\mu_{\alpha_{1}, \alpha_{2}}^{\prime}}{\mu_{\alpha_{1}, \alpha_{2}}} &\Big(\frac{\text{d}\xi}{\text{d}x}\Big)^{3/2} = \Big(\beta (1 - \xi^{2}) \frac{\text{d}x}{\text{d}\xi} \Big)^{1/2} \Big(\frac{\mu_{\alpha_{1}, \alpha_{2}}^{\prime}}{\mu_{\alpha_{1}, \alpha_{2}}}\Big) \Big(\frac{\text{d}\xi}{\text{d}x}\Big)^{3/2}\\
        &= \beta^{1/2} (1 - \xi^{2})^{1/2} \Big(\frac{\mu_{\alpha_{1}, \alpha_{2}}^{\prime}}{\mu_{\alpha_{1}, \alpha_{2}}}\Big) \frac{\text{d}\xi}{\text{d}x}\\
        &= \beta^{3/2}(1 - \xi^{2})^{(3 + r)/2}\frac{(w^{\alpha_{1}, \alpha_{2}})^{\prime} \frac{\text{d}\xi}{\text{d}x} - \beta w^{\alpha_{1}, \alpha_{2}} (2 + r) \xi  (1 - \xi^{2})^{r/2}}{2 w^{\alpha_{1}, \alpha_{2}}\frac{\text{d}\xi}{\text{d}x}}\\
        &= \beta^{3/2}(1 - \xi^{2})^{(3 + r)/2}\Big(\frac{(w^{\alpha_{1}, \alpha_{2}})^{\prime}}{2 w^{\alpha_{1}, \alpha_{2}}} - \frac{\beta (2 + r) \xi (1 - \xi^{2})^{r/2}}{2\frac{\text{d}\xi}{\text{d}x}}\Big)\\
        %&= \beta^{3/2}(1 - \xi^{2})^{(3 + r)/2}\Big(\frac{(w^{\alpha_{1}, \alpha_{2}})^{\prime}}{2 w^{\alpha_{1}, \alpha_{2}}} - \frac{(1 + r/2) \xi (1 - \xi^{2})^{r/2}}{(1 - \xi^{2})^{1 + r/2}}\Big)\\
        &= \frac{1}{2}\beta^{3/2}(1 - \xi^{2})^{(1 + r)/2} \big(-\alpha_{1}(1 + \xi) + \alpha_{2}(1 - \xi) + \xi(2 + r)\big)\\
        &\le \beta^{3/2} (1 + \alpha_{1} + \alpha_{2} + r/2).
    \end{aligned}
\end{equation}
Therefore, we have
\begin{equation}\label{A3:5}
    \begin{aligned}
        g(\xi)\Big(p_{N}(\xi) \mu_{\alpha_{1}, \alpha_{2}}^{\prime}\frac{\text{d}\xi}{\text{d}x}\Big)^{2} &= g(\xi)\Big(\frac{\mu_{\alpha_{1}, \alpha_{2}}^{\prime}}{\mu_{\alpha_{1}, \alpha_{2}}} \Big(\frac{\text{d}\xi}{dx}\Big)^{3/2}\Big)^{2} \Big(p_{N}\mu_{\alpha_{1}, \alpha_{2}} \Big(\frac{\text{d}x}{\text{d}\xi}\Big)^{1/2}\Big)^{2}\\ &\le \beta^{3} (1 + \alpha_{1} + \alpha_{2} + r/2)^{2} \Big(p_{N}\mu_{\alpha_{1}, \alpha_{2}} \Big(\frac{\text{d}x}{\text{d}\xi}\Big)^{1/2}\Big)^{2}.
    \end{aligned}
\end{equation}
By integrating Eq. \eqref{A3:5} over $\xi\in[-1, 1]$, we have
\begin{equation}\label{A3:7}
    \int_{-1}^{1}g(\xi)\Big(p_{N}(\xi) \mu_{\alpha_{1}, \alpha_{2}}^{\prime}\frac{\text{d}\xi}{\text{d}x}\Big)^{2}\text{d}\xi \le \beta^{3} (1 + \alpha_{1} + \alpha_{2} + r/2)^{2} \|U_{N}^{\beta, x_{0}}(x)\|^{2}_{L^{2}}
\end{equation}

Finally, plugging in the estimation of the two terms on the RHS of Eq. \eqref{p:1-1} using Eqs. \eqref{A3:3} and \eqref{A3:7}, we have
\begin{equation}
    \|\partial_{x}U_{N}^{\beta, x_{0}}(x)\|^{2}_{L^{2}} \le \beta^{3} N_{\alpha, r} \|U_{N}^{\beta, x_{0}}(x)\|^{2}_{L^{2}},
\end{equation}
where $N_{\alpha, r} = 2N(N + \alpha_{1} + \alpha_{2} + 1) + 2(1 + \alpha_{1} + \alpha_{2} + r/2)^{2}$.

\vspace{0.1in}
Next, we prove Eq. \eqref{eq:1-12} for mapped Jacobi spectral expansions, \textit{i.e.},
\begin{equation}\label{A3:8}
    \big\|x\partial_{x}U_{N}^{\beta, x_{0}}(x)\big\|_{L^{2}} \le \beta^{1/2}N_{\alpha, r}^{1/2} \big\|U_{N}^{\beta, x_{0}}(x)\big\|_{L^{2}},
\end{equation}
we additionally assume that $r \le 1$. When $r = 1$, we have
\begin{equation}
    \xi =h_{\beta, 1}(x)= \frac{\beta x}{\sqrt{1 + \beta^{2} x^{2}}}.
\end{equation}
Also, we have $\big(1 - \xi^{2}\big)^{-1} = 1 + (\beta x)^{2} \ge (\beta x)^2$.
When $r \le 1$, we have $h_{\beta, r}^{2}(x) \ge h_{\beta, 1}^{2}(x) \quad, \forall x \in \mathbb{R}$. Therefore, we can show that
\begin{equation}\label{p:2-8}
    \beta^{2}x^{2} \le \big(1 - h_{\beta, r}(x)^{2}\big)^{-1},~\forall r \le 1.
\end{equation}

Similar to Eq. \eqref{p:1-1}, we have
\begin{equation}\label{A3:6}
    \|x\partial_{x} U_{N}^{\beta, x_{0}}\|^{2}_{L^{2}} \le 2\int_{-1}^{1}x^{2}g(\xi)\Big(p_{N}^{\prime}\mu_{\alpha_{1}, \alpha_{2}}\frac{\text{d}\xi}{\text{d}x} \Big)^{2} + x^{2}g(\xi)\Big(p_{N}\mu_{\alpha_{1}, \alpha_{2}}^{\prime} \frac{\text{d}\xi}{\text{d}x}\Big)^{2}\text{d}\xi.
\end{equation}

Furthermore, the upper bound for the first term on the RHS of Eq. \eqref{A3:6} is similar to Eq. \eqref{p:2-5}:
\begin{equation}\label{p:2-11}
    \begin{aligned}
        \int_{-1}^{1}x^{2} g(\xi) \Big(p_{N}^{\prime}\mu_{\alpha_{1}, \alpha_{2}} & \frac{\text{d}\xi}{\text{d}x}\Big)^{2}\text{d}\xi = \beta \int_{-1}^{1} x^{2} (p_{N}^{\prime})^{2} w^{\alpha_{1} + 1, \alpha_{2} + 1}\left(\frac{\text{d}\xi}{\text{d}x}\right)^{2}\text{d}\xi\\
        &= \beta^{3} \int_{-1}^{1} x^{2} (p_{N}^{\prime})^{2} w^{\alpha_{1} + 1, \alpha_{2} + 1}(1 - \xi^{2})^{2 + r} \text{d}\xi\\
        &\le \beta \int_{-1}^{1}  (p_{N}^{\prime})^{2} w^{\alpha_{1} + 1, \alpha_{2} + 1} (1 - \xi^{2})^{1 + r} \text{d}\xi\\
        &\le \beta\int_{-1}^{1}(p_{N}(\xi)^{\prime})^{2} w^{\alpha_{1} + 1, \alpha_{2} + 1}\text{d}\xi\\
        &\leq \beta \lambda_{N}^{\alpha_{1}, \alpha_{2}} \int_{\mathbb{R}} (U_{N,x_{0}}^{\beta})^{2} \text{d}x.
    \end{aligned}
\end{equation}

Additionally, the upper bound for the second term on the RHS of Eq. \eqref{A3:6} is similar to Eq. \eqref{p:2-9}:
\begin{equation}
    \begin{aligned}
        x^{2}g(x)\Big(\frac{\mu_{\alpha_{1}, \alpha_{2}}^{\prime}}{\mu_{\alpha_{1}, \alpha_{2}}}\Big)^{2} \Big(\frac{\text{d}\xi}{\text{d}x}\Big)^{3} &= \frac{x^{2} \beta^{3}}{4}(1 - \xi^{2})^{1 + r} \big(\xi (2 + r) -(1 + \xi) \alpha_{1} + (1 - \xi) \alpha_{2}\big)^{2}\\
        &\le \frac{1}{4} \beta (1 - \xi^{2})^{r} \big(-(1 + \xi) \alpha_{1} + (1 - \xi) \alpha_{2} + \xi (2 + r)\big)^{2}\\
        &\le \beta (1 + \alpha_{1} + \alpha_{2} + r/2)^{2}.
    \end{aligned}
\end{equation}
Therefore, we have
\begin{equation}\label{p:2-12}
    \begin{aligned}
        \int_{-1}^{1} x^{2}g(\xi)\Big(p_{N}\mu_{\alpha_{1}, \alpha_{2}}^{\prime} \frac{\text{d}\xi}{\text{d}x}\Big)^{2}\text{d}\xi &= \int_{-1}^{1}x^{2}g(\xi)\Big(\frac{\mu_{\alpha_{1}, \alpha_{2}}^{\prime}}{\mu_{\alpha_{1}, \alpha_{2}}}\Big)^{2} \Big(\frac{\text{d}\xi}{\text{d}x}\Big)^{3} \Big(p_{N}\mu_{\alpha_{1}, \alpha_{2}} \Big(\frac{\text{d}x}{\text{d}\xi}\Big)^{1/2}\Big)^{2} \text{d}\xi\\
        &\le \beta (1 + \alpha_{1} + \alpha_{2} + r/2)^{2} \int_{-1}^{1} \Big(p_{N}\mu_{\alpha_{1}, \alpha_{2}} \Big(\frac{\text{d}x}{\text{d}\xi}\Big)^{1/2}\Big)^{2} \text{d}\xi\\
        &= \beta (1 + \alpha_{1} + \alpha_{2} + r/2)^{2} \|U_{N}^{\beta, x_{0}}\|_{L^{2}}^{2}
    \end{aligned}
\end{equation}

Finally, by combining Eqs. \eqref{p:2-11} and \eqref{p:2-12}, we are able to show that
\begin{equation}
    \|x\partial_{x}U_{N} ^{\beta, x_{0}}(x)\|^{2}_{L^{2}} \le \beta N_{\alpha, r} \|U_{N}^{\beta, x_{0}}(x)\|^{2}_{L^{2}}.
\end{equation}

\subsection{Inverse inequalities for $d$-dimensional mapped Jacobi spectral expansions}
Now, we extend the inverse inequalities Eqs. \eqref{A3:1} and \eqref{A3:8} to a $d$-dimensional ($d>1$) mapped Jacobi spectral expansions defined in a hyperbolic cross space
\begin{equation}
    U_{N, \gamma}^{\boldsymbol{\beta}, \boldsymbol{x}_{0}}(\boldsymbol{x}) = \sum_{\boldsymbol{n} \in \Upsilon_{N, \gamma}} u_{\boldsymbol{n}}^{\boldsymbol{\beta}, \boldsymbol{x}_{0}}\mathcal{J}_{\boldsymbol{n}}^{\boldsymbol{\beta}, \boldsymbol{x}_{0}}(\boldsymbol{x})
\end{equation}

Because the hyperbolic cross index set $\Upsilon_{N, \gamma}$ is a subset of $\mathbb{N}^{d}$, we can rewrite
\begin{equation}
    U_{N, \gamma}^{\boldsymbol{\beta}, \boldsymbol{x}_{0}}(\boldsymbol{x}) = \sum_{\boldsymbol{n} \in \mathbb{N}^{d}} u_{\boldsymbol{n}}^{\boldsymbol{\beta}, \boldsymbol{x}_{0}}\mathcal{J}_{\boldsymbol{n}}^{\boldsymbol{\beta}, \boldsymbol{x}_{0}}(\boldsymbol{x}) \text{ where } u_{\boldsymbol{n}}^{\boldsymbol{\beta}, \boldsymbol{x}_{0}} = 0,~\forall \boldsymbol{n} \in \mathbb{N}^{d} - \Upsilon_{N, \gamma}
\end{equation}
A direct computation shows that
\begin{equation}
    \big\|\partial_{x_{i}}U_{N, \gamma}^{\boldsymbol{\beta}, \boldsymbol{x}_{0}}(\boldsymbol{x})\big\|_{L^{2}}^{2} = \sum_{\boldsymbol{m} \in \mathbb{N}^{d}}\sum_{\boldsymbol{n} \in \mathbb{N}^{d}} u_{\boldsymbol{m}}^{\boldsymbol{\beta}, \boldsymbol{x}_{0}} u_{\boldsymbol{n}}^{\boldsymbol{\beta}, \boldsymbol{x}_{0}}\big(\partial_{x_{i}}\mathcal{J}_{\boldsymbol{m}}^{\boldsymbol{\beta}, \boldsymbol{x}_{0}}, \partial_{x_{i}}\mathcal{J}_{\boldsymbol{n}}^{\boldsymbol{\beta}, \boldsymbol{x}_{0}}\big).
\end{equation}

Tensorial mapped Jacobi basis functions form an orthonormal basis of $L^{2}(\mathbb{R}^{d})$:
\begin{equation}
    \big(\partial_{x_{i}} \mathcal{J}_{\boldsymbol{m}}^{\boldsymbol{\beta}, \boldsymbol{x}_{0}}, \partial_{x_{i}}\mathcal{J}_{\boldsymbol{n}}^{\boldsymbol{\beta}, \boldsymbol{x}_{0}}\big) = \prod_{j\neq i, j=1}^d\delta_{m_{j}, n_{j}}\times\big(\partial_{x_{i}} \mathcal{J}_{m_{i}}^{\beta_{i}, {x_{0}}_{i}}, \partial_{x_{i}}\mathcal{J}_{n_{i}}^{\boldsymbol{\beta}, \boldsymbol{x}_{0}}\big),
\end{equation}
where $\delta_{i, j}=1$ if $i=j$ and $\delta_{i, j}=0$ if $i\neq j$. Therefore, we have
\begin{equation}
    \begin{aligned}
        \big\|\partial_{x_{i}}U_{N, \gamma}^{\boldsymbol{\beta}, \boldsymbol{x}_{0}}(\boldsymbol{x})\big\|_{L^{2}}^{2} & = \sum_{\boldsymbol{m} \in \mathbb{N}^{d}}\sum_{\boldsymbol{n} \in \mathbb{N}^{d}}u_{\boldsymbol{m}}^{\boldsymbol{\beta}, \boldsymbol{x}_{0}} u_{\boldsymbol{n}}^{\boldsymbol{\beta}, \boldsymbol{x}_{0}}\big(\partial_{x_{i}}\mathcal{J}_{\boldsymbol{m}}^{\boldsymbol{\beta}, \boldsymbol{x}_{0}}, \partial_{x_{i}}\mathcal{J}_{\boldsymbol{n}}^{\boldsymbol{\beta}, \boldsymbol{x}_{0}}\big)\\
        &= \sum_{\boldsymbol{n}\in \mathbb{N}^{d}}\sum_{m_{i} = 1}^{N} u_{n_{1}, \cdots, m_{i},\cdots, n_{d}}^{\boldsymbol{\beta}, \boldsymbol{x}_{0}} u_{\boldsymbol{n}}^{\boldsymbol{\beta}, \boldsymbol{x}_{0}}\big(\partial_{x_{i}}\mathcal{J}_{n_{1}, \cdots, m_{i},\cdots, n_{d}}^{\boldsymbol{\beta}, \boldsymbol{x}_{0}}, \partial_{x_{i}}\mathcal{J}_{\boldsymbol{n}}^{\boldsymbol{\beta}, \boldsymbol{x}_{0}}\big)\\
        &= \sum_{\boldsymbol{n}\in \mathbb{N}^{d}}\sum_{m_{i} = 1}^{N} u_{n_{1}, \cdots, m_{i},\cdots, n_{d}}^{\boldsymbol{\beta}, \boldsymbol{x}_{0}} u_{\boldsymbol{n}}^{\boldsymbol{\beta}, \boldsymbol{x}_{0}}\big(\partial_{x_{i}} \mathcal{J}_{m_{i}}^{\beta_{i}, {x_{0}}_{i}}, \partial_{x_{i}}\mathcal{J}_{\boldsymbol{n}}^{\boldsymbol{\beta}, \boldsymbol{x}_{0}}\big)
    \end{aligned}
\end{equation}
By using the inverse inequality Eq. \eqref{A3:1}, we can prove that
\begin{equation}
    \begin{aligned}
        \big\|\partial_{x_{i}}U_{N, \gamma}^{\boldsymbol{\beta}, \boldsymbol{x}_{0}}(\boldsymbol{x})\big\|_{L^{2}}^{2} &=\sum_{\boldsymbol{n}\in \mathbb{N}^{d}}\sum_{m_{i} = 1}^{N} u_{n_{1}, \cdots, m_{i},\cdots, n_{d}}^{\boldsymbol{\beta}, \boldsymbol{x}_{0}} u_{\boldsymbol{n}}^{\boldsymbol{\beta}, \boldsymbol{x}_{0}}\big(\partial_{x_{i}} \mathcal{J}_{m_{i}}^{\beta_{i}, {x_{0}}_{i}}, \partial_{x_{i}}\mathcal{J}_{n_{i}}^{\boldsymbol{\beta}, \boldsymbol{x}_{0}}\big)\\
        &\le \beta_{i}^{3} N_{\alpha, r} \sum_{\boldsymbol{n}\in \mathbb{N}^{d}}\sum_{m_{i} = 1}^{N} u_{n_{1}, \cdots, m_{i},\cdots, n_{d}}^{\boldsymbol{\beta}, \boldsymbol{x}_{0}} u_{\boldsymbol{n}}^{\boldsymbol{\beta}, \boldsymbol{x}_{0}} \delta_{m_{i}, n_{i}}\\
        &= \beta_{i}^{3} N_{\alpha, r} \sum_{\boldsymbol{n}\in \mathbb{N}^{d}}\big(u_{\boldsymbol{n}}^{\boldsymbol{\beta}, \boldsymbol{x}_{0}}\big)^{2}\\
        &= \beta_{i}^{3} N_{\alpha, r} \big\|U_{N, \gamma}^{\boldsymbol{\beta}, \boldsymbol{x}_{0}}(\boldsymbol{x})\big\|_{L^{2}}^{2}
    \end{aligned}
\end{equation}

Similarly, by using Eq. \eqref{A3:8}, we can show that
\begin{equation}
    \begin{aligned}
        \big\|x_{i}\partial_{x_{i}}U_{N, \gamma}^{\boldsymbol{\beta}, \boldsymbol{x}_{0}}(\boldsymbol{x})\big\|_{L^{2}}^{2} &= \sum_{\boldsymbol{n}\in \mathbb{N}^{d}}\sum_{m_{i} = 1}^{N} u_{n_{1}, \cdots, m_{i},\cdots, n_{d}}^{\boldsymbol{\beta}, \boldsymbol{x}_{0}} u_{\boldsymbol{n}}^{\boldsymbol{\beta}, \boldsymbol{x}_{0}}\big(x_{i}\partial_{x_{i}} \mathcal{J}_{m_{i}}^{\beta_{i}, {x_{0}}_{i}}, x_{i}\partial_{x_{i}}\mathcal{J}_{n_{i}}^{\boldsymbol{\beta}, \boldsymbol{x}_{0}}\big)\\
        &\le \beta_{i} N_{\alpha, r} \big\|U_{N, \gamma}^{\boldsymbol{\beta}, \boldsymbol{x}_{0}}(\boldsymbol{x})\big\|_{L^{2}}^{2},
    \end{aligned}
\end{equation}
which proves Eq. \eqref{eq:1-12} for sparse mapped Jacobi spectral expansions defined in a hyperbolic cross space when $\boldsymbol{x}\in\mathbb{R}^d$.

\section{Proof of Theorem \ref{th:6-1}}\label{ap:3}
In this section, we prove the error bound Eq. \eqref{eq:2-20} for the mapped Jacobi spectral approximation in Eq.~\eqref{eq:4-23}.

\subsection{Properties of the operator $a\big(u, v; t\big)$ in Eq.~\eqref{eq:1-4}}
First, we define the following bilinear form:
\begin{equation}\label{p:1-16}
    B(u, \boldsymbol{v}; s) \coloneqq \int_{t_{\ell}}^{s}\big(\partial_{t}u, v\big) + a\big(u, v; t\big)\textnormal{d}t + \big(u(\cdot, t_{\ell}), \Tilde{v}\big),
\end{equation}
where $a\big(u, v; t\big)$ is the symmetric bilinear operator in Eq. \eqref{eq:1-4}. We can show that the bilinear form $B(u,\boldsymbol{v}; s)$ defined in Eq. \eqref{p:1-16} satisfies the following continuity and inf-sup conditions \cite{Ueda2018}.
\begin{lemma}\label{B1:3}
There exist two constants $B_{0}$ and $b_{0}$ such that
\begin{equation}
    \begin{aligned}
        B(u, \boldsymbol{v}; s) \le B_{0} \|u\|_{X(t_{\ell}, s)}\|\boldsymbol{v}\|_{Y(t_{\ell}, s)},\,\,
        b_{0} \le \inf_{u \neq 0} \sup_{\boldsymbol{v} \neq 0} \frac{B(u, \boldsymbol{v}; s)}{\|u\|_{X(t_{\ell}, s)} \|\boldsymbol{v}\|_{Y(t_{\ell}, s)}},
    \end{aligned}
\end{equation}
\end{lemma}

\begin{proof}
Since 
\begin{equation}
    a(u, v; t) \le C_{0}\|u\|_{H^{1}} \|v\|_{H^{1}},
\end{equation}
we can prove that
\begin{equation}
    \begin{aligned}
        B(u, \boldsymbol{v}; s) &\le \int_{t_{\ell}}^{s}\big(\partial_{t}u, v\big) + C_{0}\|u\|_{H^{1}} \|v\|_{H^{1}}\textnormal{d}t + \big(u(\cdot, t_{\ell}), \Tilde{v}\big)\\
        &\le \int_{t_{\ell}}^{s}\|\partial_{t}u\|_{H^{-1}} \|v\|_{H^{1}} + C_{0}\|u\|_{H^{1}} \|v\|_{H^{1}}\textnormal{d}t + \|u(\cdot, t_{\ell})\|_{L^{2}} \|\Tilde{v}\|_{L^{2}},
    \end{aligned}
\end{equation}
By using the Cauchy-Schwartz inequality, we have
\begin{equation}\label{ap:2-1}
    B(u, \boldsymbol{v}; s) \le B_{0} \|u\|_{X(t_{\ell}, s)}\|\boldsymbol{v}\|_{Y(t_{\ell}, s)} \text{ where } B_{0} \coloneqq \sqrt{2 + 2C_{0}^{2}}.
\end{equation}

Additionally, since
\begin{equation}
    a\big(u, v; t\big) \le C_{0}\|u\|_{H^{1}} \|v\|_{H^{1}}, \,\,
    c_{0}\|u\|_{H^{1}}^{2} \le a\big(u, u; t\big),
\end{equation}
by the Lax-Milgram theorem, $\forall w \in H^{-1}(\mathbb{R}^{d})$, there exists a unique $u \in H^{1}(\mathbb{R}^{d})$ such that
\begin{equation}
    a(u, v; t) = (w, v),~ \forall v\in H^{1}(\mathbb{R}^{d}).
\end{equation}
Therefore, the bilinear form $a(u, v; t)$ defines an linear operator
\begin{equation}
    A^{-1}(t): H^{-1}(\mathbb{R}^{d}) \to H^{1}(\mathbb{R}^{d}),
\end{equation}
such that
\begin{equation}\label{ap:2-18}
    a\big(A^{-1}(t)u, v; t\big) = (u, v).
\end{equation}
Furthermore, it has been shown in \citep[Lemma 5.1]{Ueda2018} that the linear operator $A^{-1}(t)$ satisfies
\begin{equation}\label{ap:2-4}
        \big\|A^{-1}(t)u\big\|_{H^{1}} \le \frac{1}{c_{0}} \|u\|_{H^{-1}}, \,\,
        \frac{1}{C_{0}} \|u\|_{H^{-1}} \le \big\|A^{-1}(t)u\big\|_{H^{1}}.
\end{equation}

To prove the inf-sup condition for $B(u, \boldsymbol{v})$, we can take
\begin{equation}\label{B1:1}
    \boldsymbol{v}_{0} = \big(u + c_{0}A^{-1}(t)\partial_{t}u, \varepsilon u(\cdot, t_{\ell})\big), \textnormal{ where } \varepsilon = (1 + 3c_{0})/2.
\end{equation}
By substituting $\boldsymbol{v}_{0}$ defined in Eq. \eqref{B1:1} into Eq. \eqref{p:1-16}, we find
\begin{equation}\label{ap:2-2}
    \begin{aligned}
        B(u, \boldsymbol{v}_{0}; s) &= \int_{t_{\ell}}^{s}\big(\partial_{t}u, u + c_{0}A^{-1}(t)\partial_{t}u\big) + a\big(u, u + c_{0}A^{-1}(t)\partial_{t}u; t\big)\text{d}t+ \varepsilon \|u(\cdot, t_{\ell})\|_{L^{2}}^{2}\\
        &= \int_{t_{\ell}}^{s}(\partial_{t}u, u) + c_{0}\big(\partial_{t}u, A^{-1}(t)\partial_{t}u\big)\\
        &\hspace{1cm}+ a(u, u; t) + c_{0}a\big(u, A^{-1}(t)\partial_{t}u; t\big)\text{d}t + \varepsilon \|u(\cdot, t_{\ell})\|_{L^{2}}^{2}.
    \end{aligned}
\end{equation}

From Eqs. \eqref{ap:2-18} and \eqref{ap:2-4}, we have
\begin{equation}\label{ap:2-33}
        \big(\partial_{t}u, A^{-1}(t)\partial_{t}u\big) = a\big(A^{-1}(t)\partial_{t}u, A^{-1}(t)\partial_{t}u; t\big)\ge c_{0}\big\|A^{-1}(t)\partial_{t}u\big\|_{H^{1}}^{2} \ge \frac{c_{0}}{C_{0}^{2}} \|\partial_{t}u\|_{H^{-1}}^{2}
\end{equation}
and $a\big(u, A^{-1}(t)\partial_{t}u; t\big) = (u, \partial_{t}u)$. 
%By substituting $A^{-1}(t)$ in Eq. \eqref{ap:2-2} by the inequalities presented in Eqs. \eqref{ap:2-33} and \eqref{ap:2-3}, 
Therefore, we conclude that
\begin{equation}\label{p:1-23}
    \begin{aligned}
        B(u, \boldsymbol{v}_{0}; s) &\ge \int_{t_{\ell}}^{s}\big((1 + c_{0})\big(\partial_{t}u, u\big) + \frac{c_{0}^{2}}{C_{0}^{2}}\|\partial_{t}u\|_{H^{-1}}^{2} + c_{0}\|u\|_{H^{1}}^{2}\big)\text{d}t + \varepsilon \|u(\cdot, t_{\ell})\|_{L^{2}}^{2}\\
        &= \int_{t_{\ell}}^{s}(\frac{c_{0}^{2}}{C_{0}^{2}}\|\partial_{t}u\|_{H^{-1}}^{2} + c_{0}\|u\|_{H^{1}}^{2})\text{d}t+ \frac{(1 + c_{0})}{2}\int_{t_{\ell}}^{s}\frac{\text{d}}{\text{d}t}\|u\|_{L^{2}}^{2}\text{d}t + \varepsilon \|u(\cdot, t_{\ell})\|_{L^{2}}^{2}\\
        &\geq \int_{t_{\ell}}^{s}(\frac{c_{0}^{2}}{C_{0}^{2}}\|\partial_{t}u\|_{H^{-1}}^{2} + c_{0}\|u\|_{H^{1}}^{2})\text{d}t + c_{0}\|u(\cdot, t_{\ell})\|_{L^{2}}^{2}\ge \ \alpha_{0}\|u\|_{X(t_{\ell}, s)}^{2},
    \end{aligned}
\end{equation}
where $\alpha_{0} \coloneqq \min\Big\{\frac{c_{0}^{2}}{C_{0}^{2}}, c_{0}\Big\}$.
%which leads to
%\begin{equation}
 %   B(u, \boldsymbol{v}_{0})  \text{ where } 
%\end{equation}
Additionally, since
\begin{equation}\label{B1:2}
    \begin{aligned}
        \|\boldsymbol{v}_{0}\|_{Y(t_{\ell}, s)}^{2} &= \int_{t_{\ell}}^{s}\|v\|_{H^{1}}^{2}\text{d}t + \|\Tilde{v}\|_{L^{2}}^{2}\\
        &= \int_{t_{\ell}}^{s}\|u + c_{0}A^{-1}(t)\partial_{t}u\|_{H^{1}}^{2}\text{d}t + \varepsilon^{2} \|u(\cdot, t_{\ell})\|_{L^{2}}^{2}\\
        & \le \int_{t_{\ell}}^{s}(2\|u\|_{H^{1}}^{2} + 2\|c_{0}A^{-1}(t)\partial_{t}u\|_{H^{1}}^{2}) \text{d}t + \varepsilon^{2} \|u(\cdot, t_{\ell})\|_{L^{2}}^{2}\\
        & \le \int_{t_{\ell}}^{s}(2\|u\|_{H^{1}}^{2} + 2\|\partial_{t}u\|_{H^{1}}^{2}) \text{d}t + \varepsilon^{2} \|u(\cdot, t_{\ell})\|_{L^{2}}^{2}\\
        &\le \max\{\sqrt{2}, \varepsilon\}^{2} \|u\|_{X(t_{\ell}, s)}^{2},
    \end{aligned}
\end{equation}
we conclude from Eq. \eqref{p:1-23} that
\begin{equation}\label{ap:2-5}
    b_{0} \le \inf_{u \neq 0} \sup_{\boldsymbol{v}_{0} \neq 0} \frac{B(u, \boldsymbol{v}_{0}; s)}{\|u\|_{X(t_{\ell}, s)} \|\boldsymbol{v}_{0}\|_{Y(t_{\ell}, s)}} \text{ where } b_{0} \coloneqq \frac{\alpha_{0}}{\max\{\sqrt{2}, \varepsilon\}}.
\end{equation}
\end{proof}

\subsection{A technical lemma}
In this subsection, we shall derive the upper error bound when using the mapped Jacobi approximation for approximating $\Tilde{u}$, which is the solution to the following equation:
\begin{equation}\label{B1:4}
    \begin{aligned}
        \big(\partial_{t}\Tilde{u}, v\big) + a\big(\Tilde{u}, v; t\big) = \big(f(\mathring{U}_{N, \gamma}^{\boldsymbol{\beta}, \boldsymbol{x}_{0}}; t), v\big),\,\,
        \big(\Tilde{u}(\cdot, t_{\ell}), \Tilde{v}\big) = \big(u(\cdot, t_{\ell}), \Tilde{v}\big),\,\,
        \forall (v, \Tilde{v}) \in Y(t_{\ell}, t_{\ell + 1}).
    \end{aligned}
\end{equation}
Here, $\mathring{U}_{N, \gamma}^{\boldsymbol{\beta}, \boldsymbol{x}_{0}}\in X_{N, \gamma}^{\boldsymbol{\beta}, \boldsymbol{x}_{0}}(t_{\ell}, t_{\ell + 1})$ is the mapped Jacobi approximation to $\Tilde{u}$ that solves the following equation:
\begin{equation}\label{B1:27}
    \begin{aligned}
        \big(\partial_{t}\mathring{U}_{N, \gamma}^{\boldsymbol{\beta}, \boldsymbol{x}_{0}}, v_{N}\big) + a\big(\mathring{U}_{N, \gamma}^{\boldsymbol{\beta}, \boldsymbol{x}_{0}}, v_{N}; t\big) &= \big(f(\mathring{U}_{N, \gamma}^{\boldsymbol{\beta}, \boldsymbol{x}_{0}}; t), v_{N}\big)\\
        \big(\mathring{U}_{N, \gamma}^{\boldsymbol{\beta}, \boldsymbol{x}_{0}}(\cdot, t_{\ell}), \Tilde{v}\big) = \big(u(\cdot, t_{\ell}),& \Tilde{v}\big),~\forall (v_{N}, \Tilde{v}_{N}) \in Y_{N, \gamma}^{\boldsymbol{\beta}, \boldsymbol{x}_{0}}(t_{\ell}, t_{\ell + 1}).
    \end{aligned}
\end{equation}
Specifically, the existence of $\mathring{U}_{N, \gamma}^{\boldsymbol{\beta}, \boldsymbol{x}_{0}}$ is given in \cite{Byszewski1991}.
%(need to prove existence first, can say the proof is similar to the proof in Appendix A)This result is in Chou \textit{et al.} \cite{Xia2023}.

\begin{lemma}\label{th:2-1}
Let $\Tilde{u}(\boldsymbol{x}, t)$ be the analytical solution of Eq.~\eqref{B1:4}, and let $\mathring{U}_{N, \gamma}^{\boldsymbol{\beta}, \boldsymbol{x}_{0}}\in X_{N, \gamma}^{\boldsymbol{\beta}, \boldsymbol{x}_{0}}(t_{\ell}, t_{\ell + 1})$ solve Eq.~\eqref{B1:27}.
Then, the error of the continuous-time mapped Jacobi approximation $\mathring{U}_{N, \gamma}^{\boldsymbol{\beta}, \boldsymbol{x}_{0}}$ satisfies
\begin{equation}
    \big\|\mathring{U}_{N, \gamma}^{\boldsymbol{\beta}, \boldsymbol{x}_{0}} - \Tilde{u}\big\|_{X(t_{\ell}, s)} \le A_{0} \big\|\Tilde{u} - \pi_{N, \gamma}^{\boldsymbol{\beta}, \boldsymbol{x}_{0}}\Tilde{u}\big\|_{X(t_{\ell}, s)},~\forall s \in [t_{\ell}, t_{\ell + 1}]
\end{equation}
where $A_{0} \coloneqq (B_{0} + b_{0})/b_{0}$, and the constants $B_{0}$ and $b_{0}$ are defined in Lemma~\ref{B1:3}.
\end{lemma}

\begin{proof}
First, note that $\tilde{u}$ and $\mathring{U}_{N, \gamma}^{\boldsymbol{\beta}, \boldsymbol{x}_{0}}$ are the solutions to the following two equations, respectively:
\begin{equation}\label{B1:5}
    \begin{aligned}
         B(\Tilde{u}, \boldsymbol{v}; s) &= \int_{t_{\ell}}^{s}\big(f(\mathring{U}_{N, \gamma}^{\boldsymbol{\beta}, \boldsymbol{x}_{0}}; t), v\big) \textnormal{d}t + \big(u(\cdot, t_{\ell}), \Tilde{v}\big),\\
         \forall \boldsymbol{v} &= (v, \Tilde{v}) \in Y(t_{\ell}, s),~\forall s\in [t_{\ell}, t_{\ell + 1}],
    \end{aligned}
\end{equation}
and
\begin{equation}\label{B1:5_1}
    \begin{aligned}
        B(\mathring{U}_{N, \gamma}^{\boldsymbol{\beta}, \boldsymbol{x}_{0}}, \boldsymbol{v}_{N}; s) &= \int_{t_{\ell}}^{s}\big(f(\mathring{U}_{N, \gamma}^{\boldsymbol{\beta}, \boldsymbol{x}_{0}}, t), v_{N}\big) \textnormal{d}t + \big(u(\cdot, t_{\ell}), \Tilde{v}_{N}\big),\\
        \forall \boldsymbol{v}_{N} &= (v_{N}, \Tilde{v}_{N}) \in Y_{N, \gamma}^{\boldsymbol{\beta}, \boldsymbol{x}_{0}}(t_{\ell}, s),~\forall s\in [t_{\ell}, t_{\ell + 1}],
    \end{aligned}
\end{equation}
where $B(\cdot, \cdot; s)$ is defined in Eq.~\eqref{p:1-16}. 

By replacing $\boldsymbol{v}\in Y(t_{\ell}, s)$ in Eq.~\eqref{B1:5} with $\boldsymbol{v}_N$ and subtracting Eq.~\eqref{B1:5_1} from Eq. \eqref{B1:5}, we have
\begin{equation}
    B(\mathring{U}_{N, \gamma}^{\boldsymbol{\beta}, \boldsymbol{x}_{0}} - \Tilde{u}, \boldsymbol{v}_{N}; s) = 0,~ \forall \boldsymbol{v}_{N} \in Y_{N, \gamma}^{\boldsymbol{\beta}, \boldsymbol{x}_{0}}(t_{\ell}, s).
\end{equation}
Therefore, from Lemma \ref{B1:3}, we have, for any $z \in X_{N, \gamma}^{\boldsymbol{\beta}, \boldsymbol{x}_{0}}(t_{\ell}, s)$, 
\begin{equation}
\begin{aligned}
    \|\mathring{U}_{N, \gamma}^{\boldsymbol{\beta}, \boldsymbol{x}_{0}} - z\|_{X(t_{\ell}, s)} &\le \frac{1}{b_{0}}\sup_{0 \neq \boldsymbol{v}_{N} \in Y_{N, \gamma}^{\boldsymbol{\beta}, \boldsymbol{x}_{0}}(t_{\ell}, s)} \frac{B(\mathring{U}_{N, \gamma}^{\boldsymbol{\beta}, \boldsymbol{x}_{0}} - z, \boldsymbol{v}_{N})}{\|\boldsymbol{v}_{N}\|_{Y(t_{\ell}, s)}}\\
    & = \frac{1}{b_{0}}\sup_{0 \neq \boldsymbol{v}_{N} \in Y_{N, \gamma}^{\boldsymbol{\beta}, \boldsymbol{x}_{0}}(t_{\ell}, s)} \frac{B(\Tilde{u} - z, \boldsymbol{v}_{N})}{\|\boldsymbol{v}_{N}\|_{Y(t_{\ell}, s)}}\le \frac{B_{0}}{b_{0}}\|\Tilde{u} - z\|_{X(t_{\ell}, s)}.
\end{aligned}
\end{equation}
Finally, by the triangle inequality, we have
\begin{equation}
    \begin{aligned}
        \|\mathring{U}_{N, \gamma}^{\boldsymbol{\beta}, \boldsymbol{x}_{0}} - \Tilde{u}\|_{X(t_{\ell}, s)} &\le \inf_{z\in X_{N, \gamma}^{\boldsymbol{\beta}, \boldsymbol{x}_{0}}(t_{\ell}, s)} \Big(\|\mathring{U}_{N, \gamma}^{\boldsymbol{\beta}, \boldsymbol{x}_{0}} - z\|_{X(t_{\ell}, s)} + \|z - \Tilde{u}\|_{X(t_{\ell}, s)}\Big)\\
        &\le \frac{B_{0} + b_{0}}{b_{0}}\big\|\Tilde{u} - \pi_{N, \gamma}^{\boldsymbol{\beta}, \boldsymbol{x}_{0}}\Tilde{u}\big\|_{X(t_{\ell}, s)}.
    \end{aligned}
\end{equation}
\end{proof}

\subsection{Single-time-step error bound of the mapped Jacobi approximation}
In this subsection, we provide a single-time-step upper bound for the error of the mapped Jacobi method. We first prove the following lemma.
\begin{lemma}\label{B1:14}
Let $s\in[t_{\ell}, t_{\ell+1}]$. Suppose that $\mathring{U}_{N, \gamma}^{\boldsymbol{\beta}, \boldsymbol{x}_{0}}$ is the mapped Jacobi approximation defined in Eq.~\eqref{B1:27}. Let $u(\boldsymbol{x}, t)$ be the solution of the model problem Eq.~\eqref{eq:2-2}, which also solves the following equation:
\begin{equation}\label{B1:25}
    B(u, \boldsymbol{v}; s) = \int_{t_{\ell}}^{s}\big(f(u; t), v\big) \textnormal{d}t + \big(u(\cdot, t_{\ell}), \Tilde{v}\big),~ \forall \boldsymbol{v} \in Y(t_{\ell}, s),~\forall s\in [t_{\ell}, t_{\ell + 1}].
\end{equation}
There exists two constants $C_{\mathcal{M}}$ and $c_{\mathcal{J}}$ such that
\begin{equation}
    \big\|u(\cdot, t_{\ell + 1}) - \mathring{U}_{N, \gamma}^{\boldsymbol{\beta}, \boldsymbol{x}_{0}}(\cdot, t_{\ell + 1})\big\|_{L^{2}} \le C_{\mathcal{M}} \exp\big(c_{\mathcal{J}} (t_{\ell + 1} - t_{\ell})\big) \big\|u - \pi_{N, \gamma}^{\boldsymbol{\beta}, \boldsymbol{x}_{0}}u\big\|_{X(t_{\ell}, t_{\ell + 1})}.
\end{equation}
\end{lemma}

\begin{proof}
 We decompose the following approximation error into two parts
\begin{equation}\label{ap:2-6}
    \|\mathring{U}_{N, \gamma}^{\boldsymbol{\beta}, \boldsymbol{x}_{0}} - u\|_{X(t_{\ell}, s)} \le \big\|\mathring{U}_{N, \gamma}^{\boldsymbol{\beta}, \boldsymbol{x}_{0}} - \Tilde{u}\big\|_{X(t_{\ell}, s)} + \big\|\Tilde{u} - u\big\|_{X(t_{\ell}, s)},
\end{equation}
where $\Tilde{u}$ is defined in Eq. \eqref{B1:4}. The upper bound of the first term on the RHS of Eq. \eqref{ap:2-6} is given in Lemma \ref{th:2-1}:
\begin{equation}\label{B1:13}
    \big\|\mathring{U}_{N, \gamma}^{\boldsymbol{\beta}, \boldsymbol{x}_{0}} - \Tilde{u}\big\|_{X(t_{\ell}, s)} \le A_{0} \big\|\Tilde{u} - \pi_{N, \gamma}^{\boldsymbol{\beta}, \boldsymbol{x}_{0}}\Tilde{u}\big\|_{X(t_{\ell}, s)}.
\end{equation}

Next, we give an upper bound of the second term on the RHS of Eq. \eqref{ap:2-6}, To this end, subtracting Eq.~\eqref{B1:25} from Eq. \eqref{B1:5_1}, we have
\begin{equation}
    B(\Tilde{u} - u, \boldsymbol{v}; s) = \int_{t_{\ell}}^{s}\big(f(\mathring{U}_{N, \gamma}^{\boldsymbol{\beta}, \boldsymbol{x}_{0}}; t) - f(u; t), v\big) \text{d}t.
\label{B1:201}
\end{equation}

By replacing $\boldsymbol{v}$ in Eq. \eqref{B1:201} with
\begin{equation}\label{B1:28}
    \boldsymbol{v}_{1} = (v_{1}, \Tilde{v}_{1}) \coloneqq \big(\Tilde{u} - u + c_{0}A^{-1}(t)\partial_{t}(\Tilde{u} - u), 0\big),
\end{equation}
we have
\begin{equation}\label{ap:2-10}
    \begin{aligned}
        B(\Tilde{u} - u, \boldsymbol{v}_{1}; s) &= \int_{t_{\ell}}^{s}\big(\partial_{t}(\Tilde{u} - u),  \Tilde{u} - u\big) + \big(\partial_{t}(\Tilde{u} - u), c_{0}A^{-1}(t)\partial_{t}(\Tilde{u} - u)\big)\\
        &\quad + a\big(\Tilde{u} - u, \Tilde{u} - u; t\big) + a\big(\Tilde{u} - u, c_{0}A^{-1}(t)\partial_{t}(\Tilde{u} - u); t\big)\text{d}t\\
        & \ge \int_{t_{\ell}}^{s}(1 + c_{0})\big(\partial_{t}(\Tilde{u} - u), \Tilde{u} - u\big) \text{d}t + \alpha_{0}\big\|\Tilde{u} - u\big\|_{X(t_{\ell}, s)}^{2},
    \end{aligned}
\end{equation}
where $\alpha_{0} \coloneqq \min\{c_{0}^{2}/C_{0}^{2}, c_{0}\}$. The proof of inequality~\eqref{ap:2-10} is similar to that of Eq. \eqref{ap:2-2}. Additionally, since
\begin{equation}\label{ap:2-11}
    \big(f(\mathring{U}_{N, \gamma}^{\boldsymbol{\beta}, \boldsymbol{x}_{0}}; t) - f(u; t), v_{1}\big) \le L\|\mathring{U}_{N, \gamma}^{\boldsymbol{\beta}, \boldsymbol{x}_{0}} - u\|_{L^{2}} \|v_{1}\|_{L^{2}},
\end{equation}
by combining Eqs. \eqref{ap:2-10} and \eqref{ap:2-11}, we have
\begin{equation}\label{B1:10}
\begin{aligned}
    \int_{t_{\ell}}^{s}(1 + c_{0})\big(\partial_{t}(\Tilde{u} - u),& \Tilde{u} - u\big) \text{d}t + \alpha_{0}\big\|\Tilde{u} - u\big\|_{X(t_{\ell}, s)}^{2}\\
    &\le B(\Tilde{u} - u, \boldsymbol{v}_{1}) =  \int_{t_{\ell}}^{s}\big(f(\mathring{U}_{N, \gamma}^{\boldsymbol{\beta}, \boldsymbol{x}_{0}}; t) - f(u; t), v_{1})\text{d}t\\
    & \le L\int_{t_{\ell}}^{s}\|\mathring{U}_{N, \gamma}^{\boldsymbol{\beta}, \boldsymbol{x}_{0}} - u\|_{L^{2}} \|v_{1}\|_{L^{2}}\text{d}t.
\end{aligned}
\end{equation}

Since $v_{1} = \Tilde{u} - u + c_{0}A^{-1}(t)\partial_{t}(\Tilde{u} - u)$ (defined in Eq. \eqref{B1:28}), we have, , for any $t\in[t_{\ell}, s]$:
\begin{equation}\label{B1:23}
    \|v_1(\cdot, t)\|_{L^2}^2 = \|\Tilde{u} - u + c_{0}A^{-1}(t)\partial_{t}(\Tilde{u} - u)\|_{L^{2}}^2\le 2 \|\partial_{t}(\Tilde{u} - u)(\cdot, t)\|_{H^{-1}}^{2} + 2\|(\Tilde{u} - u)(\cdot, t)\|_{H^{1}}^{2}.
\end{equation}
Integrating Eq.~\eqref{B1:23} over time from $t_{\ell}$ to $s$, we have
\begin{equation}\label{B1:9}
    \begin{aligned}
        \int_{t_{\ell}}^{s} \|v_{1}\|_{L^{2}}^{2}\text{d}t \le \int_{t_{\ell}}^{s} 2 \|\partial_{t}(\Tilde{u} - u)\|_{H^{-1}}^{2} + 2\|\Tilde{u} - u\|_{H^{1}}^{2}\text{d}t= 2\|\Tilde{u} - u\|_{X(t_{\ell}, s)}^{2}.
    \end{aligned}
\end{equation}

Combining Eq.~\eqref{B1:9} and Eq.~\eqref{B1:10}, we have
\begin{equation}\label{B1:11}
    \begin{aligned}
        L\int_{t_{\ell}}^{s}\|\mathring{U}_{N, \gamma}^{\boldsymbol{\beta}, \boldsymbol{x}_{0}} - u\|_{L^{2}} \|v_{1}\|_{L^{2}}\text{d}t
        &\le \frac{L^{2}}{\alpha_{0}}\int_{t_{\ell}}^{s} \|\mathring{U}_{N, \gamma}^{\boldsymbol{\beta}, \boldsymbol{x}_{0}} - u\|_{L^{2}}^{2}\text{d}t + \frac{\alpha_{0}}{4}\int_{t_{\ell}}^{s}\|v_{1}\|_{L^{2}}^{2}\text{d}t\\
        &\le \frac{L^{2}}{\alpha_{0}}\int_{t_{\ell}}^{s} \|\mathring{U}_{N, \gamma}^{\boldsymbol{\beta}, \boldsymbol{x}_{0}} - u\|_{L^{2}}^{2}\text{d}t + \frac{\alpha_{0}}{2} \|\Tilde{u} - u\|_{X(t_{\ell}, s)}^{2}.
    \end{aligned}
\end{equation}
Additionally, since
\begin{equation}\label{B1:12}
    \begin{aligned}
        (1 + c_{0})\int_{t_{\ell}}^{s}&\big(\partial_{t}(\Tilde{u} - u), \Tilde{u} - u\big)\text{d}t = \frac{1 + c_{0}}{2}\int_{t_{\ell}}^{s}\frac{\text{d}}{\text{d}t} \|\Tilde{u} - u\|_{L^{2}}^{2} \text{d}t\\
        & = \frac{1 + c_{0}}{2}\big(\|\Tilde{u}(\cdot, s) - u(\cdot, s)\|_{L^{2}}^{2} - \|\Tilde{u}(\cdot, t_{\ell}) - u(\cdot, t_{\ell})\|_{L^{2}}^{2}\big)\\
        & = \frac{1 + c_{0}}{2} \|\Tilde{u}(\cdot, s) - u(\cdot, s)\|_{L^{2}}^{2} \ge 0,
    \end{aligned}
\end{equation}
by combining Eqs. \eqref{B1:10}, \eqref{B1:11}, and \eqref{B1:12}, we can conclude that
\begin{equation}
    \alpha_{0}\|\Tilde{u} - u\|_{X(t_{\ell}, s)} \le \frac{L^{2}}{\alpha_{0}}\int_{t_{\ell}}^{s} \|\mathring{U}_{N, \gamma}^{\boldsymbol{\beta}, \boldsymbol{x}_{0}} - u\|_{L^{2}}^{2}\text{d}t + \frac{\alpha_{0}}{2} \|\Tilde{u} - u\|_{X(t_{\ell}, s)}^{2}.
\end{equation}
Therefore, we have 
\begin{equation}
\label{B1:30}
\|\Tilde{u} - u\|_{X(t_{\ell}, s)}^{2} \le \frac{2L^{2}}{\alpha_{0}^{2}} \int_{t_{\ell}}^{s} \|\mathring{U}_{N, \gamma}^{\boldsymbol{\beta}, \boldsymbol{x}_{0}} - u\|_{L^{2}}^{2}\text{d}t.
\end{equation}

Finally, using the upper bounds (Eqs. \eqref{B1:13} and \eqref{B1:30}) for the two terms on the RHS of Eq. \eqref{ap:2-6}, we have
\begin{equation}
\begin{aligned}
    \|u - \mathring{U}_{N, \gamma}^{\boldsymbol{\beta}, \boldsymbol{x}_{0}}\|_{X(t_{\ell}, s)}^{2} &\le \big(\|u - \Tilde{u}\|_{X(t_{\ell}, s)} + \|\Tilde{u} - \mathring{U}_{N, \gamma}^{\boldsymbol{\beta}, \boldsymbol{x}_{0}}\|_{X(t_{\ell}, s)}\big)^{2}\\
    &\le \big(\|u - \Tilde{u}\|_{X(t_{\ell}, s)} + A_{0}\|\Tilde{u} - \pi_{N, \gamma}^{\boldsymbol{\beta}, \boldsymbol{x}_{0}}\Tilde{u}\|_{X(t_{\ell}, s)}\big)^{2}\\
    &= \big(\|u - \Tilde{u}\|_{X(t_{\ell}, s)}+ A_{0}\big\|(I - \pi_{N, \gamma}^{\boldsymbol{\beta}, \boldsymbol{x}_{0}})\big(u - (u - \Tilde{u})\big)\big\|_{X(t_{\ell}, s)}\big)^{2}\\
    &\le \big((1 + A_{0})\|u - \Tilde{u}\|_{X(t_{\ell}, s)} + A_{0}\|u - \pi_{N, \gamma}^{\boldsymbol{\beta}, \boldsymbol{x}_{0}}u\|_{X(t_{\ell}, s)}\big)^{2}\\
    &\le 2(1 + A_{0})^{2}\|u - \Tilde{u}\|_{X(t_{\ell}, s)}^{2} + 2A_{0}^{2}\big\|u - \pi_{N, \gamma}^{\boldsymbol{\beta}, \boldsymbol{x}_{0}}u\big\|_{X(t_{\ell}, s)}^{2}\\
    &\hspace{-0.9cm}\le 4(1 + A_{0})^{2}\frac{L^{2}}{\alpha_{0}^{2}} \int_{t_{\ell}}^{s} \|u - \mathring{U}_{N, \gamma}^{\boldsymbol{\beta}, \boldsymbol{x}_{0}}\|_{L^{2}}^{2}\text{d}t+ 2A_{0}^{2}\big\|u - \pi_{N, \gamma}^{\boldsymbol{\beta}, \boldsymbol{x}_{0}}u\big\|_{X(t_{\ell}, s)}^{2},
\end{aligned}
\end{equation}
where $A_{0}$ is the constant defined in Lemma \ref{th:2-1}.

Futhermore, since 
\begin{equation}
    \begin{aligned}
        \|u - \mathring{U}_{N, \gamma}^{\boldsymbol{\beta}, \boldsymbol{x}_{0}}\|_{X(t_{\ell}, s)}^{2} &= \int_{t_{\ell}}^{s}\Big(\|\partial_{t}\big(u(\cdot, t) - \mathring{U}_{N, \gamma}^{\boldsymbol{\beta}, \boldsymbol{x}_{0}}(\cdot, t)\big)\|_{H^{-1}}^{2} \\
        &\quad+ \|u - \mathring{U}_{N, \gamma}^{\boldsymbol{\beta}, \boldsymbol{x}_{0}}(\cdot, t_{\ell})\|_{H^{1}}^{2}\Big)\text{d}t + \|u(\cdot, t_{\ell}) - \mathring{U}_{N, \gamma}^{\boldsymbol{\beta}, \boldsymbol{x}_{0}}(\cdot, t_{\ell})\|_{L^{2}}^{2}\\
        &\ge \int_{t_{\ell}}^{s}2 \big(\partial_{t}(u(\cdot, t) - \mathring{U}_{N, \gamma}^{\boldsymbol{\beta}, \boldsymbol{x}_{0}}(\cdot, t)), u(\cdot, t) - \mathring{U}_{N, \gamma}^{\boldsymbol{\beta}, \boldsymbol{x}_{0}}(\cdot, t)\big)\text{d}t \\
        &\quad+ \|u(\cdot, t_{\ell}) - \mathring{U}_{N, \gamma}^{\boldsymbol{\beta}, \boldsymbol{x}_{0}}(\cdot, t_{\ell})\|_{L^{2}}^{2} \\
        &= \|u(\cdot, s) - \mathring{U}_{N, \gamma}^{\boldsymbol{\beta}, \boldsymbol{x}_{0}}(\cdot, s)\|_{L^{2}}^{2},
    \end{aligned}
\end{equation}
we have
\begin{equation}
\begin{aligned}
    \|u(\cdot, s) - \mathring{U}_{N, \gamma}^{\boldsymbol{\beta}, \boldsymbol{x}_{0}}(\cdot, s)\|_{L^{2}}^{2} 
    \le 4(1 + A_{0})^{2}\frac{L^{2}}{\alpha_{0}^{2}} \int_{t_{\ell}}^{s} \|u - \mathring{U}_{N, \gamma}^{\boldsymbol{\beta}, \boldsymbol{x}_{0}}\|_{L^{2}}^{2}\text{d}t + 2A_{0}^{2}\big\|u - \pi_{N, \gamma}^{\boldsymbol{\beta}, \boldsymbol{x}_{0}}u\big\|_{X(t_{\ell}, s)}^{2}.
\end{aligned}
\label{t_ell_ineq}
\end{equation}
Actually, the inequality~\eqref{t_ell_ineq} holds for any $t_{\ell} \le s \le t_{\ell + 1}$, and $\big\|u - \pi_{N, \gamma}^{\boldsymbol{\beta}, \boldsymbol{x}_{0}}u\big\|_{X(t_{\ell}, s)}^{2}$ is a non-decreasing function of $s$. By using the Gronwall inequality, we can show that
\begin{equation}\label{B1:31}
    \|u(\cdot, s) - \mathring{U}_{N, \gamma}^{\boldsymbol{\beta}, \boldsymbol{x}_{0}}(\cdot, s)\|_{L^{2}}^{2} \le C_{\mathcal{M}}^{2} \exp(2c_{\mathcal{J}} (s - t_{\ell})) \|u - \pi_{N, \gamma}^{\boldsymbol{\beta}, \boldsymbol{x}_{0}}u\|_{X(t_{\ell}, s)}^{2}.
\end{equation}
Finally, by taking $s = t_{\ell + 1}$ in Eq.~\eqref{B1:31}, we have
\begin{equation}
    \|u(\cdot, t_{\ell + 1}) - \mathring{U}_{N, \gamma}^{\boldsymbol{\beta}, \boldsymbol{x}_{0}}(\cdot, t_{\ell + 1})\|_{L^{2}} \le C_{\mathcal{M}} \exp(c_{\mathcal{J}} \Delta t) \|u - \pi_{N, \gamma}^{\boldsymbol{\beta}, \boldsymbol{x}_{0}}u\|_{X(t_{\ell}, t_{\ell + 1})},
\end{equation}
where $\Delta t \coloneqq t_{\ell + 1} - t_{\ell}$, and the constants are given by
\begin{equation}
    C_{\mathcal{M}} \coloneqq \sqrt{2}A_{0}\quad c_{\mathcal{J}} = (1 + A_{0})^{2} \frac{2L^{2}}{\alpha_{0}^{2}}.
\end{equation}
and
\begin{equation}
    \begin{aligned}
        A_{0} \coloneqq \frac{B_{0} + b_{0}}{b_{0}},~ \alpha_{0} \coloneqq \min\{c_{0}^{2}/C_{0}^{2}, c_{0}\},\,\,
        B_{0} \coloneqq \sqrt{2 + 2C_{0}^{2}},~b_{0} \coloneqq \frac{\alpha_{0}}{\max\{\sqrt{2}, \varepsilon\}}.
    \end{aligned}
\end{equation}
where $\varepsilon = (1 + 3c_{0})/2$.
\end{proof}

In the following, we conclude the mapped Jacobi error analysis by estimating the error of the mapped Jacobi approximation $t=t_{\ell+1}$ when the initial condition at $t=t_{\ell}$ is not $u(\cdot, t_{\ell})$. Suppose that $\tilde{U}_{N, \gamma}^{\boldsymbol{\beta}, \boldsymbol{x}_{0}}\in X_{N, \gamma}^{\boldsymbol{\beta}, \boldsymbol{x}_{0}}(t_{\ell}, t_{\ell + 1})$ solves Eq. \eqref{eq:4-23}, which also satisfies the following equation:
\begin{equation}\label{B1:17}
    \begin{aligned}
        B(\tilde{U}_{N, \gamma}^{\boldsymbol{\beta}, \boldsymbol{x}_{0}}, \boldsymbol{v}_{N}; s) &= \int_{t_{\ell}}^{s}\big(f(\Tilde{U}_{N, \gamma}^{\boldsymbol{\beta}, \boldsymbol{x}_{0}}; t), v_{N}\big) \textnormal{d}t + \big(U(\cdot, t_{\ell}), \Tilde{v}_{N}\big),\\
        \forall \boldsymbol{v}_{N} &= (v_{N}, \tilde{v}_{N}) \in Y_{N, \gamma}^{\boldsymbol{\beta}, \boldsymbol{x}_{0}}(t_{\ell}, s),~\forall s\in [t_{\ell}, t_{\ell + 1}].
    \end{aligned}
\end{equation}
We decompose the error into two parts:
\begin{equation}\label{B1:15}
    \begin{aligned}
        \|u(\cdot, t_{\ell + 1}) - \tilde{U}_{N, \gamma}^{\boldsymbol{\beta}, \boldsymbol{x}_{0}}(\cdot, t_{\ell + 1})\|_{L^{2}} &\le \|u(\cdot, t_{\ell + 1}) - \mathring{U}_{N, \gamma}^{\boldsymbol{\beta}, \boldsymbol{x}_{0}}(\cdot, t_{\ell + 1})\|_{L^{2}}+ \|\mathring{U}_{N, \gamma}^{\boldsymbol{\beta}, \boldsymbol{x}_{0}}(\cdot, t_{\ell + 1}) - \tilde{U}_{N, \gamma}^{\boldsymbol{\beta}, \boldsymbol{x}_{0}}(\cdot, t_{\ell + 1})\|_{L^{2}}.
    \end{aligned}
\end{equation}
The upper bound of the first term on the RHS of Eq. \eqref{B1:15} is in Theorem \ref{B1:14}:
\begin{equation}
    \|u(\cdot, t_{\ell + 1}) - \mathring{U}_{N, \gamma}^{\boldsymbol{\beta}, \boldsymbol{x}_{0}}(\cdot, t_{\ell + 1})\|_{L^{2}} \le C_{\mathcal{M}} \exp(c_{\mathcal{J}} \Delta t) \|u - \pi_{N, \gamma}^{\boldsymbol{\beta}, \boldsymbol{x}_{0}}u\|_{X(t_{\ell}, t_{\ell + 1})}.
\end{equation}
Next, we derive an upper bound for the second term $\|\mathring{U}_{N, \gamma}^{\boldsymbol{\beta}, \boldsymbol{x}_{0}}(\cdot, t_{\ell + 1}) - \tilde{U}_{N, \gamma}^{\boldsymbol{\beta}, \boldsymbol{x}_{0}}(\cdot, t_{\ell + 1})\|_{L^{2}}$ in Eq. \eqref{B1:15}. By subtracting Eq.~\eqref{B1:17} from Eq. \eqref{B1:5}, we have
\begin{equation}\label{B1:18}
    \begin{aligned}
        B(\mathring{U}_{N, \gamma}^{\boldsymbol{\beta}, \boldsymbol{x}_{0}} - \tilde{U}_{N, \gamma}^{\boldsymbol{\beta}, \boldsymbol{x}_{0}}, \boldsymbol{v}_{N}; s) &= \int_{t_{\ell}}^{s}\big(f(\mathring{U}_{N, \gamma}^{\boldsymbol{\beta}, \boldsymbol{x}_{0}}; t) - f(\tilde{U}_{N, \gamma}^{\boldsymbol{\beta}, \boldsymbol{x}_{0}}; t), v_{N}\big)\text{d}t + \big(u(\cdot, t_{\ell}) - U(\cdot, t_{\ell}), \tilde{v}_{N}\big).
    \end{aligned}
\end{equation}
We can set
\begin{equation}\label{B1:19}
    \boldsymbol{v}_{2} = (v_{2}, \Tilde{v}_{2}) \coloneqq \big(\mathring{U}_{N, \gamma}^{\boldsymbol{\beta}, \boldsymbol{x}_{0}} - \tilde{U}_{N, \gamma}^{\boldsymbol{\beta}, \boldsymbol{x}_{0}}, \mathring{U}_{N, \gamma}^{\boldsymbol{\beta}, \boldsymbol{x}_{0}}(\cdot, t_{\ell}) - \tilde{U}_{N, \gamma}^{\boldsymbol{\beta}, \boldsymbol{x}_{0}}(\cdot, t_{\ell})\big),
\end{equation}
and by replacing $\boldsymbol{v}$ in Eq. \eqref{B1:18} with $\boldsymbol{v_{2}}$ defined in \eqref{B1:19}, we have
\begin{equation}\label{B1:20}
    \begin{aligned}
        B(\mathring{U}_{N, \gamma}^{\boldsymbol{\beta}, \boldsymbol{x}_{0}} - \tilde{U}_{N, \gamma}^{\boldsymbol{\beta}, \boldsymbol{x}_{0}}, \boldsymbol{v}_{2}; s) &= \int_{t_{\ell}}^{s} \big(\partial_{t}v_{2}, v_{2}\big) + a\big(v_{2}, v_{2}\big) \text{d}t + \|v_{2}(\cdot, t_{\ell})\|_{L^{2}}\\
        &\ge \int_{t_{\ell}}^{s} \frac{1}{2}\frac{\text{d}}{\text{d}t}\|v_{2}\|_{L^{2}}^{2} + c_{0}\|v_{2}\|_{H^{1}} \text{d}t + \|v_{2}(\cdot, t_{\ell})\|_{L^{2}}\\
        &\ge \int_{t_{\ell}}^{s} \frac{1}{2}\frac{\text{d}}{\text{d}t}\|v_{2}\|_{L^{2}}^{2} + c_{0}\|v_{2}\|_{L^{2}} \text{d}t + \|v_{2}(\cdot, t_{\ell})\|_{L^{2}}.
    \end{aligned}
\end{equation}
Additionally, since
\begin{equation}\label{B1:21}
    \begin{aligned}
        \int_{t_{\ell}}^{s}\big(f(\mathring{U}_{N, \gamma}^{\boldsymbol{\beta}, \boldsymbol{x}_{0}}; t) &- f(\tilde{U}_{N, \gamma}^{\boldsymbol{\beta}, \boldsymbol{x}_{0}}; t), v_{2}\big)\text{d}t + \big(u(\cdot, t_{\ell}) - U(\cdot, t_{\ell}), \tilde{v}_{2}\big)\\
        &\le \int_{t_{\ell}}^{s}L\|v_{2}\|_{L^{2}}^{2}\text{d}t + \|u(\cdot, t_{\ell}) - U(\cdot, t_{\ell})\|_{L^{2}}\|v_{2}(\cdot, t_{\ell})\|_{L^{2}} \\
        &\le \int_{t_{\ell}}^{s}L\|v_{2}\|_{L^{2}}^{2}\text{d}t + \frac{\|u(\cdot, t_{\ell}) - U(\cdot, t_{\ell})\|_{L^{2}}^{2} + \|v_{2}(\cdot, t_{\ell})\|_{L^{2}}^{2}}{2},
    \end{aligned}
\end{equation}
by combining Eqs. \eqref{B1:20} and \eqref{B1:21}, we have
\begin{equation}
    \begin{aligned}
        \int_{t_{\ell}}^{s} \frac{1}{2}\frac{\text{d}}{\text{d}t} & \|v_{2}\|_{L^{2}}^{2}+ c_{0}\|v_{2}\|_{L^{2}}^{2} \text{d}t + \|v_{2}(\cdot, t_{\ell})\|_{L^{2}}^{2} \le B(\mathring{U}_{N, \gamma}^{\boldsymbol{\beta}, \boldsymbol{x}_{0}} - \tilde{U}_{N, \gamma}^{\boldsymbol{\beta}, \boldsymbol{x}_{0}}, \boldsymbol{v}_{2})\\
%        &= \int_{t_{\ell}}^{s}\big(f(\mathring{U}_{N, \gamma}^{\boldsymbol{\beta}, \boldsymbol{x}_{0}}; t) - f(\tilde{U}_{N, \gamma}^{\boldsymbol{\beta}, \boldsymbol{x}_{0}}; t), v_{2}\big)\text{d}t + \big(u(\cdot, t_{\ell}) - U(\cdot, t_{\ell}), \tilde{v}_{2}\big)\\
        &\le \int_{t_{\ell}}^{s}L\|v_{2}\|_{L^{2}}^{2}\text{d}t + \frac{\|u(\cdot, t_{\ell}) - U(\cdot, t_{\ell})\|_{L^{2}}^{2} + \|v_{2}(\cdot, t_{\ell})\|_{L^{2}}^{2}}{2}.
    \end{aligned}
\end{equation}
Therefore, we have
\begin{equation}
    \|v_{2}(\cdot, s)\|_{L^{2}}^{2} + 2\int_{t_{\ell}}^{s} (c_{0} - L)\|v_{2}(\cdot, t)\|_{L^{2}}^{2} \text{d}t \le \|u(\cdot, t_{\ell}) - U(\cdot, t_{\ell})\|_{L^{2}}^{2}.
\label{v2_ineq}
\end{equation}
The inequality~\eqref{v2_ineq} actually holds for any $t_{\ell + 1} \ge s \ge t_{\ell}$, \textit{i.e.},
\begin{equation}
    \|v_{2}(\cdot, s)\|_{L^{2}}^{2} + 2\int_{t_{\ell}}^{s} (c_{0} - L)\|v_{2}(\cdot, t)\|_{L^{2}}^{2} \text{d}t \le \|u(\cdot, t_{\ell}) - U(\cdot, t_{\ell})\|_{L^{2}}^{2}.
\end{equation}
By using the Gronwall inequality, we can show that
\begin{equation}\label{B1:32}
    \|v_{2}(\cdot, s)\|_{L^{2}}^{2} \le \exp\big(2(L - c_{0}) (s - t_{\ell})\big) \|u(\cdot, t_{\ell}) - U(\cdot, t_{\ell})\|_{L^{2}}^{2}
\end{equation}
Finally, by taking $s = t_{\ell + 1}$ in Eq.~\eqref{B1:32}, we conclude that:
\begin{equation}\label{B1:22}
    \begin{aligned}
        \|\mathring{U}_{N, \gamma}^{\boldsymbol{\beta}, \boldsymbol{x}_{0}}(\cdot, t_{\ell + 1}) - \tilde{U}_{N, \gamma}^{\boldsymbol{\beta}, \boldsymbol{x}_{0}}(\cdot, t_{\ell + 1})\|_{L^{2}}^{2} &= \|v_{2}(\cdot, t_{\ell + 1})\|_{L^{2}}^{2}\\
        &\le \exp\big(2(L - c_{0}) \Delta t\big) \|u(\cdot, t_{\ell}) - U(\cdot, t_{\ell})\|_{L^{2}}^{2}.
    \end{aligned}
\end{equation}
Plugging this upper bound for $\|\mathring{U}_{N, \gamma}^{\boldsymbol{\beta}, \boldsymbol{x}_{0}}(\cdot, t_{\ell + 1}) - \tilde{U}_{N, \gamma}^{\boldsymbol{\beta}, \boldsymbol{x}_{0}}(\cdot, t_{\ell + 1})\|_{L^{2}}$ into Eq. \eqref{B1:15}, we have
\begin{equation}
    \begin{aligned}
        \|u(\cdot, t_{\ell + 1}) - \tilde{U}_{N, \gamma}^{\boldsymbol{\beta}, \boldsymbol{x}_{0}}(\cdot, t_{\ell + 1})\|_{L^{2}} &\le \exp\big((L - c_{0}) \Delta t\big) \|u(\cdot, t_{\ell}) - U(\cdot, t_{\ell})\|_{L^{2}}\\
        &\quad + C_{\mathcal{M}} \exp(c_{\mathcal{J}} \Delta t) \big\|u - \pi_{N, \gamma}^{\boldsymbol{\beta}, \boldsymbol{x}_{0}}u\big\|_{X(t_{\ell}, t_{\ell + 1})}.
    \end{aligned}
\end{equation}

\section{Proof of Theorem \ref{th:6-2}}\label{ap:4}
Here, we prove the error bound Eq. \eqref{RK_error_bound} for the IRK scheme Eq. \eqref{eq:2-21}. The result is based on the approach in Lubich et al. \cite{Ostermann1995}. 

We denote the solution of the IRK scheme at $t=t_{\ell+1}$ given the initial condition $U_{N}(\boldsymbol{x}, t_{\ell})$ at $t=t_{\ell}$ as $U_{N, \gamma}^{\boldsymbol{\beta}, \boldsymbol{x}_{0}}(\boldsymbol{x}, t_{\ell + 1})$. The stage variable at $t=t_{0s}$ will be denoted as $U_{N}^{0s}$. Additionally, the continuous-time mapped Jacobi approximation will be denoted as $\Tilde{U}_{N, \gamma}^{\boldsymbol{\beta}, \boldsymbol{x}_{0}}(\boldsymbol{x}, t)$ and the continuous-time mapped Jacobi approximations on the stage time node $t_{0s}$ will be as $\tilde{U}_{N}^{0s} \coloneqq \Tilde{U}_{N, \gamma}^{\boldsymbol{\beta}, \boldsymbol{x}_{0}}(\boldsymbol{x}, t_{0s})$. We shall derive the IRK time discretization error bound within the interval $[t_{\ell}, t_{\ell + 1}]$. The IRK scheme is formulated as
\begin{equation}\label{B2:1}
    \begin{aligned}
        &U_{N}^{0r} = U_{N}(\boldsymbol{x}, t_{\ell}) + \Delta t \sum_{s = 1}^{q}a_{RK}^{rs}\big(F(U^{0s}_{N}; t_{0s}) - A(U^{0s}_{N}; t_{0s})\big),\\
        &U_{N, \gamma}^{\boldsymbol{\beta}, \boldsymbol{x}_{0}}(\boldsymbol{x}, t_{\ell + 1}) = U_{N}(\boldsymbol{x}, t_{\ell}) + \Delta t \sum_{s = 1}^{q}b_{RK}^{s}\big(F(U^{0s}_{N}; t_{0s}) - A(U^{0s}_{N}; t_{0s})\big).
    \end{aligned}
\end{equation}
Here, $A(\cdot; t)$ and $F(\cdot, t)$ are defined as
\begin{equation}
    \begin{aligned}
        A(U_{N, \gamma}^{\boldsymbol{\beta}, \boldsymbol{x}_{0}}; t) &\coloneqq \sum_{\boldsymbol{m} \in \Upsilon_{N, \gamma}} \sum_{\boldsymbol{n} \in \Upsilon_{N, \gamma}} a(\mathcal{J}_{\boldsymbol{n}}^{\boldsymbol{\beta}, \boldsymbol{x}_{0}}, \mathcal{J}_{\boldsymbol{m}}^{\boldsymbol{\beta}, \boldsymbol{x}_{0}}; t) u_{\boldsymbol{n}}^{\boldsymbol{\beta}, \boldsymbol{x}_{0}} \mathcal{J}_{\boldsymbol{m}}^{\boldsymbol{\beta}, \boldsymbol{x}_{0}},\\
        F(U_{N, \gamma}^{\boldsymbol{\beta}, \boldsymbol{x}_{0}}; t) &\coloneqq \pi_{N, \gamma}^{\boldsymbol{\beta}, \boldsymbol{x}_{0}} f(U_{N, \gamma}^{\boldsymbol{\beta}, \boldsymbol{x}_{0}}; t),
    \end{aligned}
\end{equation}
where $\{u_{\boldsymbol{n}}^{\boldsymbol{\beta}, \boldsymbol{x}_{0}}\}_{\boldsymbol{n}\in\Upsilon_{N, \gamma}}$ are the spectral expansion coefficients of the solution $U_{N, \gamma}^{\boldsymbol{\beta}, \boldsymbol{x}_{0}}$ such that $U_{N, \gamma}^{\boldsymbol{\beta}, \boldsymbol{x}_{0}}(\boldsymbol{x}) = \sum_{\boldsymbol{n} \in \Upsilon_{N, \gamma}} u_{\boldsymbol{n}}^{\boldsymbol{\beta}, \boldsymbol{x}_{0}} \mathcal{J}_{\boldsymbol{n}}^{\boldsymbol{\beta}, \boldsymbol{x}_{0}}(\boldsymbol{x})$. Additionally, it can be verified that $\forall v_{N} \in V_{N, \gamma}^{\boldsymbol{\beta}, \boldsymbol{x}_{0}}$,
\begin{equation}\label{B1:26}
        a(U_{N, \gamma}^{\boldsymbol{\beta}, \boldsymbol{x}_{0}}, v_{N}; t) = \big(A(U_{N, \gamma}; t)^{\boldsymbol{\beta}, \boldsymbol{x}_{0}}, v_{N}\big),\,\,
        (F(U_{N, \gamma}^{\boldsymbol{\beta}, \boldsymbol{x}_{0}}; t), v_{N}) = \big(f(U_{N, \gamma}^{\boldsymbol{\beta}, \boldsymbol{x}_{0}}; t), v_{N}\big).
\end{equation}

Plugging the continuous-time mapped Jacobi approximation $\Tilde{U}_{N, \gamma}^{\boldsymbol{\beta}, \boldsymbol{x}_{0}}(t)$ into the IRK scheme Eq. \eqref{B2:1}, we have
\begin{equation}\label{p:3-1}
    \begin{aligned}
        \Tilde{U}_{N, \gamma}^{\boldsymbol{\beta}, \boldsymbol{x}_{0}}(\boldsymbol{x}, t_{0r}) &= U_{N}(\boldsymbol{x}, t_{\ell}) + \Delta t \sum_{s = 1}^{q}a_{RK}^{rs} \big(F(\Tilde{U}^{0s}_{N}; t_{0s}) - A(\Tilde{U}^{0s}_{N};t_{0s})\big) + Q_{0r},\\~
        \Tilde{U}_{N, \gamma}^{\boldsymbol{\beta}, \boldsymbol{x}_{0}}(\boldsymbol{x}, t_{\ell + 1}) &= U_{N}(\boldsymbol{x}, t_{\ell}) + \Delta t \sum_{s = 1}^{q}b_{RK}^{s} \big(F(\Tilde{U}^{0s}_{N}; t_{0s}) - A(\Tilde{U}^{0s}_{N};t_{0s})\big) + Q_{1},
    \end{aligned}
\end{equation}
where $Q_{0r}$ and $Q_{1}$ denote the defects of the equation. $Q_{0r}$ and $Q_{1}$ can be represented with the Peano kernels $\kappa_{0s}$ and $\kappa_{1}$ \citep[Theorem 1.1]{Ostermann1995},
\begin{equation}\label{p:2-13}
    Q_{0 s} = \Delta t^{q} \int_{t_{\ell}}^{t_{\ell + 1}} \kappa_{0s}\Big(\frac{t}{\Delta t}\Big) \partial_{t}^{(q + 1)}\Tilde{U}_{N, \gamma}^{\boldsymbol{\beta}, \boldsymbol{x}_{0}}(t)\text{d}t,
\end{equation}
\begin{equation}\label{p:2-14}
    \begin{aligned}
        Q_{1} &= \Delta t^{q + 1} \int_{t_{\ell}}^{t_{\ell + 1}}\kappa_{1}\Big(\frac{t}{\Delta t}\Big)\partial_{t}^{(q + 2)} \Tilde{U}_{N, \gamma}^{\boldsymbol{\beta}, \boldsymbol{x}_{0}}(t)\text{d}t\\
        &= -\Delta t^{q} \int_{t_{\ell}}^{t_{\ell + 1}} \kappa^{\prime}_{1}\Big(\frac{t}{\Delta t}\Big) \partial_{t}^{(q + 1)} \Tilde{U}_{N, \gamma}^{\boldsymbol{\beta}, \boldsymbol{x}_{0}}(t)\text{d}t,
    \end{aligned}
\end{equation}
where $q$ is the stage order of the IRK scheme. The Peano kernels $\kappa_{0s}$, $\kappa_{1}$ are uniformly bounded \cite{Ostermann1995}, and we shall denote their upper bound as $C_{p}/2$, \textit{i.e.},
\begin{equation}
    \kappa_{0s}(t) \le C_{p}/2,~\kappa_{1}(t) \le C_{p}/2.
\end{equation}
%Additionally, the errors $Q_{0s}$ and $Q_{1}$ in Eq. \eqref{p:2-14} can be estimated by using the Cauchy inequality \citep[Eq. (1.6)]{Ostermann1995},
Additionally, we can use the Cauchy inequality \citep[Eq. (1.6)]{Ostermann1995} to bound the defects $Q_{0s}$ and $Q_{1}$ in Eqs.~\eqref{p:2-13} and \eqref{p:2-14}:
\begin{equation}\label{p:1-15}
    \begin{aligned}
        \Delta t\sum_{r = 1}^{q}\|Q_{0 r}\|_{H^{1}}^{2} &+ \Delta t\|Q_{1}\|_{H^{1}}^{2} + \Delta t \|Q_{1}/\Delta t\|_{H^{-1}}^{2}\\
        &\le C_{p}^{2}(\Delta t^{q + 1})^{2} \int_{t_{\ell}}^{t_{\ell + 1}} \|\partial_{t}^{(q + 2)} \Tilde{U}_{N, \gamma}^{\boldsymbol{\beta}, \boldsymbol{x}_{0}}\|_{H^{-1}} + \|\partial_{t}^{(q + 1)} \Tilde{U}_{N, \gamma}^{\boldsymbol{\beta}, \boldsymbol{x}_{0}}\|_{H^{1}} \text{d}t\\
        &\le C_{p}^{2}(\Delta t^{q + 1})^{2} \|\partial_{t}^{(q + 1)}\Tilde{U}_{N, \gamma}^{\boldsymbol{\beta}, \boldsymbol{x}_{0}}\|_{X(t_{\ell}, t_{\ell + 1})}^{2}.
    \end{aligned}
\end{equation}

In the following, we leverage Eq. \eqref{p:1-15} to derive the IRK scheme error bound. The IRK scheme error $U_{N, \gamma}^{\boldsymbol{\beta}, \boldsymbol{x}_{0}}(\cdot, t_{\ell + 1}) - \Tilde{U}_{N, \gamma}^{\boldsymbol{\beta}, \boldsymbol{x}_{0}}(\cdot, t_{\ell + 1})$ is denoted by $E_{1}$, the stage variable errors $U_{N}^{0s} - \Tilde{U}_{N}^{0s}$ are denoted by $E_{0s}$. Additionally, we define $E_{0s}^{\prime}$ as:
\begin{equation}\label{p:1-4}
    E_{0s}^{\prime} \coloneqq F(U_{0r}; t_{0s}) - F(\Tilde{U}_{0s}; t_{0s}) - A(E_{0s};t_{0s}).
\end{equation}
By subtracting Eq. \eqref{p:3-1} from Eq. \eqref{B2:1}, we have
\begin{equation}\label{p:1-5}
    E_{0 r} = \Delta t \sum_{s = 1}^{q}a_{RK}^{rs} E_{0s}^{\prime} - Q_{0 r},
\end{equation}
and
\begin{equation}\label{p:1-6}
    E_{1} = \Delta t \sum_{r = 1}^{q}b_{RK}^{r} E_{0r}^{\prime} - Q_{1}.
\end{equation}

We first prove the following lemma to estimate the error of the IRK scheme.
\begin{lemma}\label{th:7-1} The following inequality holds for the IRK scheme,
\begin{equation}\label{p:3-18}
    \begin{aligned}
        \|E_{1}\|^{2}_{L^{2}} &\le - c_{0}\Delta t \sum_{r = 1}^{q}b_{RK}^{r} \|E_{0r}\|^{2}_{H^{1}} + 4L \Delta t \sum_{r = 1}^{q} b_{RK}^{r}\|E_{0r}\|^{2}_{L^{2}}\\
        &\quad + (\Delta t^{q + 1})^{2} \Big(\frac{2C_{0}^{2}}{c_{0}} + L + 1\Big) C_{p}^{2} \big\|\partial_{t}^{(q + 1)} \Tilde{U}_{N, \gamma}^{\boldsymbol{\beta}, \boldsymbol{x}_{0}}\big\|_{X(t_{\ell}, t_{\ell + 1})}^{2}.
    \end{aligned}
\end{equation}
\end{lemma}

\begin{proof}
By Eq. \eqref{p:1-6}, we have
\begin{equation}\label{p:1-3}
    \|E_{1}\|^{2}_{L^{2}} = \Big\|\Delta t \sum_{r = 1}^{q} b_{RK}^{r} E_{0r}^{\prime}\Big\|^{2}_{L^{2}} - 2\Big(\Delta t \sum_{r = 1}^{q}b_{RK}^{r} E_{0 r}^{\prime}, Q_{1}\Big) + \|Q_{1}\|^{2}_{L^{2}}.
\end{equation}

For the first term on the RHS of Eq. \eqref{p:1-3}, by using Eq. \eqref{p:1-5}, we have
\begin{equation}\label{p:3-13}
    \begin{aligned}
        \Big\|\Delta t \sum_{r = 1}^{q}b_{RK}^{r} E_{0 r}^{\prime}\Big\|^{2}_{L^{2}} &= \Delta t^{2}\sum_{r = 1}^{q}\sum_{s = 1}^{q}b_{RK}^{r}b_{RK}^{s}(E_{0 r}^{\prime}, E_{0 s}^{\prime})\\
        &\hspace{0.5cm} + \Delta t \sum_{r = 1}^{q}b_{RK}^{r}(E_{0 r}^{\prime}, E_{0 r} - \Delta t \sum_{s = 1}^{q}a_{RK}^{rs} E_{0s}^{\prime} + Q_{0 r})\\
        &\hspace{1cm} + \Delta t \sum_{r = 1}^{q}b_{RK}^{r}(E_{0 r} - \Delta t \sum_{s = 1}^{q}a_{RK}^{rs} E_{0s}^{\prime} + Q_{0 r}, E_{0 r}^{\prime})\\
        &= 2\Delta t\sum_{r = 1}^{q}b_{RK}^{r}(E_{0 r}^{\prime}, E_{0r} + Q_{0 r})- \Delta t^{2} \sum_{r = 1}^{q}\sum_{s = 1}^{q}\mathcal{M}_{rs}(E_{0 r}^{\prime}, E_{0 s}^{\prime}),
    \end{aligned}
\end{equation}
where $\mathcal{M}\coloneqq (a^{rs}_{RK}b^{r}_{RK} + a^{sr}_{RK}b^{s}_{RK} - b^{r}_{RK}b^{s}_{RK})_{r, s = 1}^{q}\in\mathbb{R}^{q\times q}$. Eq. \eqref{p:3-13} is given by \citep[page 606]{Ostermann1995}. Additionally, since the matrix $\mathcal{M}$ is assumed to be positive semi-definite, we have
% and the Gamma matrix $(E_{0 r}^{\prime}, E_{0 s}^{\prime})_{r,s = 1}^{q}$ is positive semi-definite, we have
\begin{equation}\label{p:1-9}
    \Big\|\Delta t \sum_{r = 1}^{q}b_{RK}^{r} E_{0 r}^{\prime}\Big\|^{2}_{L^{2}} \le 2\Delta t\sum_{r = 1}^{q}b_{RK}^{r}(E_{0 r}^{\prime}, E_{0 r} + Q_{0r}).
\end{equation}

Furthermore, by plugging Eq. \eqref{p:1-4} of $E_{0 r}^{\prime}$ into the RHS of Eq. \eqref{p:1-9}, we have
\begin{equation}\label{p:3-3}
    \begin{aligned}
        (E_{0 r}^{\prime}, E_{0 r} + Q_{0 r}) = \big(F(U_{N}^{0 r}; t_{0r}) - F(\tilde{U}_{N}^{0 r}; t_{0r}), E_{0 r} + Q_{0 r}\big)
         - \big(A(t_{0r})E_{0 r}, E_{0 r} + Q_{0 r}\big).
    \end{aligned}
\end{equation}
As the operator $A(\cdot; t)$ in Eq. \eqref{B1:26} satisfies $\big(A(E_{0 r}; t
), E_{0 r}\big) = a(E_{0 r}, E_{0 r}; t)$, we have
\begin{equation}\label{p:3-17}
        c_{0} \|E_{0 r}\|_{H^{1}}^{2} \le \big(A(E_{0 r};t_{0r}), E_{0 r}\big),\,\,
        \big(A(E_{0 r};t_{0r}), Q_{0r}\big) \le C_{0} \|E_{0 r}\|_{H^{1}} \|Q_{0 r}\|_{H^{1}}.
\end{equation}
For the nonlinear operator $F(u; t)$ defined in Eq. \eqref{B1:26}, we have
\begin{equation}\label{p:1-14}
    \big(F(U_{N}^{0 r}; t_{0r}) - F(\tilde{U}_{N}^{0 r}; t_{0r}), E_{0 r} + Q_{0 r}\big) \le L\|E_{0 r}\|_{L^{2}} \|E_{0r} + Q_{0r}\|_{L^{2}}.
\end{equation}
Combining Eqs. \eqref{p:3-3}, \eqref{p:3-17}, and~\eqref{p:1-14}, we have
\begin{equation}\label{p:3-4}
    \begin{aligned}
        (E_{0 r}^{\prime}, E_{0 r} + Q_{0 r}) \le -c_{0}\|E_{0 r}\|_{H^{1}}^{2} + C_{0}\|E_{0 r}\|_{H^{1}} \|Q_{0 r}\|_{H^{1}}+ L \|E_{0 r}\|_{L^{2}}^{2} + L \|E_{0 r}\|_{L^{2}} \|Q_{0 r}\|_{L^{2}}.
    \end{aligned}
\end{equation}
Finally, using the upper bound for $(E_{0 r}^{\prime}, E_{0 r} + Q_{0 r})$ (Eq. \eqref{p:3-4}) on the RHS of Eq. \eqref{p:1-9}, we have
\begin{equation}
    \begin{aligned}
        \Big\|\Delta t \sum_{r = 1}^{q}b_{RK}^{r} E_{0 r}^{\prime}\Big\|^{2}_{L^{2}} &\le 2\Delta t \sum_{r = 1}^{q}b_{RK}^{r} \big(-c_{0}\|E_{0 r}\|_{H^{1}}^{2} + C_{0}\|E_{0 r}\|_{H^{1}} \|Q_{0 r}\|_{H^{1}}\\
        &\hspace{2cm}+ L \|E_{0 r}\|_{L^{2}}^{2} + L \|E_{0 r}\|_{L^{2}} \|Q_{0 r}\|_{L^{2}}\big).
    \end{aligned}
\label{B2:3}
\end{equation}

For the second term on the RHS of Eq. \eqref{p:1-3}, using the Cauchy inequality, we have
\begin{equation}\label{B2:2}
    2\Big(\Delta t \sum_{r = 1}^{q}b_{RK}^{r} E_{0 r}^{\prime}, Q_{1}\Big) \le 2\Delta t \sum_{r = 1}^{q}b_{RK}^{r}\|E_{0r}^{\prime}\|_{L^{2}}\|Q_{1}\|_{L^{2}}.
\end{equation}
From Eq. \eqref{p:1-4}, we have
\begin{equation}\label{B2:8}
    (E_{0 r}^{\prime}, v) = \big(F(U_{N}^{0 r}; t_{0r}) - F(\tilde{U}_{N}^{0 r}; t_{0r}), v\big) - \big(A(E_{0 r};t_{0r}), v\big),~\forall v \in H^{1}(\mathbb{R}^{d}).
\end{equation}
Furthermore, since
\begin{equation}\label{B2:9}
    \big|\big(F(U_{0 r}) - F(\tilde{U}_{0 r}), v\big) - \big(A(E_{0 r};t_{0r}), v\big)\big| \le L\|E_{0r}\|_{L^{2}}\|v\|_{L^{2}} + C_{0}\|E_{0r}\|_{H^{1}}\|v\|_{L^{2}},
\end{equation}
we can show that
\begin{equation}
    \big|(E_{0 r}^{\prime}, v)\big|
    \le L\|E_{0r}\|_{L^{2}}\|v\|_{L^{2}} + C_{0}\|E_{0r}\|_{H^{1}}\|v\|_{L^{2}},
\end{equation}
and therefore, we have
\begin{equation}\label{B2:13}
    \begin{aligned}
        \|E_{0 r}^{\prime}\|_{H^{-1}} = \max_{v \in H^{1}} \frac{|(E_{0 r}^{\prime}, v)|}{\|v\|_{H^{1}}} &\le \frac{C_{0}\|E_{0 r}\|_{H^{1}} \|v\|_{H^{1}} + L\|E_{0 r}\|_{L^{2}}\|v\|_{L^{2}}}{\|v\|_{H^{1}}}\\
        &\le C_{0} \|E_{0 r}\|_{H^{1}} + L\|E_{0 r}\|_{L^{2}}.
    \end{aligned}
\end{equation}
Combining the upper bound of $\|E_{0r}^{\prime}\|_{H^{-1}}$ (Eq. \eqref{B2:13}) with Eq. \eqref{B2:2}, we have
\begin{equation}\label{p:3-6}
    \begin{aligned}
        \Big(\Delta t \sum_{r = 1}^{q}b_{RK}^{r} E_{0 r}^{\prime}, Q_{1}\Big) &\le \Delta t \sum_{r = 1}^{q}b_{RK}^{r} \|E_{0 r}^{\prime}\|_{H^{-1}} \|Q_{1}\|_{H^{1}}\\
        &\le \Delta t C_{0} \sum_{r = 1}^{q}b_{RK}^{r} \|E_{0 r}\|_{H^{1}}\|Q_{1}\|_{H^{1}}+ \Delta t L\sum_{r = 1}^{q} b_{RK}^{r} \|E_{0r}\|_{L^{2}}\|Q_{1}\|_{H^{1}}.
    \end{aligned}
\end{equation}

Considering the upper bounds (Eqs. \eqref{B2:3} and \eqref{p:3-6}) for the first two terms on the RHS of Eq. \eqref{p:1-3}, we can show that
\begin{equation}\label{p:3-8}
    \begin{aligned}
        \|E_{1}\|^{2}_{L^{2}} &\le 2\Delta t \sum_{r = 1}^{q}b_{RK}^{r}\big(-c_{0} \|E_{0 r}\|^{2}_{H^{1}} + C_{0} \|E_{0 r}\|_{H^{1}}\|Q_{0 r}\|_{H^{1}}\\
        &\quad + L\|E_{0 r}\|_{L^{2}}^{2} + L\|E_{0 r}\|_{L^{2}}\|Q_{0 r}\|_{L^{2}}\big) + 2\Delta t C_{0}\sum_{r = 1}^{q}b_{RK}^{r}\|E_{0 r}\|_{H^{1}}\|Q_{1}\|_{H^{1}}\\
        &\quad\quad  + 2\Delta t L\sum_{r = 1}^{q}b_{RK}^{r}\|E_{0 r}\|_{L^{2}}\|Q_{1}\|_{H^{1}} + \|Q_{1}\|^{2}_{L^{2}}.
    \end{aligned} 
\end{equation}
Applying the Cauchy inequality in Eq. \eqref{p:3-8}, we have
\begin{equation}\label{B2:14}
    \begin{aligned}
       & \|E_{1}\|^{2}_{L^{2}} + 2\Delta t \sum_{r = 1}^{q}b_{RK}^{r} c_{0}\|E_{0 r}\|_{H^{1}}^{2} \le \|Q_{1}\|^{2}_{L^{2}} + \Delta t \sum_{r = 1}^{q}b_{RK}^{r}\Big(2L\|E_{0r}\|_{L^{2}}^{2}\\
        &\hspace{1cm}+ \frac{c_{0}}{2}\|E_{0r}\|_{H^{1}}^{2} + \frac{2C_{0}^{2}}{c_{0}} \|Q_{0r}\|_{H^{1}}^{2}+ L\|E_{0r}\|_{L^{2}}^{2} + L\|Q_{0r}\|_{L^{2}}^{2}\\
       &\hspace{1.5cm}+ \frac{c_{0}}{2}\|E_{0r}\|_{H^{1}}^{2} + \frac{2C_{0}^{2}}{c_{0}} \|Q_{1}\|_{H^{1}}^{2} + L\|E_{0r}\|_{L^{2}}^{2} + L\|Q_{1}\|_{L^{2}}^{2}\Big).
    \end{aligned}
\end{equation}
Furthermore, 
\begin{equation}\label{B2:15}
    \|Q_{1}\|^{2}_{L^{2}} \le \Delta t \|Q_{1}\|_{H^{1}}^{2} + \frac{\Delta t}{2} \|Q_{1}/\Delta t\|_{H^{-1}}^{2}.
\end{equation}
Thus, by replacing the first $\|Q_{1}\|^{2}_{L^{2}}$ on the RHS of Eq. \eqref{B2:14} with its upper bound in Eq. \eqref{B2:15}, we have
\begin{equation}
    \begin{aligned}
        \|E_{1}\|^{2}_{L^{2}} &+ 2c_{0}\Delta t\sum_{r = 1}^{q}b_{RK}^{r}\|E_{0 r}\|_{H^{1}}^{2}\\
        &\le c_{0}\Delta t\sum_{r = 1}^{q}b_{RK}^{r}\|E_{0 r}\|_{H^{1}}^{2} + 4L \Delta t \sum_{r = 1}^{q}b_{RK}^{r}\|E_{0 r}\|^{2}_{L^{2}}\\
        &\quad + \Big(\frac{2C_{0}^{2}}{c_{0}} + L + 1\Big)\Delta t \sum_{r = 1}^{q}b_{RK}^{r}\big(\|Q_{0 r}\|_{H^{1}}^{2} + \|Q_{1}\|_{H^{1}}^{2} + \|Q_{1}/\Delta t\|_{H^{-1}}^{2}\big).
    \end{aligned}
\label{p:3-15}
\end{equation}
$\sum_{r = 1}^{q}b_{RK}^{r} = 1$ holds for all the IRK schemes \cite{Lubich2016}. Additionally, we have assumed that $b_{RK}^{r} \ge 0$. Thus, we have
\begin{equation}\label{p:3-19}
    \begin{aligned}
        \sum_{r = 1}^{q} b_{RK}^{r}\big(\|Q_{0 r}\|_{H^{1}}^{2} + \|Q_{1}\|_{H^{1}}^{2} + \|Q_{1}/\Delta t\|_{H^{-1}}^{2}\big)&= \sum_{r = 1}^{q}b_{RK}^{r}\|Q_{0r}\|_{H^{1}}^{2}  + \|Q_{1}\|_{H^{1}}^{2} + \|Q_{1}/\Delta t\|_{H^{-1}}^{2}\\
        &\le \sum_{r = 1}^{q}\|Q_{0r}\|_{H^{1}}^{2} + \|Q_{1}\|_{H^{1}}^{2} + \|Q_{1}/\Delta t\|_{H^{-1}}^{2}.
    \end{aligned}
\end{equation}
By plugging in the upper bounds for the summation of squared $H_1$ norms of $Q_{0r}$ and $Q_{1}$ given in Eq. \eqref{p:1-15} into Eq. \eqref{p:3-15}, we have
\begin{equation}\label{p:1-10}
    \begin{aligned}
        \|E_{1}\|^{2}_{L^{2}} & \le -c_{0} \Delta t \sum_{r = 1}^{q} b_{RK}^{r}\|E_{0 r}\|_{H^{1}}^{2} + 4L \Delta t \sum_{r = 1}^{q} b_{RK}^{r} \|E_{0 r}\|^{2}_{L^{2}}\\
        &\quad + \Big(\frac{2C_{0}^{2}}{c_{0}} + L + 1\Big)\Delta t \Big(\sum_{r = 1}^{q}\|Q_{0r}\|_{H^{1}}^{2} + \|Q_{1}\|_{H^{1}}^{2} + \|Q_{1}/\Delta t\|_{H^{-1}}^{2}\Big)\\
        & \le -c_{0} \Delta t \sum_{r = 1}^{q} b_{RK}^{r}\|E_{0 r}\|_{H^{1}}^{2} + 4L \Delta t \sum_{r = 1}^{q} b_{RK}^{r} \|E_{0 r}\|^{2}_{L^{2}}\\
        &\quad + (\Delta t^{q + 1})^{2}\Big(\frac{2C_{0}^{2}}{c_{0}} + L + 1\Big) C_{p}^{2}\big\|\partial_{t}^{(q + 1)} \Tilde{U}_{N, \gamma}^{\boldsymbol{\beta}, \boldsymbol{x}_{0}}\big\|_{X(t_{\ell}, t_{\ell + 1})}^{2},
    \end{aligned} 
\end{equation}
which proves Lemma~\ref{th:7-1}.
\end{proof}

From Lemma~\ref{th:7-1}, the error bound of the IRK scheme is readily available when $4L \le c_{0}$, and in this case, we have
\begin{equation}
    \|E_{1}\|^{2}_{L^{2}} \le (\Delta t^{q + 1})^{2}\Big(\frac{2C_{0}^{2}}{c_{0}} + L + 1\Big) C_{p}^{2}\big\|\partial_{t}^{(q + 1)} \Tilde{U}_{N, \gamma}^{\boldsymbol{\beta}, \boldsymbol{x}_{0}}\big\|_{X(t_{\ell}, t_{\ell + 1})}^{2}.
\end{equation}
In the following, we consider the case where $c_{0} < 4L$. Based on Lemma \ref{th:7-1}, we give the error bound of the IRK scheme by eliminating the stage variable errors $E_{0r}$.

Applying the Cauchy inequality in Eq. \eqref{p:1-5}, we have
\begin{equation}\label{B2:4}
    \begin{aligned}
        \|E_{0r}\|_{H^{-1}}^{2} = \Big\|\Delta t\sum_{s = 1}^{q}a_{RK}^{rs}E_{0s}^{\prime} - Q_{0r}\Big\|_{H^{-1}}^{2}&\le 2\Delta t^{2} \Big(\sum_{s = 1}^{q}a_{RK}^{rs}\|E_{0s}^{\prime}\|_{H^{-1}}\Big)^{2} + 2\|Q_{0r}\|_{H^{-1}}^{2}\\
        &\le 2C_{ab, r}^{2} \Delta t^{2} \sum_{s = 1}^{q}b_{RK}^{s}\|E_{0s}^{\prime}\|_{H^{-1}}^{2} + 2\|Q_{0r}\|_{H^{-1}}^{2}.
    \end{aligned}
\end{equation}
where $C_{ab, r}^{2} = \sum_{s = 1}^{q}\big(a_{RK}^{rs}\big)^{2}/b_{RK}^{s}$. Furthermore, by replacing the stage variable errors $\|E_{0r}\|_{H^{-1}}$ with their upper bounds given in Eq. \eqref{B2:4}, we have
\begin{equation}\label{p:3-12}
    \begin{aligned}
        \|E_{0r}\|^{2}_{L^{2}} & = (E_{0 r}, E_{0 r})\le \frac{1}{2\tau_{0}}\|E_{0 r}\|_{H^{1}}^{2} + \frac{\tau_{0}}{2}\|E_{0 r}\|_{H^{-1}}^{2}\\
        & \le \frac{1}{2\tau_{0}}\|E_{0 r}\|_{H^{1}}^{2} + \tau_{0}\Big(C_{ab, r}^{2} \Delta t^{2} \sum_{s = 1}^{q}b_{RK}^{s}\|E_{0s}^{\prime}\|_{H^{-1}}^{2} + \|Q_{0r}\|_{H^{-1}}^{2}\Big).
    \end{aligned}
\end{equation}
From Eq. \eqref{B2:13}, we have
\begin{equation}\label{B2:5}
    \|E_{0s}^{\prime}\|_{H^{-1}}^{2} \le (C_{0} + L)^{2} \|E_{0s}\|_{H^{1}}^{2}.
\end{equation}
By replacing $\|E_{0s}^{\prime}\|_{H^{-1}}^{2}$ on the RHS of Eq. \eqref{p:3-12} with its upper bound in Eq. \eqref{B2:5}, we have the inequality
\begin{equation}\label{B2:6}
    \begin{aligned}
        \|E_{0r}\|^{2}_{L^{2}} &\le \frac{1}{2\tau_{0}}\|E_{0r}\|_{H^{1}}^{2} + \tau_{0} (C_{0} + L)^{2}C_{ab, r}^{2}\Delta t^{2} \sum_{s = 1}^{q}b_{RK}^{s} \|E_{0s}\|_{H^{1}}^{2}+ \tau_{0} \|Q_{0r}\|_{H^{-1}}^{2}.
    \end{aligned}
\end{equation}
Furthermore, by replacing $\|E_{0s}\|_{L^{2}}^2$ on the RHS of Eq. \eqref{p:3-18} with its upper bound given in Eq. \eqref{B2:6}, we have
\begin{equation}\label{B2:12}
    \begin{aligned}
        \|E_{1}\|^{2}_{L^{2}} &\le - c_{0}\Delta t \sum_{r = 1}^{q}b_{RK}^{r} \|E_{0r}\|^{2}_{H^{1}} + 4L\Delta t \sum_{r = 1}^{q}b_{RK}^{r}\big(\frac{1}{2\tau_{0}}\|E_{0r}\|_{H^{1}}^{2}\\
        &\quad + \tau_{0} (C_{0} + L)^{2}C_{ab, r}^{2}\Delta t^{2} \sum_{s = 1}^{q}b_{RK}^{s} \|E_{0s}\|_{H^{1}}^{2} + \tau_{0} \|Q_{0r}\|_{H^{-1}}^{2}\big)\\
        &\quad + (\Delta t^{q + 1})^{2} \Big(\frac{2C_{0}^{2}}{c_{0}} + L + 1\Big) C_{p}^{2} \big\|\partial_{t}^{(q + 1)} \Tilde{U}_{N, \gamma}^{\boldsymbol{\beta}, \boldsymbol{x}_{0}}\big\|_{X(t_{\ell}, t_{\ell + 1})}^{2}.
    \end{aligned}
\end{equation}

We plug in the value $\tau_{0} = 4L/c_{0}$ into Eq. \eqref{B2:12} to get
\begin{equation}\label{B2:16}
    \begin{aligned}
        &\|E_{1}\|^{2}_{L^{2}}  \le -\frac{c_{0}}{2}\Delta t \sum_{r = 1}^{q}b_{RK}^{r} \|E_{0r}\|^{2}_{H^{1}}\\
        &\hspace{1cm} + \frac{16L^{2}}{c_{0}} (C_{0} + L)^{2}C_{ab, r}^{2}\Delta t^{3} \sum_{s = 1}^{q}b_{RK}^{s} \|E_{0s}\|_{H^{1}}^{2} + \frac{16L^{2}}{c_{0}}\Delta t \sum_{r = 1}^{q}b_{RK}^{r} \|Q_{0r}\|_{H^{-1}}^{2}\\
        &\hspace{1cm}  + (\Delta t^{q + 1})^{2} \Big(\frac{2C_{0}^{2}}{c_{0}} + L + 1\Big) C_{p}^{2} \big\|\partial_{t}^{(q + 1)} \Tilde{U}_{N, \gamma}^{\boldsymbol{\beta}, \boldsymbol{x}_{0}}\big\|_{X(t_{\ell}, t_{\ell + 1})}^{2}.
    \end{aligned}
\end{equation}
Additionally, by using Eq. \eqref{p:1-15}, we can prove that
\begin{equation}\label{B2:17}
    \begin{aligned}
        \Delta t \sum_{r = 1}^{q}b_{RK}^{r} \|Q_{0r}\|_{H^{-1}}^{2} &\le \Delta t \sum_{r = 1}^{q}b_{RK}^{r} \|Q_{0r}\|_{H^{1}}^{2}\\
        &\le \Delta t \sum_{r = 1}^{q} \big(\|Q_{0r}\|_{H^{1}}^{2} + \|Q_{1}\|_{H^{1}}^{2} + \|Q_{1}/\Delta t\|_{H^{-1}}^{2}\big)\\
        &\le (\Delta t^{q + 1})^{2}  C_{p}^{2} \big\|\partial_{t}^{(q + 1)} \Tilde{U}_{N, \gamma}^{\boldsymbol{\beta}, \boldsymbol{x}_{0}}\big\|_{X(t_{\ell}, t_{\ell + 1})}^{2}.
    \end{aligned}
\end{equation}
Finally, by replacing the $\|Q_{0r}\|_{H^{-1}}$ in Eq. \eqref{B2:16} with its upper bound given in \eqref{B2:17}, we have
\begin{equation}\label{p:1-13}
    \begin{aligned}
        \|E_{1}\|^{2}_{L^{2}} &\le C_{R}^{2}C_{p}^{2} (\Delta t^{q + 1})^{2}\|\partial_{t}^{(q + 1)} \Tilde{U}_{N, \gamma}^{\boldsymbol{\beta}, \boldsymbol{x}_{0}}\|_{X(t_{\ell}, t_{\ell + 1})}^{2} \\
        &\quad + \Big(\frac{16L^2}{c_{0}} (C_{0} + L)^{2}C_{ab}^{2}\Delta t^{2} - \frac{c_{0}}{2}\Big) \Delta t\sum_{r = 1}^{q}b_{RK}^{r}\|E_{0r}\|_{H^{1}}^{2},
    \end{aligned}
\end{equation}
where 
\begin{equation}
    C_{R}^{2} \coloneqq \frac{2C_{0}^{2} + 32L^{2}}{c_{0}} + L + 1 \text{ and } C_{ab}^{2} = \sum_{r = 1}^{q}b_{RK}^{r}C_{ab, r}^{2} = \sum_{s = 1}^{q}\sum_{r = 1}^{q} \big(a_{RK}^{rs}\big)^{2}b_{RK}^{r}/b_{RK}^{s}.
\end{equation}
Finally, if
\begin{equation}
    \Delta t \le \frac{c_{0}}{4\sqrt{2}L(C_{0} + L)C_{ab}},
\label{cond_t}
\end{equation}
we have
\begin{equation}\label{p:1-17}
    \frac{16L^2}{c_{0}} (C_{0} + L)^{2}C_{ab}^{2}\Delta t^{2} - \frac{c_{0}}{2} \le 0.
\end{equation}
Thus, if the condition Eq.~\eqref{cond_t} on the time step $\Delta t$ is satisfied, the last term of Eq. \eqref{p:1-13} is nonpositive, and we can conclude that
\begin{equation}
    \big\|U_{N, \gamma}^{\boldsymbol{\beta}, \boldsymbol{x}_{0}}(\cdot, t_{\ell + 1}) - \Tilde{U}_{N, \gamma}^{\boldsymbol{\beta}, \boldsymbol{x}_{0}}(\cdot, t_{\ell + 1})\big\|_{L^{2}} \le C_{R}C_{p} \Delta t^{q + 1}\|\partial_{t}^{(q + 1)} \Tilde{U}_{N, \gamma}^{\boldsymbol{\beta}, \boldsymbol{x}_{0}}\|_{X(t_{\ell}, t_{\ell + 1})}.
\end{equation}

\section{Source terms}\label{appendix:1}
In this part, we shall give the source terms on the RHS of 
Eqs.~\eqref{exam:1}, \eqref{exam:23}, \eqref{exam:35}, and \eqref{exam:33} in Examples \ref{example1}, \ref{example3}, \ref{example4}, and \ref{example5}, respectively. First, we define
\begin{equation}
    A_{s}^{\gamma} \coloneqq \frac{2^{2s} \Gamma(s + \gamma) \Gamma(s + 1/2)}{\sqrt{\pi} \Gamma(\gamma)}, 
\end{equation}
where $\Gamma(\cdot)$ is the Gamma function. Additionally, we denote the hypergeometric function of the first kind (see \cite{Tang2018}) as 
\begin{equation}
    _{1}F_{1}(a, b; x) = \sum_{k = 0}^{\infty}\frac{(a)_{k}}{(b)_{k}}\frac{x^{k}}{k!},
\end{equation}
where $(a)_{k}$ denotes the rising Pochhammer symbol
\begin{equation}
    (a)_{0} = 1,~ (a)_{k} = \frac{\Gamma(a + k)}{\Gamma(a)}.
\end{equation}
We denote the hypergeometric function of the second kind (see \cite{Tang2019}) as
\begin{equation}
    _{2}F_{1}(a, b; c; x) = \sum_{k = 0}^{\infty}\frac{(a)_{k}(b)_{k}}{(c)_{k}} \frac{x^{k}}{k!}.
\end{equation}

~\\
\textbf{Example \ref{example1}}
\begin{equation}\label{f1}
    \begin{aligned}
        f(x, t) &= \frac{12}{(t + 1)^{5}} \frac{z^{2}}{(1 + z^2)^7} + \frac{(1 + z^2)^{12} - 1}{(1 + z^2)^{18}}\\
        &\quad + \frac{1}{t + 1}\frac{1}{(1 + z^2)^{13/2}}A_{1/2}^{6}~_{2}F_{1}\Big(\frac{13}{2}, -\frac{1}{2}; \frac{1}{2}, \frac{z^2}{1 + z^2}\Big),
    \end{aligned}
\end{equation}
where
\begin{equation}
    z = \frac{x}{t + 1}.
\end{equation}

~\\
\textbf{Example \ref{example3}}
\begin{equation}\label{f3}
    \begin{aligned}
        f(x, y, t) &= \frac{1}{t + 1} \frac{z^2}{(1 + z^2)^{8}} + \frac{1}{t + 1}\frac{1}{(1 + z^2)^7}\Big(1 - \frac{1}{t + 1}\frac{1}{(1 + z^2)^7}\Big)\\
        &\quad + \frac{1}{t + 1}A^{7}_{1/2}~_{2}F_{1}\Big(\frac{15}{2}, -\frac{1}{2}; \frac{1}{2}; \frac{z^2}{1 + z^2}\Big),
    \end{aligned}
\end{equation}
where
\begin{equation}
    z = \frac{1}{t + 1}\left(\big(x - \cos(\tfrac{\pi}{3})t\big)^{2} + \big(y - \sin(\tfrac{\pi}{3}) t\big)^{2}\right)^{1/2}.
\end{equation}

~\\
\textbf{Example \ref{example4}}
\begin{equation}\label{f5}
    \begin{aligned}
        f(x, y, z; t) &= \frac{\sin(x + 6y/5 + z/2)}{(3t + 1)^{3/2}}\left(\frac{6r^{2}}{(6t + 2)^{2}} - \frac{9}{6t + 2}\right)\exp\left(\frac{-r^{2}}{6t + 2}\right)\\
        &\quad+ (-\Delta)^{3/4}\left(\frac{\sin\big(x + 6y/5 + z/2\big)}{(3t + 1)^{3/2}} \exp\left(\frac{-r^{2}}{6t + 2}\right)\right),
    \end{aligned}
\end{equation}
where 
\begin{equation}
    r^{2} = x^{2} + y^{2} + z^{2}.
\end{equation}
%In this example, since the solution $u$ of Eq. \eqref{exam:5} is anisotropic, the fractional Laplacian term $(-\Delta)^{3/4}u$ can not be express analytically. We use the mapped Jacobi method to compute the fractional Laplacian term $(-\Delta)^{3/4}u$ with a relatively large expansion order, \textit{i.e.} $N = 100$.

~\\
\textbf{Example \ref{example5}}
\begin{equation}\label{f4}
    \begin{aligned}
        f(x, y, z, w,t) &=  \left(\frac{4\cos(p)}{(t + 1)^{2}} - \frac{2 p\sin(p)}{(t + 1)^3}\right)\exp\left(\frac{-r^{2}}{2(t + 1)}\right)\\
        &\quad + \frac{r^{2}\cos(p)}{(t + 1)^{2}} \exp\left(\frac{-r^{2}}{t + 1}\right),
    \end{aligned}
\end{equation}
where
\begin{equation}
    p = x + y + z + w\text{ and } r^{2} = x^{2} + y^{2} + z^{2} + w^{2}.
\end{equation}

\section{Implementation details of adaptive techniques}\label{appendix:2}
In this section, we discuss the implementation details of adaptive techniques in our AHMJ method. We first provide the overview of the AHMJ method in Fig. \ref{fig:flow chart}. Next, we present the algorithms for the moving, scaling, and $p$-adaptive techniques in Algorithms \ref{algorithm1}, \ref{algorithm2} and \ref{algorithm3}. Hyperparameters of the adaptive techniques used in all examples are in Table \ref{table:2}.
\newpage
\begin{figure}[h!]
    \centering
    \includegraphics[width = 4in]{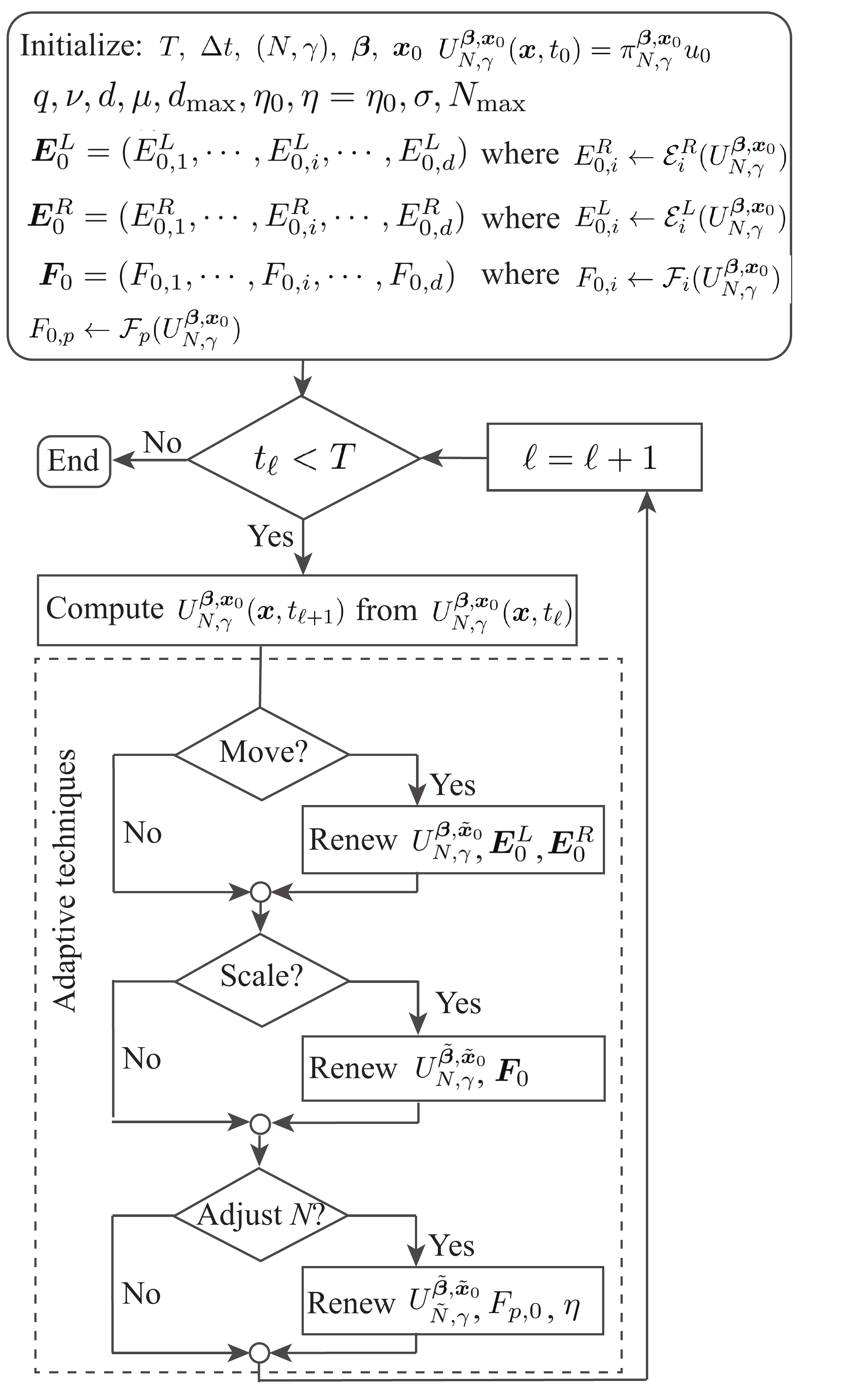}
    \caption{\footnotesize{\textnormal{The overview of the implementation of our proposed AHMJ method. $\tilde{\boldsymbol{x}}_{0}$, $\tilde{\boldsymbol{\beta}}_{0}$, and $\tilde{N}$ are updated by using the moving, scaling, and $p$-adaptive techniques in Algorithms \ref{algorithm1}, \ref{algorithm2} and \ref{algorithm3}, respectively.}}}
    \label{fig:flow chart}
\end{figure}

\begin{algorithm}
    \caption{The moving technique} \label{algorithm1}
    \begin{algorithmic}
        \STATE{\textbf{input} $U_{N, \gamma}^{\boldsymbol{\beta}, \boldsymbol{x}_{0}}$, $\mu$, $\delta$, $(d_{\text{max}, 1}, \cdots, d_{\text{max}, i}, \cdots, d_{\text{max}, d})$\\
        \quad $\boldsymbol{E}_{0}^{L} = (E_{0, 1}^{L}, \cdots, E_{0, i}^{L}, \cdots, E_{0, d}^{L})$, and $\boldsymbol{E}_{0}^{R} = (E_{0, 1}^{R}, \cdots, E_{0, i}^{R}, \cdots, E_{0, d}^{R})$.}
        \STATE{~}
        \FOR{$i = 1 : d$}
        \STATE{compute $E_{i}^{L} \gets \mathcal{E}_{i}^{L}(U_{N, \gamma}^{\boldsymbol{\beta}, \boldsymbol{x}_{0}})$}
        \STATE{$x_{i}^{L} \gets x_{0, i}$}
        \WHILE{$E_{i}^{L} > \mu E_{0, i}^{L}$ and $\big|x_{i}^{L}\big| \le d_{\max, i}$}
            \STATE{$x_{i}^{L} \gets x_{i}^{L} - \delta$}
            \STATE{$\Tilde{U}_{x_{0, i}} \gets \pi_{N, \gamma}^{\boldsymbol{\beta}, \tilde{\boldsymbol{x}}_{0}}U_{N, \gamma}^{\boldsymbol{\beta}, \boldsymbol{x}_{0}}$ where $\tilde{\boldsymbol{x}}_{0} = (x_{1}, \cdots, x_{i}^{L}, \cdots, x_{d})$}
            \STATE{compute $E_{i}^{L} \gets \mathcal{E}_{i}^{L}(\Tilde{U}_{x_{0, i}})$}
        \ENDWHILE
        \STATE{~}
        \STATE{compute $E_{i}^{R} \gets \mathcal{E}_{i}^{R}(U_{N, \gamma}^{\boldsymbol{\beta}, \boldsymbol{x}_{0}})$}
        \STATE{$x^{R}_{i} \gets x_{0, i}$}
        \WHILE{$E_{i}^{R} > \mu E_{0, i}^{R}$ and $\big|x^{R}_{i}\big| \le d_{\max, i}$}
            \STATE{$x^{R}_{i} \gets x^{R}_{i} + \delta$}
            \STATE{$\Tilde{U}_{x_{0, i}} \gets \pi^{\boldsymbol{\beta}, \tilde{\boldsymbol{x}}_{0}}_{N, \gamma}U_{N, \gamma}^{\boldsymbol{\beta}, \boldsymbol{x}_{0}}$ where $\tilde{\boldsymbol{x}}_{0} = (x_{1}, \cdots, x_{i}^{R}, \cdots, x_{d})$}
            \STATE{compute $E_{i}^{R} \gets \mathcal{E}_{i}^{R}(\Tilde{U}_{x_{0, i}})$}
        \ENDWHILE
        \STATE{~}
        \STATE{set $\Tilde{x}_{0, i} = x_{i}^{R} + x_{i}^{L} - x_{0, i}$}
        \ENDFOR
        \STATE{~}
        \STATE{update $\tilde{\boldsymbol{x}}_{0} \gets \big(\Tilde{x}_{0,1}, \cdots, \Tilde{x}_{0, i},\cdots, \Tilde{x}_{0, d}\big)$}
        \RETURN moved solution $U_{N, \gamma}^{\boldsymbol{\beta}, \Tilde{\boldsymbol{x}}_{0}} \gets \pi_{N, \gamma}^{\boldsymbol{\beta}, \tilde{\boldsymbol{x}}_{0}}U_{N, \gamma}^{\boldsymbol{\beta}, \boldsymbol{x}_{0}}$
        \FOR{$i = 1 : d$}
        \IF{the displacement has been changed in the $i^{\text{th}}$ direction}
            \STATE{update $E_{0, i}^{L} \gets \mathcal{E}^{L}_{i}(U_{N, \gamma}^{\boldsymbol{\beta}, \Tilde{\boldsymbol{x}}_{0}})$}
            \STATE{update $E_{0, i}^{R} \gets \mathcal{E}_{i}^{R}(U_{N, \gamma}^{\boldsymbol{\beta}, \Tilde{\boldsymbol{x}}_{0}})$}
        \ENDIF
\ENDFOR
\RETURN solution with updated $\tilde{\boldsymbol{x}}_0$: $U_{N, \gamma}^{\boldsymbol{\beta}, \Tilde{\boldsymbol{x}}_{0}} \gets \pi_{N, \gamma}^{\boldsymbol{\beta}, \tilde{\boldsymbol{x}}_{0}}U_{N, \gamma}^{\boldsymbol{\beta}, \Tilde{\boldsymbol{x}}_{0}}$
    \end{algorithmic}
\end{algorithm}

\begin{algorithm}
    \caption{The scaling technique} \label{algorithm2}
    \begin{algorithmic}
        \STATE{\textbf{input} $U_{N, \gamma}^{\boldsymbol{\beta}, \Tilde{\boldsymbol{x}}_{0}}$, $\boldsymbol{F}_{0} = (F_{0, 1}, \cdots, F_{0, i}, \cdots, F_{0, d}), \nu$}
        
        \STATE{~}
        \FOR{$i = 1 : d$}
        \STATE{compute $F_{i} \gets \mathcal{F}_{i}(U_{N, \gamma}^{\boldsymbol{\beta}, \Tilde{\boldsymbol{x}}_{0}})$}
        \IF{$F_{i} > \nu F_{0, i}$}
            \STATE{ $\Tilde{U}_{\beta_{i}} \gets \pi_{N, \gamma}^{\Tilde{\boldsymbol{\beta}}, \tilde{\boldsymbol{x}_{0}}}U_{N, \gamma}^{\boldsymbol{\beta}, \Tilde{\boldsymbol{x}}_{0}}$ where $\Tilde{\boldsymbol{\beta}} = (\beta_{1}, \cdots, \beta_{i} q, \cdots, \beta_{d})$}
            \STATE{compute $\Tilde{F}_{i} \gets \mathcal{F}_{i}(\Tilde{U}_{\beta_{i}})$}
            \IF{$\Tilde{F}_{i} < F_{i}$}
                \STATE{$\Tilde{\beta}_{i} \gets \beta_{i}$}
                \WHILE{$\Tilde{F}_{i} < F_{i}$}
                    \STATE{$\Tilde{\beta}_{i} \gets \Tilde{\beta}_{i}q$}
                    \STATE{$F_{0, i} \gets F_{i}$}
                    \STATE{$F_{i} \gets \Tilde{F}_{i}$}
                    \STATE{ $\Tilde{U}_{\beta_{i}} \gets \pi_{N, \gamma}^{\Tilde{\boldsymbol{\beta}}, \tilde{\boldsymbol{x}_{0}}}\Tilde{U}_{\beta_{i}}$ where $\Tilde{\boldsymbol{\beta}} = (\beta_{1}, \cdots, \Tilde{\beta}_{i} q, \cdots, \beta_{d})$}
                    \STATE{compute $\Tilde{F}_{i} \gets \mathcal{F}_{i}(\Tilde{U}_{\beta_{i}})$}
                \ENDWHILE
            \ELSE
                \STATE{$\Tilde{U}_{\beta_{i}} \gets \pi_{N, \gamma}^{\Tilde{\boldsymbol{\beta}}, \tilde{\boldsymbol{x}_{0}}}U_{N, \gamma}^{\boldsymbol{\beta}, \Tilde{\boldsymbol{x}}_{0}}$  where $\Tilde{\boldsymbol{\beta}} = (\beta_{1}, \cdots, \beta_{i}/q, \cdots, \beta_{d})$}
                \STATE{compute $\Tilde{F}_{i} \gets \mathcal{F}_{i}(\Tilde{U}_{\beta_{i}})$}
                \STATE{$\Tilde{\beta}_{i} \gets \beta_{i}$}
                \WHILE{$\Tilde{F}_{i} < F_{i}$}
                    \STATE{$\Tilde{\beta}_{i} \gets \Tilde{\beta}_{i}/q$}
                    \STATE{$F_{0, i} \gets F_{i}$}
                    \STATE{$F_{i} \gets \Tilde{F}_{i}$}
                    \STATE{$\Tilde{U}_{\beta_{i}} \gets \pi_{N, \gamma}^{\Tilde{\boldsymbol{\beta}}, \tilde{\boldsymbol{x}_{0}}}\Tilde{U}_{\beta_{i}}$ where $\Tilde{\boldsymbol{\beta}} = (\beta_{1}, \cdots, \Tilde{\beta}_{i}/q, \cdots, \beta_{d})$}
                    \STATE{compute $\Tilde{F}_{i} \gets \mathcal{F}_{i}(\Tilde{U}_{\beta_{i}})$}
                \ENDWHILE
            \ENDIF
        \ENDIF
        \ENDFOR
        \STATE{~}
        \STATE{update $\Tilde{\boldsymbol{\beta}} = (\Tilde{\beta}_{1}, \cdots, \Tilde{\beta}_{i}, \cdots, \Tilde{\beta}_{d})$}
        \RETURN solution with updated $\tilde{\boldsymbol{\beta}}$: $U_{N, \gamma}^{\Tilde{\boldsymbol{\beta}}, \boldsymbol{x}_{0}} \gets \pi_{N, \gamma}^{\Tilde{\boldsymbol{\beta}}, \boldsymbol{x}_{0}}U_{N, \gamma}^{\boldsymbol{\beta}, \boldsymbol{x}_{0}}$
    \end{algorithmic}
\end{algorithm}

\begin{algorithm}
    \caption{The $p$-adaptive technique} \label{algorithm3}
    \begin{algorithmic}
        \STATE{\textbf{input} $U_{N, \gamma}^{\Tilde{\boldsymbol{\beta}}, \Tilde{\boldsymbol{x}}_{0}}$, $F_{p, 0}$, $\eta$, $\eta_{0}$, $\sigma$, $N_{\text{max}}$}
        \STATE{~}
        \STATE{compute $F_{p} \gets \mathcal{F}_{p}(U_{N, \gamma}^{\Tilde{\boldsymbol{\beta}}, \Tilde{\boldsymbol{x}}_{0}})$}
        \IF{$F_{p} > \eta F_{p, 0}$}
            \STATE{$\Tilde{N} \gets N$}
            \WHILE{$F_{p} > \eta F_{p, 0}$ and $\Tilde{N} \le N + N_{\max}$}
                \STATE{$\Tilde{N} \gets \Tilde{N} + 1$}
                \STATE{$\Tilde{U}_{N} \gets \pi_{\Tilde{N}, \gamma}^{\tilde{\boldsymbol{\beta}}, \tilde{\boldsymbol{x}}_{0}}U_{N, \gamma}^{\Tilde{\boldsymbol{\beta}}, \Tilde{\boldsymbol{x}}_{0}}$}
                \STATE{compute $F_{p} \gets \mathcal{F}(\Tilde{U}_{N})$}
            \ENDWHILE
        \ELSIF{$F_{p} < F_{p, 0}/\eta_{0} $}
            \STATE{$\Tilde{N} \gets N$}
            \WHILE{$F_{p} < F_{p,0}$}
                \STATE{$\Tilde{N} \gets \Tilde{N} - 1$}
                \STATE{$\Tilde{U}_{N} \gets \pi_{\Tilde{N}, \gamma}^{\tilde{\boldsymbol{\beta}}, \tilde{\boldsymbol{x}}_{0}}U_{N, \gamma}^{\Tilde{\boldsymbol{\beta}}, \Tilde{\boldsymbol{x}}_{0}}$}
                \STATE{compute $F_{p} \gets \mathcal{F}(\Tilde{U}_{N})$}
            \ENDWHILE
        \ENDIF
        \STATE{~}
        \STATE{$F_{0, p} \gets F_{p}$}
        \IF{$\Tilde{N} > N$ \COMMENT{the expansion order is increased}}
            \STATE{$\eta \gets \sigma \eta$}
        \ENDIF
        \RETURN{solution with updated $N$: $U_{\Tilde{N}, \gamma}^{\boldsymbol{\beta}, \boldsymbol{x}_{0}} \gets \pi_{\Tilde{N}, \gamma}^{\boldsymbol{\beta}, \boldsymbol{x}_{0}}U_{N, \gamma}^{\boldsymbol{\beta}, \boldsymbol{x}_{0}}$}
    \end{algorithmic}
\end{algorithm}

\begin{table}[h!]\label{table:2}
    \centering
    \caption{\footnotesize{\textnormal{Hyperparameters when implementing the adaptive techniques. $q$ and $\delta$ are the smallest scaling factor change ratio and the smallest displacement change, respectively. $d_{\max}$ stands for the maximum displacement allowed per time step. $\mu$ stands for the moving technique activation threshold and $\nu$ is the scaling technique activation threshold. Additionally, $N_{\max}$ is the maximum increase of expansion order $N$ per time step. $\eta$ is the threshold for changing the expansion order $N$ and $\eta_{0}$ is the initial value of $\eta$. $\sigma$ is the adjustment ratio for $\eta$ if the expansion order $N$ is increased.}}}
\vspace{0.1in}
    \begin{tabular}{cccccccccc}
    \toprule
    \toprule
    \small Example \& Method & $\delta$ & $d_{\text{max}}$ & $\mu$ & $q$ & $\nu$ & $N_{\text{max}}$ & $\eta_{0}$ & $\sigma$\\
    \midrule
    \small4.1 mapped Jacobi &0.01 & $\slash$ & $\slash$ & 0.99 & 1.01 & 3 & 1.2 & 1.2\\
    \small4.1 Hermite &0.01 & $\slash$ & $\slash$ & 0.99 & 1.01 & 3 & 1.2 & 1.2\\
    \small4.2 mapped Jacobi &0.01 & 0.2 & 1.0005 & 0.99 & 1.01 & $\slash$ & $\slash$ & $\slash$\\
    \small4.2 Hermite &0.01 & 0.2 & 1.0005 & 0.99 & 1.01 & $\slash$ & $\slash$ & $\slash$\\
    \small4.3 Hyperbolic &0.01 & (0.06, 0.09) & 1.0005 & 0.99 & 1.01 & $\slash$ & $\slash$ & $\slash$\\
    \small4.3 Truncation &0.01 & (0.06, 0.09) & 1.0005 & 0.99 & 1.01 & $\slash$ & $\slash$ & $\slash$\\
    \small4.4 Hyperbolic &0.01 & $\slash$ & $\slash$ & 0.99 & 1.01 & 5 & 1.2 & 1.1\\
    \small4.4 Truncation &0.01 & $\slash$ & $\slash$ & 0.99 & 1.01 & 5 & 1.2 & 1.1\\
    \small4.5 Hyperbolic &0.01 & $\slash$ & $\slash$ & 0.99 & 1.01 & 5 & 1.2 & 1.2\\
    \small4.5 Truncation &0.01 & $\slash$ & $\slash$ & 0.99 & 1.01 & 5 & 1.2 & 1.2\\
    \bottomrule
    \bottomrule
    \end{tabular}
\end{table}
\newpage
\bibliographystyle{elsarticle-num}
\bibliography{main}
\end{document}